\documentclass[11pt, reqno]{amsart}
\usepackage[english]{babel}
\usepackage[colorlinks,citecolor=green,linkcolor=red]{hyperref}
\usepackage{amsthm}
\usepackage{color}
\usepackage{mathrsfs}
\usepackage{amsmath}
\usepackage{mathtools}
\usepackage{amsfonts}
\usepackage{amssymb}
\usepackage{bm}
\usepackage{physics}
\usepackage{enumitem}
\usepackage{a4wide}
\usepackage{cleveref}
\usepackage{yfonts}
\usepackage[normalem]{ulem}

\numberwithin{equation}{section}

\newtheorem{thm}{Theorem}[section]
\newtheorem*{thm*}{Theorem}
\newtheorem{lem}[thm]{Lemma}
\newtheorem{prop}[thm]{Proposition}

\theoremstyle{definition}
\newtheorem{defn}[thm]{Definition}

\theoremstyle{remark}
\newtheorem{rem}[thm]{Remark}

\DeclareMathOperator*{\aplimsup}{ap\,\varlimsup}
\DeclareMathOperator*{\apliminf}{ap\,\varliminf}

\newcommand{\TestF}{{\mathrm {TestF}}}
\newcommand{\TestV}{{\mathrm {TestV}}}

\newcommand{\lip}{{\mathrm {lip}}}

\newcommand{\diff}{{\mathrm{d}}}
\newcommand{\DIFF}{{\mathrm{D}}}

\newcommand{\de}{\ensuremath{\,\mathrm d}}

\newcommand{\capa}{{\mathrm {Cap}}}

\newcommand{\mres}{\mathbin{\vrule height 1.6ex depth 0pt width 0.13ex\vrule height 0.13ex depth 0pt width 1.3ex}}

\let \llcorner \mres

\newcommand{\dist}{{\mathsf{d}}}
\newcommand{\mass}{{\mathsf{m}}}

\newcommand{\XX}{{\mathsf{X}}}
\newcommand{\YY}{{\mathsf{Y}}}
\newcommand{\ZZ}{{\mathsf{Z}}}

\newcommand{\FF}{{\mathcal{F}}}
\newcommand{\GG}{{\mathcal{G}}}
\newcommand{\RR}{\mathbb{R}}
\newcommand{\Reg}{\mathcal{R}}
\newcommand{\NN}{\mathbb{N}}
\newcommand{\QQ}{\mathbb{Q}}

\newcommand{\HH}{\mathcal{H}}

\newcommand{\Tan}{\mathrm {Tan}}

\newcommand{\defeq}{\mathrel{\mathop:}=}

\newcommand{\LIP}{\mathrm {LIP}}

\newcommand{\BV}{\mathrm {BV}}
\newcommand{\BVv}{{\mathrm {BV}}(\XX)}

\newcommand{\tanXcap}{L^0_\capa(T\XX)}

\newcommand{\RCD}{{\mathrm {RCD}}}

\newcommand{\PI}{{\mathrm {PI}}}

\renewcommand{\epsilon}{\varepsilon}
\renewcommand{\limsup}{\varlimsup}
\renewcommand{\liminf}{\varliminf}

\def\Xint#1{\mathchoice
	{\XXint\displaystyle\textstyle{#1}}%
	{\XXint\textstyle\scriptstyle{#1}}%
	{\XXint\scriptstyle\scriptscriptstyle{#1}}%
	{\XXint\scriptscriptstyle\scriptscriptstyle{#1}}%
	\!\int}
\def\XXint#1#2#3{{\setbox0=\hbox{$#1{#2#3}{\int}$}
		\vcenter{\hbox{$#2#3$}}\kern-.5\wd0}}

\def\dashint{\Xint-}
\let\fint\dashint

\title{The Rank-One Theorem on \(\rm RCD\) spaces}

\author[Gioacchino Antonelli]{Gioacchino Antonelli}
\address[Gioacchino Antonelli]{Scuola Normale Superiore, Piazza dei Cavalieri, 7, 56126 Pisa, Italy.}
\email{ga2434@nyu.edu}

\author[Camillo Brena]{Camillo Brena}
\address[Camillo Brena]{Scuola Normale Superiore, Piazza dei Cavalieri, 7, 56126 Pisa, Italy.}
\email{camillo.brena@sns.it}

\author[Enrico Pasqualetto]{Enrico Pasqualetto}
\address[Enrico Pasqualetto]{Scuola Normale Superiore, Piazza dei Cavalieri, 7, 56126 Pisa, Italy.}
\email{enrico.e.pasqualetto@jyu.fi}

\begin{document}

\date{\today}

\begin{abstract}
We extend Alberti's Rank-One Theorem to $\mathrm{RCD}(K,N)$ metric measure spaces.
\end{abstract}

\maketitle
\tableofcontents

\section{Introduction}
\subsection{The Rank-One Theorem in the Euclidean setting}
Let $\Omega$ be an open subset of $\mathbb R^n$, and let $u\in\mathrm{BV}(\Omega;\mathbb R^k)$, {i.e.\ $u=(u_1,\dots,u_k)\in (\BV(\Omega))^k$}. By using the Lebesgue--Radon--Nikod\'{y}m theorem one can write the distributional derivative of $u$ as 
\[
\DIFF u=\DIFF^au+\DIFF^su,
\]
where $\DIFF^au$ is the absolutely continuous part of $\DIFF u$ with respect to the Lebesgue measure $\mathcal{L}^n$, and $\DIFF^s u$ is the singular part of $\DIFF u$. We denote with $\DIFF u/|\DIFF u|$ the matrix-valued Lebesgue--Radon--Nikod\'{y}m density of $\DIFF u$ with respect to the total variation $|\DIFF u|$. Notice that the total variation of the singular part $|\DIFF^s u|$ is equal to the singular part of the total variation $|\DIFF u|^s$.
\smallskip

In 1988 Ambrosio and De Giorgi \cite{DeGiorgiAmbrosio}, motivated by the study of some functionals coming from the Mathematical Physics, conjectured the following:
\begin{center}
\textbf{Rank-One property}: For every $u\in\mathrm{BV}(\Omega;\mathbb R^k)$ the matrix $\DIFF u/|\DIFF u|$ has rank-one $|\DIFF u|^s$-almost everywhere.
\end{center}

In 1993 Alberti \cite{alberti1993} solved in the affirmative the previous conjecture, see also the account in \cite{DeLellis2008}. 
\smallskip 

It is worth observing that the ideas used in \cite{alberti1993} showed up to be very robust for further developments of Geometric Measure Theory and the rectifiability theory in Euclidean spaces and even beyond in the metric setting. In \cite{alberti1993}, as a main step of the proof, Alberti proved that given an arbitrary Radon measure $\mu$ on a $k$-dimensional plane $V$ in $\mathbb R^n$ that is singular with respect to $\mathcal{H}^k\llcorner V$, one can associate to $\mu$ a bundle $E(\mu,\cdot)$ whose fibers have dimension at most 1. The fiber $E(\mu,x)$ of this bundle is made by all the vectors $v\in\mathbb R^k$ such that $v\mu$ is \emph{tangent} in $x$, in a precise sense, to the derivative of a $\mathrm{BV}$ function on $V$. What happens, moreover, is that the restriction of $\mu$ to the set where $E(\mu,\cdot)$ is 1 can be written as $\int_I \mu_t\,\dd t$, where $\mu_t=\mathcal{H}^{k-1}\llcorner S_t$, and $S_t$ is $(k-1)$-rectifiable in $V$.

In the language of \cite{ACP10}, which collects several other fine results for the theory of rectifiability in $\mathbb R^n$, the previous result means that, on the set where the fiber is one-dimensional, $\mu$ is \emph{$(k-1)$-representable}: namely, it can be written as the integral of measures that are $(k-1)$-rectifiable. At the basis of this possibility of representing a measure as integral of rectifiable measures is the idea of the Alberti representations.

Another interesting contribution that originated from this circle of ideas is the result by Alberti and Marchese in \cite{AlbertiMarchese}. In that paper the authors associate to every Radon measure $\mu$ on $\mathbb R^n$ the minimal (unique $\mu$-almost everywhere) bundle $V(\mu,\cdot)$ such that every real-valued Lipschitz function on $\mathbb R^n$ is differentiable along $V(\mu,x)$ for $\mu$-almost every $x\in\mathbb R^n$. Alberti representations were also recently used by Bate and Bate--Li in the study of rectifiability in the general metric setting, see \cite{BateJAMS, BateLi}. For further readings, one can consult the recent survey by Mattila, in particular \cite[Chapter 8]{MattilaSurvey}, and \cite[Chapter 13]{MattilaSurvey}.
\smallskip

Besides its theoretical interest, the Rank-One Theorem soon gave important consequences in the Calculus of Variations. In \cite{AMBROSIO199276} Ambrosio and Dal Maso exploited it to derive the expression of the relaxation (in $\mathrm{BV}$) of a functional defined on $C^1$ functions as the integral of a quasi-convex function of linear growth of the gradient. See also \cite{Kristensen2009RelaxationOS} for a generalization. Moreover, Fonseca and M\"uller generalized the result in \cite{AMBROSIO199276} for integrands that might not depend solely on the gradient, but also on the space variable and the function itself \cite{Fonseca1993RelaxationOQ}. For further details we refer the reader to \cite[Chapter 5]{AFP00}.
\smallskip

As an added value to the theoretical interest of the Rank-One Theorem, in 2016 De Philippis and Rindler \cite{DPR} showed a general structure theorem for $\mathcal{A}$-free vector-valued Radon measures on Euclidean spaces, where $\mathcal{A}$ is a linear constant-coefficient differential operator, from which the Rank-One Theorem can be derived as a consequence. We also remark that Massaccesi and Vittone recently gave a very short proof of the Rank-One Theorem based on the theory of sets of finite perimeter \cite{MassaccesiVittone19}, and with Don they used this simplified strategy to prove the analogue of the Rank-One Theorem in some Carnot groups \cite{DonMasVit17}.

\subsection{Main result}
Nowadays a well-established notion of $\mathrm{BV}$ function is available in the metric measure setting. Such a notion was proposed by Miranda \cite{MIRANDA2003}, then studied by Ambrosio \cite{AmbrosioAhlfors, amb01}, and recently by Ambrosio--Di Marino \cite{ADM2014}.
\smallskip 

According to this approach, given a metric measure space $(\XX,\dist,\mass)$ the total variation of the derivative of $f\in L^1_{\mathrm{loc}}(\XX,\mass)$ is the relaxation in $L^1_{\mathrm{loc}}(\XX,\mass)$ of the energy given by the integral of the local Lipschitz constant. Such a definition can be readily extended to define the total variation of a vector-valued function whose components are in $\mathrm{BV}_{\mathrm{loc}}(\XX,\dist,\mass)$, see
\Cref{def:TotalVariationBV} for the precise definition. 

Nevertheless, in this way one is giving a meaning to the total variation $|\DIFF F|$ of an arbitrary $F\in \mathrm{BV}_{\mathrm{loc}}(\XX,\dist,\mass)^k$, while it is in general missing a good notion for the Lebesgue--Radon--Nikod\'{y}m derivative $\DIFF F/|\DIFF F|$. 
\smallskip

In the setting of $\RCD$ metric measure spaces, see \Cref{sec:RCD} for details and references, the study of calculus has been blossoming very fast in the last decade. In particular, very recently in \cite{debin2019quasicontinuous} the authors propose and study the notion of $L^0(\mathrm{Cap})$-normed $L^0(\mathrm{Cap})$-module, and the notion of capacitary tangent module $L^0_{\mathrm{Cap}}(T\XX)$, where $\mathrm{Cap}$ denotes the usual Capacity \eqref{eqn:Capacity}. We refer to \Cref{sec:RCD} for the definitions and further details.
\medskip

A fundamental contribution of \cite{bru2019rectifiability}, building on \cite{debin2019quasicontinuous}, is the fact that, in the setting of $\RCD(K,N)$ spaces, for an arbitrary set of finite perimeter $E$ with finite mass, one can give a meaning to the unit normal $\nu_E=\DIFF\chi_E/|\DIFF\chi_E|$ as an element of the capacitary tangent module $L^0_{\mathrm{Cap}}(T\XX)$ such that the Gauss--Green formula holds, see \cite[Theorem 2.4]{bru2019rectifiability}. The Gauss--Green formula has been successfully employed, together with the former work by Ambrosio--Bruè--Semola \cite{ambrosio2018rigidity}, to obtain the $(n-1)$-rectifiability of the essential boundary of any set of locally finite perimeter in an $\RCD$ space of essential dimension $n$, see \cite{bru2019rectifiability, bru2021constancy}.

The Gauss--Green formula in \cite[Theorem 2.4]{bru2019rectifiability} has been generalized by the second named author together with Gigli in \cite{BGBV} for vector--valued $\mathrm{BV}$ functions. We give below the statement of the Gauss--Green formula in \cite{BGBV}, where the density $\nu_F=\DIFF F/|\DIFF F|$ is implicitly defined.

\begin{thm}[{\cite[Theorem 3.13]{BGBV}}]\label{thm:Intro}
Let $k\geq 1$ be a natural number, let $K\in\mathbb R$, and let $N\geq 1$. Let $(\XX,\dist,\mass)$ be an $\RCD(K,N)$ space and let $F\in \BV(\XX,\dist,\mass)^k$. Then there exists a unique $\nu_F\in L^0_\capa(T\XX)^k$, up to $|\DIFF F|$-almost everywhere equality, such that $|\nu_F|=1\ |\DIFF F|$-almost everywhere, and 
    $$
    \sum_{j=1}^k\int_\XX F_j{\rm div}(v_j)\dd{\mass}=-\int_\XX \pi_{|\DIFF F|}(v)\,\cdot\,\nu_F\dd{|\DIFF F|},\quad\text{ for every $v=(v_1,\dots,v_k)\in\TestV(\XX)^k$.}
    $$
\end{thm}
For the notion of divergence of a vector field, the notion of test vector fields $\TestV(\XX)$, the notion of the projection $\pi_{|\DIFF F|}$ and of the norm $|\cdot|$ in $L^0_{\mathrm{Cap}}(T\XX)^k$, we refer the reader to \Cref{sec:RCD}.
\smallskip

The previous \Cref{thm:Intro} tells us that in the setting of $\RCD(K,N)$ spaces we can give a precise meaning to $\DIFF F/|\DIFF F|$ for an arbitrary vector-valued $\mathrm{BV}$ function $F$. Hence it is meaningful to ask if $\DIFF F/|\DIFF F|$ is a rank one matrix $|\DIFF F|^s$-almost everywhere, where $|\DIFF F|^s$ is the singular part of the total variation $|\DIFF F|$. Before giving the main result of this paper we clarify this last sentence by means of a definition. For the definition of the space $L^0(\mathrm{Cap})$, see \Cref{sec:RCD}.

\begin{defn}\label{def:InTheSenseOf}
Let $k\geq 1$ be a natural number, let $K\in\mathbb R$, and let $N\geq 1$. Let $(\XX,\dist,\mass)$ be an $\RCD(K,N)$ space, let $\nu\in L^0_\capa(T\XX)^k$, and let $\mu\ll\capa$ be a Radon measure, where $\mathrm{Cap}$ is the usual Capacity \eqref{eqn:Capacity}. We say that 
$$
{\rm Rk}(\nu)=1\quad\text{$\mu$-almost everywhere},
$$
if there exist $\omega\in L^0_\capa(T\XX)$ and $\lambda_1,\dots,\lambda_k\in L^0(\capa)$ such that for every $i=1,\dots,k$,
$$
\nu_i=\lambda_i\omega\qquad\text{$\mu$-almost everywhere}.
$$
\end{defn}
We remark that this is one of the possible definitions we could have given of {having rank one}. For example, one can give an alternative and equivalent definition exploiting the existence of a local basis (with respect to a decomposition of the space in Borel sets) of $L^0_\capa(T\XX)$, to recover the language of rank of a matrix. It is however clear that in Euclidean spaces, the definition given above coincides with the usual one.

We are now ready to state the main theorem of this paper, which is the generalization of the Rank-One Theorem in the setting of $\RCD(K,N)$ metric measure spaces $(\XX,\dist,\mass)$.
\begin{thm}[Rank-One Theorem for $\mathrm{RCD}(K,N)$ spaces]\label{thm:RankOne}
Let $k\geq 1$ be a natural number, let $K\in\mathbb R$, and let $N\geq 1$. Let $(\XX,\dist,\mass)$ be an $\RCD(K,N)$ space, and let $F\in\BV(\XX,\dist,\mass)^k$. Then 
$$
{\rm Rk}(\nu_F)=1\qquad|\DIFF F|^s\text{-almost everywhere},
$$
in the sense of \Cref{def:InTheSenseOf}, where $\nu_F$ is defined in \Cref{thm:Intro}, and $|\DIFF F|^s$ is the singular part of the total variation $|\DIFF F|$.
\end{thm}

As far as we know, apart from the result of Don--Massaccesi--Vittone \cite{DonMasVit17} that holds for a special class of Carnot groups, \Cref{thm:RankOne} is one of the first instances of the validity of the Rank-One Theorem in a large class of metric measure spaces. 
\smallskip 

We stress that, even if the proof of \cite{DonMasVit17} covers a large class of Carnot groups, some distinguished examples are still not covered. For example, as of today it is not known if the Rank-One Theorem holds for vector-valued $\mathrm{BV}$ functions in the first Heisenberg group $\mathbb H^1$.  We stress that our strategy for the proof of \Cref{thm:RankOne} seems not to be applicable to prove the Rank-One Theorem in $\mathbb H^1$. Indeed, we are fundamentally exploiting the fact that we have good bi-Lipschitz charts on the space valued in the tangents. But, even if on $\mathbb H^1$ the boundary of a set of locally finite perimeter is intrinsic $C^1$-rectifiable, see \cite{FSSC01}, it is nowadays not known whether intrinsic $C^1$ surfaces can be almost everywhere covered by (bi)-Lipschitz images of their tangents, see \cite{DDFO20} for partial results in this direction.

{ We stress that our strategy cannot be easily adapted to prove Rank-One-type results for $\mathrm{BV}$ functions in $\RCD(K,\infty)$ spaces. In fact, our proof works mainly by blow-up. Since $\RCD(K,\infty)$ spaces might be not locally doubling, we do not have a good notion of Gromov--Hausdorff tangent at their points. In particular, it would even be challenging to understand whether the results in \cite{AmbrosioAhlfors,amb01,ambrosio2018rigidity, bru2019rectifiability, bru2021constancy}, which are the starting point of our analysis, can be adapted to the $\RCD(K,\infty)$ setting.}

Moreover, we point out that very recently Lahti proposed an alternative formulation of Alberti's Rank-One Theorem that could make sense in arbitrary metric measure spaces, we refer to \cite[Section 6]{LahtiAlberti}.

\subsection{Outline of the proof}
Our proof is inspired by the one by Massaccesi--Vittone \cite{MassaccesiVittone19}. First, given $F\in \mathrm{BV}(\XX,\dist,\mass)^k$, the singular part of the total variation $|\DIFF F|^s$ can be written as the sum of the jump part $|\DIFF F|^j$, which is concentrated on the set where the approximate lower and upper limits of the components of $F$ do not coincide, and the Cantor part $|\DIFF F|^c$, see \Cref{decomptv}. As a consequence of a result by the second named author and Gigli, see the forthcoming  \Cref{rank1lem1}, it is enough to show the Rank-One Theorem only on the Cantor part. 

We stress that in the proofs of the main results in \Cref{sec:Main} we shall always restrict to sets where the Cantor part of the components of $F$ is concentrated, and where we have good density and blow-up properties: we collect all of them in the technical \Cref{prop:def_set_C_f}.
\smallskip

The core and the most technically demanding part of the proof is \Cref{mainlemma}, in which we adapt to our setting the main Lemma of the short proof of the Rank-One Theorem by Massaccesi--Vittone \cite{MassaccesiVittone19}. In fact, Massaccesi and Vittone prove that given two $C^1$-hypersurfaces $\Sigma_1,\Sigma_2$ in $\mathbb R^{n}\times\mathbb R$, the set $T$ of points $p\in\Sigma_1$ such that there exists $q\in \Sigma_2$ for which $p$ and $q$ have the same first $n$ coordinates, $\nu_{\Sigma_1}(p)_{n+1}=\nu_{\Sigma_2}(q)_{n+1}=0$, and $\nu_{\Sigma_1}(p)\neq \pm \nu_{\Sigma_2}(q)$, is $\mathcal{H}^n$-negligible. Clearly the latter statement makes no sense in our non smooth setting, but what one really needs for the proof of the Rank-One Theorem is \Cref{mainlemma}. 
\smallskip

Following the strategy in the Proof of the Lemma of \cite{MassaccesiVittone19}, one writes $T$ as the projection of a set $\widetilde T \subset \mathbb R^{n}\times\mathbb R\times\mathbb R$ adding one fake coordinate, and proves that $T=\pi(\widetilde T)$ is $\mathcal{H}^n$-negligible by means of the area formula. In \Cref{mainlemma} we adapt the same strategy, compare with the definition of the set \eqref{eqn:SigmaFSigmag}. We prove the analogue of Massaccesi--Vittone Lemma substituting the hypersurfaces $\Sigma_i$'s with the (essential) boundaries of sets of the form $\mathcal{G}_f:=\{(x,t):t<f(x)\}$, where $f\in\mathrm{BV}(\XX,\dist,\mass)$. This is enough to implement in our setting their strategy.
However, to adapt the proof in \cite{MassaccesiVittone19, DonMasVit17} to our framework one faces non trivial technical difficulties. Indeed, the key ingredient used by Massaccesi--Vittone was a well-known transversality lemma: given two hypersurfaces in $\RR^{n+2}$, their intersection is locally an $n$-dimensional manifold provided that at every intersection point the given hypersurfaces meet transversally, i.e., have different tangent spaces. This result then extends to the case of the intersection of two $(n+1)$-rectifiable subsets of $\RR^{n+2}$: their intersection is $\sigma$-finite with respect to $\HH^{n}$ provided that the transversality condition is satisfied and that one discards a set that turns out to be negligible when proving the Rank-One property.

It is clear that one needs also information on codimension $2$ objects (namely, the intersection of two transverse hypersurfaces) and this kind of information is unavailable on $\RCD$ spaces. Therefore, adopting directly this approach is not possible in our framework.
Our strategy is then to translate part of the problem from the $\RCD$ setting to the Euclidean setting (which allows us to use transversality results as above), via the use of suitable $\delta$-splitting maps that play the role of charts, leveraging heavily on the results of \cite{bru2019rectifiability,bru2021constancy}. The fact that the domains of these charts are not open sets  is a source of difficulty and is morally the burden of the proof of Lemma \ref{mainlemma}. In other words, we could not work directly arguing with infinitesimal considerations in the $\RCD$ case (i.e., using directly difference of blow-ups) but we had to argue locally and then infinitesimally in a Euclidean space.
\smallskip

As an important part of the proof, to manipulate the vector that is normal to the boundary of the set $\mathcal{G}_f$, we need to introduce a family of charts in which we write those normals in coordinates, see \Cref{real_normal}. We construct these charts in \Cref{defn:GoodCollection}, and we call them a \emph{good collection of splitting maps}. The latter definition is based on the following fact, which is proved in \Cref{lemsplitting}. Given an $\mathrm{RCD}$ space of essential dimension $n$, we prove that for every $\eta>0$ small enough we can find a sequence of $n$-tuples of harmonic maps $\{u_{k,\eta}\}_{k\in\mathbb N}$ defined on balls, and a disjointed family of Borel sets $\{D_{k,\eta}\}_{k\in\mathbb N}$ such that for every $x\in D_{k,\eta}$, $u_{k,\eta}$ is an $\eta$-splitting map on $B_{r_k}(x)$, and the total variation of every $\mathrm{BV}_{\mathrm{loc}}$ function is concentrated on $\bigsqcup_{k\in\mathbb N}D_{k,\eta}$.
\smallskip

The other two ingredients to adapt in our setting the strategy of \cite{MassaccesiVittone19} are \Cref{coincide} and \Cref{thmcantor}. In the first we prove that, given $f\in\mathrm{BV}$, restricting to the good set on the Cantor part as in \Cref{prop:def_set_C_f}, we have that (in coordinates) the density $\nu_f(x)$ is equal to the first coordinates of the normal $\nu_{\mathcal{G}_f}(x,f(x))$, where $\mathcal{G}_f:=\{(x,t):t<f(x)\}$. In the second we prove that, restricting to the good set on the Cantor part as in \Cref{prop:def_set_C_f}, the $(n+1)$-th coordinate of the normal $\nu_{\mathcal{G}_f}$ is almost everywhere zero. This is essentially due to the fact that we are on the singular part of $\DIFF f$.

Again, not having at our disposal a linear structure is source of difficulty, as the distributional derivative has no more a direction-wise meaning, in the sense that it is not possible to define the distributional derivative of a function of bounded variation with respect to a given direction without giving up the differential meaning of this object. To overcome this difficulty, we employ blow-up arguments and density arguments.

\smallskip

Finally, putting together \Cref{mainlemma}, \Cref{coincide}, and \Cref{thmcantor}, we conclude that given two $\mathrm{BV}$ functions $f,g$, we have that $\nu_f=\pm\nu_g$ holds $|\DIFF f|\wedge |\DIFF g|$-almost everywhere on the intersection of the good sets $C_f\cap C_g$ defined in \Cref{prop:def_set_C_f}, see \Cref{rank1lem0}. {Here $\wedge$ stands for the minimum between the two measures, i.e.\ the biggest measure $\phi$ such that $\phi\le \mu$ and $\phi\le \nu$.}
This, together with the same property on the jump part, see \Cref{rank1lem1}, concludes the proof. 
\smallskip

We stress that along the way in \Cref{sec:Representation}, building on \cite{deng2020holder} (compare with \cite{ColdingNaber12} for the H\"older continuity property of tangents along geodesics in the Ricci-limit case), we improve a previous result of \cite{bru2021constancy} by showing that every $\mathrm{BV}$ function on an $\mathrm{RCD}$ space of essential dimension $n$ has total variation concentrated on the set $\mathcal{R}_n^*$ of $n$-regular points with positive and finite $n$-density, see \Cref{thm:const_dim_cod1}. We exploit the latter result to answer in the affirmative a conjecture proposed in \cite{Sem20} about the representation of the perimeter measure, see \Cref{thm:repr_per}. 
\smallskip 

In \Cref{sec:Appendix} we exploit the previously described result proved in \Cref{thm:const_dim_cod1}, together with the recently proved metric variant of Marstrand--Mattila rectifiability criterion \cite{BateMM}, to give an alternative and shorter proof of the $(n-1)$-rectifiability of the essential boundaries of sets of locally finite perimeter in $\mathrm{RCD}$ spaces with essential dimension $n$.
We believe that this result is of independent interest but we point out that it originated as a side remark due to the fact that we were interested to prove the Rank-One property in general $\RCD(K,N)$ spaces, without restricting ourselves to \emph{non-collapsed} $\RCD$ spaces. Indeed, the information that the perimeter measure and the $\HH^{n-1}$ measure restricted to the reduced boundary are mutually absolutely continuous (already known in the \emph{non-collapsed} case) is crucial in the proof of Lemma \ref{mainlemma}. Anyway, we point out that even if the proof presented in \Cref{sec:Appendix} is much shorter than the original one, it is heavily based on the ideas and techniques exploited in \cite{bru2019rectifiability, bru2021constancy}, i.e., looking at what happens at the space locally and infinitesimally by using well-behaved charts.

\subsection*{Structure of the paper}
In \Cref{sec:Preliminaries} we discuss the basic tools and notation that we shall use throughout the paper. 

In particular, in \Cref{sec:mms} we discuss the basic toolkit for metric measure spaces. We recall the definition of PI space, the notion of pointed measured Gromov--Hausdorff convergence and tangents, and the basic Sobolev and BV calculus in arbitrary metric measure spaces. 

In \Cref{sec:RCD} we recall basic structure results of $\mathrm{RCD}$ spaces, and the main important notions of Sobolev and $\mathrm{BV}$ calculus on $\mathrm{RCD}$ spaces. We further recall the notion of \emph{good coordinates} introduced in \cite{bru2021constancy}, the notion of splitting maps, and finally we prove \Cref{lemsplitting} that leads to the notion of \emph{good collection of splitting maps}, see \Cref{defn:GoodCollection}.
\smallskip

In \Cref{sec:Main} we prove the main results of this paper, and in particular we give the proof of the Rank-One Theorem in \Cref{thm:RankOne}. 

In particular in \Cref{sec:Representation}, building on \cite{deng2020holder}, we prove \Cref{thm:const_dim_cod1} described above.

In \Cref{sec:Auxiliary} we prove some auxiliary results toward the proof of the Rank-One Theorem, namely \Cref{prop:def_set_C_f}, \Cref{coincide}, and \Cref{thmcantor}.

Finally, in \Cref{sec:Final} we exploit the previous results, together with the main result in \Cref{mainlemma}, which is the adaptation to our setting of the Lemma of \cite{MassaccesiVittone19}, to show the Rank-One property on the Cantor part, see \Cref{rank1lem0}. This is enough to conclude the proof of the Rank-One Theorem by exploiting also \Cref{rank1lem1}, which is the Rank-One property on the jump part.

In \Cref{sec:Appendix} we give the alternative proof of the rectifiability of the essential boundaries of sets of locally finite perimeter in $\mathrm{RCD}$ spaces that we described above.

\subsection*{Acknowledgments}
The first author was partially supported by the Swiss National Science Foundation
(grant 200021-204501 `\emph{Regularity of sub-Riemannian geodesics and
applications}'), by the European Research Council (ERC Starting Grant 713998 GeoMeG
`\emph{Geometry of Metric Groups}'), and the PRIN Project 2017TEXA3H.
The third author is supported by the Balzan project led by Luigi Ambrosio.

The authors wish to thank Sebastiano Don and Daniele Semola for precious comments on a draft of this paper. {We warmly thank the anonymous referees for many valuable comments that improved the readability of the paper.}
\section{Preliminaries}\label{sec:Preliminaries}
{
We often need to bound quantities in terms of constants that depend only on geometric parameters but whose precise value is not important. For this reason, we denote with $C_{a,b,\dots}$ a constant depending only on the parameters $a,b,\dots$, whose value might change from line to line or even within the same line. 
}\medskip

Given \(n\in\NN\) and non-empty sets \(\XX_1,\dots,\XX_n\), for any \(i=1,\ldots,n\) we will always tacitly denote by \(\pi^i\)
the projection of the Cartesian product \(\XX_1\times\dots\times\XX_n\) onto its \(i^{\rm th}\) factor:
\[
\pi^i\colon\XX_1\times\dots\times\XX_n\to\XX_i,\quad(x_1,\dots,x_n)\mapsto x_i.
\]
Moreover, we denote by $\pi^{i,j}$ the projection of the Cartesian product \(\XX_1\times\dots\times\XX_n\) onto its $(i,j)$ factor, namely 
\[
\pi^{i,j}\colon\XX_1\times\dots\times\XX_n\to\XX_i\times\XX_j,\quad(x_1,\dots,x_n)\mapsto (x_i,x_j).
\]
Finally, we denote by $\tau$ the inversion map on the last two factors on a product of three factors, namely 
\begin{equation}\label{eqn:Tau}
\tau:\XX_1\times\XX_2\times\XX_3\rightarrow\XX_1\times\XX_3\times\XX_2,\quad (x_1,x_2,x_3)\mapsto (x_1,x_3,x_2).
\end{equation}
\subsection{Metric measure spaces}\label{sec:mms}
For the purposes of this paper, a \emph{metric measure space} is a triple \((\XX,\dist,\mass)\), where \((\XX,\dist)\) is a complete and separable metric space,
while \(\mass\ge 0\) is a boundedly-finite Borel measure on \(\XX\). By a \emph{pointed metric measure space} \((\XX,\dist,\mass,p)\) we mean a metric measure
space \((\XX,\dist,\mass)\) together with a distinguished point \(p\in{\rm spt}(\mass)\), where
\[
{\rm spt}(\mass)\coloneqq\big\{x\in\XX\;\big|\;\mass(B_r(x))>0,\text{ for every }r>0\big\}
\]
stands for the \emph{support} of \(\mass\). Given an open set \(\Omega\subseteq\XX\), we denote by \({\rm LIP}_{\mathrm{loc}}(\Omega)\) and
\({\rm LIP}(\Omega)\) the spaces of all locally Lipschitz and Lipschitz functions on \(\Omega\), respectively, while we set
\[
{\rm LIP}_{\mathrm{bs}}(\Omega)\coloneqq\big\{f\in{\rm LIP}(\Omega)\;\big|\;{\rm spt}(f)\text{ is bounded and }\dist(\partial\Omega,{\rm spt}(f))>0\big\}.
\]
Given any \(f\in{\rm LIP}_{\mathrm{loc}}(\Omega)\), its \emph{local Lipschitz constant} \(\lip f\coloneqq\Omega\to[0,+\infty)\) is defined as
\[
\lip f(x)\coloneqq\left\{\begin{array}{ll}
\varlimsup_{y\to x}\frac{|f(x)-f(y)|}{\dist(x,y)},\\
0,
\end{array}\quad\begin{array}{ll}
\text{ if }x\in\Omega\text{ is an accumulation point},\\
\text{ if }x\in\Omega\text{ is an isolated point}.
\end{array}\right.
\]
For any \(k\in[0,+\infty)\) and \(\delta>0\), we will denote by \(\mathcal H^k_\delta\) and \(\mathcal H^k\) the \emph{\(k\)-dimensional
Hausdorff \(\delta\)-premeasure} and the \emph{\(k\)-dimensional Hausdorff measure} on \((\XX,\dist)\), respectively. Namely,
\[\begin{split}
\mathcal H^k_\delta(E)&\coloneqq\inf\bigg\{\sum_{i=1}^\infty\omega_k\bigg(\frac{{\rm diam}(E_i)}{2}\bigg)^k\;\bigg|\;
E\subseteq\bigcup_{i\in\NN}E_i\subseteq\XX,\;\sup_{i\in\NN}{\rm diam}(E_i)<\delta\bigg\},\\
\mathcal H^k(E)&\coloneqq\lim_{\delta\searrow 0}\mathcal H^k_\delta(E)=\sup_{\delta>0}\mathcal H^k_\delta(E)
\end{split}\]
for every set \(E\subseteq\XX\), where \(\omega_k\coloneqq\frac{\pi^{k/2}}{\Gamma(1+k/2)}\) and \(\Gamma\) stands for the Euler's gamma function. Notice that, for every $n\in\mathbb N$, $\omega_n$ is the Euclidean volume of the unit ball in $\mathbb R^n$.
\subsubsection{PI spaces}
Throughout the whole paper, we will work in the setting of PI spaces.
We say that a metric measure space \((\XX,\dist,\mass)\) is \emph{uniformly locally doubling} provided for every radius \(R>0\)
there exists a constant \(C_D>0\) such that
\[
\mass(B_{2r}(x))\leq C_D\mass(B_r(x)),\quad\text{ for every }x\in\XX\text{ and }r\in(0,R).
\]
Moreover, we say that \((\XX,\dist,\mass)\) supports a \emph{weak local \((1,1)\)-Poincar\'{e} inequality} provided there exists
a constant \(\lambda\ge 1\) for which the following property holds: given any \(R>0\), there exists a constant \(C_P>0\) such that
for any function \(f\in{\rm LIP}_{\mathrm{loc}}(\XX)\) it holds that
\[
\fint_{B_r(x)}\bigg|f-\fint_{B_r(x)}f\,\diff\mass\bigg|\,\diff\mass\leq C_P r\fint_{B_{\lambda r}(x)}\lip f\,\diff\mass,
\quad\text{ for every }x\in\XX\text{ and }r\in(0,R).
\]
\begin{defn}[PI space]
We say that a metric measure space is a \emph{PI space} provided it is uniformly locally doubling
and it supports a weak local \((1,1)\)-Poincar\'{e} inequality.
\end{defn}
In the context of PI spaces, we will consider the \emph{codimension-\(1\) Hausdorff \(\delta\)-premeasure} \(\mathcal H^h_\delta\)
(for any \(\delta>0\)) and the \emph{codimension-\(1\) Hausdorff measure} \(\mathcal H^h\), which are given by
\[\begin{split}
\mathcal H^h_\delta(E)&\coloneqq\inf\bigg\{\sum_{i=1}^\infty\frac{\mass(B_{r_i}(x_i))}{{\rm diam}(B_{r_i}(x_i))}\;\bigg|\;
E\subseteq\bigcup_{i\in\NN}B_{r_i}(x_i),\;\sup_{i\in\NN}{\rm diam}(B_{r_i}(x_i))<\delta\bigg\},\\
\mathcal H^h(E)&\coloneqq\lim_{\delta\searrow 0}\mathcal H^h_\delta(E)=\sup_{\delta>0}\mathcal H^h_\delta(E),
\end{split}\]
respectively, for every set \(E\subseteq\XX\).
\subsubsection{Measured Gromov-Hausdorff convergence and tangents}
Let us recall the notion of \emph{pointed measured Gromov-Hausdorff} convergence (see e.g.\ \cite{GMS15}). We say that a pointed
metric measure space \((\XX,\dist,\mass,p)\) is \emph{normalised} provided \(C_p^1(\mass)=1\), where we set
\[
C_p^r=C_p^r(\mass)\coloneqq\int_{B_r(p)}\left(1-\frac{\dist(\cdot,p)}{r}\right)\,\diff\mass,\quad\text{ for every }r>0.
\]
If \((\XX,\dist,\mass,p)\) is any pointed metric measure space, then \((\XX,\dist,\mass_p^1,p)\) is normalised, where
\[
\mass_p^r\coloneqq C_p^r(\mass)^{-1}\mass,\quad\text{ for every }r>0.
\]
Let \(C\colon(0,+\infty)\to(0,+\infty)\) be a given non-decreasing function. Then we denote by \(\mathbb X_{C(\cdot)}\) the family
of all the equivalence classes of normalised pointed metric measure spaces that are \(C(\cdot)\)-doubling, in the sense that
\[
\mass(B_{2r}(x))\leq C(R)\,\mass(B_r(x)),\quad\text{ for every }x\in\XX\text{ and }0<r\leq R.
\]
The equivalence classes are intended with respect to the following equivalence relation: we identify two pointed metric
measure spaces \((\XX_1,\dist_1,\mass_1,p_1)\), \((\XX_2,\dist_2,\mass_2,p_2)\) provided there exists a bijective isometry
\(\varphi\colon{\rm spt}(\mass_1)\to{\rm spt}(\mass_2)\) such that \(\varphi(p_1)=p_2\) and \(\varphi_*\mass_1=\mass_2\).
\begin{defn}[Pointed measured Gromov-Hausdorff]
Let \(C\colon(0,+\infty)\to(0,+\infty)\) be non-decreasing. Let \((\XX,\dist,\mass,p),(\XX_i,\dist_i,\mass_i,p_i)\in\mathbb X_{C(\cdot)}\)
for \(i\in\NN\) be given. Then we say that \((\XX_i,\dist_i,\mass_i,p_i)\to(\XX,\dist,\mass,p)\) in the \emph{pointed measured Gromov-Hausdorff sense}
(briefly, in the \emph{pmGH sense}) provided there exist a proper metric space \((\ZZ,\dist_\ZZ)\) and isometric embeddings
\(\iota\colon\XX\to\ZZ\) and \(\iota_i\colon\XX_i\to\ZZ\) for \(i\in\NN\) such that \(\iota_i(p_i)\to\iota(p)\) and
\((\iota_i)_*\mass_i\rightharpoonup\iota_*\mass\) in duality with \(C_{bs}(\ZZ)\), meaning that \(\int f\circ\iota_i\,\diff\mass_i
\to\int f\circ\iota\,\diff\mass\) for every \(f\in C_{\rm bs}(\ZZ)\). The space \(\ZZ\) is called a \emph{realisation} of the pmGH convergence
\((\XX_i,\dist_i,\mass_i,p_i)\to(\XX,\dist,\mass,p)\).
\end{defn}
For brevity, we will identify \((\iota_i)_*\mass_i\) with \(\mass_i\) itself.
It is possible to construct a distance \(\dist_{\rm pmGH}\) on \(\mathbb X_{C(\cdot)}\) whose converging sequences are exactly those converging
in the pointed measured Gromov-Hausdorff sense. Moreover, the metric space \((\mathbb X_{C(\cdot)},\dist_{\rm pmGH})\) is compact.
\begin{defn}[pmGH tangent]
Let \(C\colon(0,+\infty)\to(0,+\infty)\) be non-decreasing. Then
\[
{\rm Tan}_p(\XX,\dist,\mass)\coloneqq\bigg\{(\YY,\dist_\YY,\mass_\YY,q)\in{\mathbb X}_{C(\cdot)}\;\bigg|\;
\exists\,r_i\searrow 0:\,(\XX,r_i^{-1}\dist,\mass_p^{r_i},p)\overset{\rm pmGH}\longrightarrow(\YY,\dist_\YY,\mass_\YY,q)\bigg\}.
\]
\end{defn}
Notice that \((\XX,r^{-1}\dist,\mass_p^r,p)\in\mathbb X_{C(\cdot)}\) holds for every \((\XX,\dist,\mass,p)\in\mathbb X_{C(\cdot)}\) and \(r\in(0,1)\),
thus accordingly the family \({\rm Tan}_p(\XX,\dist,\mass)\) is (well-defined and) non-empty.
\begin{defn}[Regular set]
Let \(n\in\NN\) be given. Let \(C\colon(0,+\infty)\to(0,+\infty)\) be any non-decreasing function such that
\((\RR^n,\dist_e,\underline{\mathcal L}^n,0)\in\mathbb X_{C(\cdot)}\), where \(\dist_e\) stands for the Euclidean distance
\(\dist_e(x,y)\coloneqq|x-y|\) on \(\RR^n\), while \(\underline{\mathcal L}^n\)
is the normalised measure
\((\mathcal L^n)_0^1=\frac{n+1}{\omega_n}\mathcal L^n\). Then the set of \emph{\(n\)-regular points} of a given element
\((\XX,\dist,\mass,p)\in\mathbb X_{C(\cdot)}\) is defined as
\[
\mathcal R_n=\mathcal R_n(\XX)\coloneqq\Big\{x\in\XX\;\Big|\;{\rm Tan}_x(\XX,\dist,\mass)=\big\{(\RR^n,\dist_e,\underline{\mathcal L}^n,0)\big\}\Big\}.
\]
\end{defn}
\begin{rem}\label{RnBorel}
We point out that the set \(\mathcal R_n(\XX)\) of \(n\)-regular points is Borel measurable. To check it, define
\(\phi\colon\XX\to[0,+\infty)\) as \(\phi(x)\coloneqq\varlimsup_{r\searrow 0}\dist_{\rm pmGH}\big((\XX,r^{-1}\dist,\mass_x^r,x),
(\RR^n,\dist_e,\underline{\mathcal L}^n,0)\big)\). One can readily verify that \((0,1)\ni r\mapsto(\XX,r^{-1}\dist,\mass_x^r,x)\in\mathbb X_{C(\cdot)}\)
is \(\dist_{\rm pmGH}\)-continuous for any given \(x\in\XX\), whence it follows that
\begin{equation}\label{eq:formula_dist_tg}
\phi(x)=\inf_{k\in\NN}\sup_{q\in\QQ\cap(0,1/k)}\dist_{\rm pmGH}\big((\XX,q^{-1}\dist,\mass_x^q,x),(\RR^n,\dist_e,\underline{\mathcal L}^n,0)\big),
\quad\text{ for every }x\in\XX.
\end{equation}
Since \(\XX\ni x\mapsto(\XX,r^{-1}\dist,\mass_x^r,x)\in\mathbb X_{C(\cdot)}\) is \(\dist_{\rm pmGH}\)-continuous for any given \(r\in(0,1)\),
we deduce that \(\XX\ni x\mapsto\dist_{\rm pmGH}\big((\XX,q^{-1}\dist,\mass_x^q,x),(\RR^n,\dist_e,\underline{\mathcal L}^n,0)\big)\) is a continuous
function for any \(q\in\QQ\cap(0,1)\). Consequently, \eqref{eq:formula_dist_tg} ensures that \(\mathcal R_n(\XX)=\{x\in\XX\,:\,\phi(x)=0\}\)
is a Borel set (in fact, a countable intersection of \(F_\sigma\) sets), as we claimed.
\end{rem}
\begin{defn}[Convergences along pmGH converging sequences]
Let \((\XX,\dist,\mass,p)\in\mathbb X_{C(\cdot)}\) and \((\XX_i,\dist_i,\mass_i,p_i)\in\mathbb X_{C(\cdot)}\) for \(i\in\NN\)
satisfy \((\XX_i,\dist_i,\mass_i,p_i)\to(\XX,\dist,\mass,p)\) in the pmGH sense, with realisation \(\ZZ\). Then we give the following
definitions:
\begin{itemize}
\item[\(\rm i)\)]
Let \(f_i\colon\XX_i\to\RR\) for \(i\in\NN\) and \(f\colon\XX\to\RR\) be given functions. Then we say that \(f_i\)
\emph{uniformly converges} to \(f\) provided for any \(\varepsilon>0\) there exists \(\delta>0\) such that \(|f_i(x_i)-f(x)|\leq\varepsilon\)
holds for every \(i\geq\delta^{-1}\) and \(x_i\in\XX_i\), \(x\in\XX\) with \(\dist_\ZZ(x_i,x)\leq\delta\).
\item[\(\rm ii)\)] Let \(f_i\colon\XX_i\to\RR\) for \(i\in\NN\) and \(f\colon\XX\to\RR\) be given functions. Then we say that \(f_i\)
\emph{locally uniformly converges} to \(f\) provided for any $R>0$, $f_{i|B_R(p_i)}$ uniformly converges to $f_{| B_R(p)}$.
\item[\(\rm iii)\)] Let \(E_i\subseteq\XX_i\) for \(i\in\NN\) and \(E\subseteq\XX\) be given Borel sets. Suppose that \(\mass_i(E_i)<+\infty\)
for every \(i\in\NN\) and \(\mass(E)<+\infty\). Then we say that \(E_i\to E\) \emph{(strongly) in \(L^1\)} provided \(\mass_i(E_i)\to\mass(E)\)
and \(\mass_i\llcorner E_i\rightharpoonup\mass\llcorner E\) in duality with \(C_{\rm bs}(\ZZ)\).
\item[\(\rm iv)\)]  Let \(E_i\subseteq\XX_i\) for \(i\in\NN\) and \(E\subseteq\XX\) be given Borel sets. Then we say that \(E_i\to E\)
\emph{(strongly) in \(L^1_{\rm loc}\)} provided \(E_i\cap B_R(p_i)\to E\cap B_R(p)\) in \(L^1\) for every \(R>0\).
\end{itemize}
\end{defn}
\subsubsection{Sobolev calculus}
Given a metric measure space \((\XX,\dist,\mass)\), we define the \emph{Sobolev space} \(W^{1,2}(\XX)\) as the set of all
functions \(f\in L^2(\mass)\) for which there exists \((f_n)_{n\in\NN}\subseteq{\rm LIP}_{\rm bs}(\XX)\) such that \(f_n\to f\)
in \(L^2(\mass)\) and \((\lip f_n)_{n\in\NN}\) is a bounded sequence in \(L^2(\mass)\). Then \(W^{1,2}(\XX)\) becomes a
Banach space if endowed with the following norm:
\[
\|f\|_{W^{1,2}(\XX)}\coloneqq\bigg(\int|f|^2\,\diff\mass+\inf_{(f_n)_n}\varliminf_{n\to\infty}\int\lip^2 f_n\,\diff\mass\bigg)^{1/2},
\quad\text{ for every }f\in W^{1,2}(\XX),
\]
where the infimum is taken among all those sequences \((f_n)_{n\in\NN}\subseteq{\rm LIP}_{\rm bs}(\XX)\) such that \(f_n\to f\) in \(L^2(\mass)\) and
\((\lip f_n)_{n\in\NN}\) is bounded in \(L^2(\mass)\). Given any function \(f\in W^{1,2}(\XX)\), there exists a unique element \(|\DIFF f|\in L^2(\mass)\),
called the \emph{minimal relaxed slope} of \(f\), such that the Sobolev norm of \(f\) can be expressed as
\(\|f\|_{W^{1,2}(\XX)}^2=\|f\|_{L^2(\mass)}^2+\||\DIFF f|\|_{L^2(\mass)}^2\). Moreover, there exists a sequence
\((f_n)_{n\in\NN}\subseteq{\rm LIP}_{\rm bs}(\XX)\) such that \(f_n\to f\) and \(\lip f_n\to|\DIFF f|\) in \(L^2(\mass)\).
This notion of Sobolev space, proposed in \cite{AmbrosioGigliSavare11-3}, is an equivalent reformulation of the one
introduced in \cite{Cheeger00}. See \cite{AmbrosioGigliSavare11-3} for the equivalence between these two and other approaches.

The \emph{Sobolev capacity} is the set-function on \(\XX\) defined as follows:
\begin{equation}\label{eqn:Capacity}
{\rm Cap}(E)\coloneqq\inf_{f}\|f\|_{W^{1,2}(\XX)}^2,\quad\text{ for every set }E\subseteq\XX;
\end{equation}
where the infimum is taken among all \(f\in W^{1,2}(\XX)\) such that \(f\geq 1\) holds \(\mass\)-a.e.\ on some open neighbourhood
of \(E\). Here we adopt the convention that \({\rm Cap}(E)\coloneqq+\infty\) whenever no such \(f\) exists. It holds that \(\rm Cap\)
is a submodular outer measure on \(\XX\), which is finite on bounded sets and satisfies \(\mass(E)\leq{\rm Cap}(E)\) for every \(E\subseteq\XX\) Borel.
\medskip

We shall also work with \emph{local Sobolev spaces}, whose definition we are going to recall. Fix an open set \(\Omega\subseteq\XX\).
Then we define \(W^{1,2}_{\rm loc}(\Omega)\) as the space of all functions \(f\in L^2_{\rm loc}(\Omega,\mass)\) such that
\(\eta f\in W^{1,2}(\XX)\) holds for every \(\eta\in{\rm LIP}_{\rm bs}(\Omega)\). Since the minimal relaxed slope is a local object,
meaning that for any choice of \(f_1,f_2\in W^{1,2}(\XX)\) we have that
\[
|\DIFF f_1|=|\DIFF f_2|,\quad\text{ holds }\mass\text{-a.e.\ on }\{f_1=f_2\},
\]
it makes sense to associate to any \(f\in W^{1,2}_{\rm loc}(\Omega)\) the function \(|\DIFF f|\in L^2_{\rm loc}(\Omega,\mass)\) given by
\[
|\DIFF f|\coloneqq|\DIFF(\eta f)|,\quad\mass\text{-a.e.\ on }\{\eta=1\},
\]
for every \(\eta\in{\rm LIP}_{\rm bs}(\Omega)\). The local Sobolev space \(W^{1,2}(\Omega)\) is defined as
\[
W^{1,2}(\Omega)\coloneqq\big\{f\in W^{1,2}_{\rm loc}(\Omega)\;\big|\;f,|\DIFF f|\in L^2(\mass)\big\}.
\]
Finally, we define \(W^{1,2}_0(\Omega)\) as the closure of \({\rm LIP}_{\rm bs}(\Omega)\) in \(W^{1,2}(\Omega)\).
\medskip

Following the terminology introduced in \cite{Gigli12}, we say that a given metric measure space \((\XX,\dist,\mass)\) is
\emph{infinitesimally Hilbertian} provided \(W^{1,2}(\XX)\) (and thus also \(W^{1,2}(\Omega)\) for any \(\Omega\subseteq\XX\) open)
is a Hilbert space. Under this assumption, the mapping
\[
W^{1,2}(\Omega)\times W^{1,2}(\Omega)\ni(f,g)\mapsto\nabla f\cdot\nabla g\coloneqq
\frac{|\DIFF(f+g)|^2-|\DIFF f|^2-|\DIFF g|^2}{2}\in L^1(\Omega,\mass)
\]
is bilinear and continuous. We say that a given function \(f\in W^{1,2}(\Omega)\) has a \emph{Laplacian}, briefly \(f\in D(\Delta,\Omega)\),
provided there exists a function \(\Delta f\in L^2(\Omega,\mass)\) such that
\begin{equation}\label{eq:def_Laplacian}
\int_\Omega\nabla f\cdot\nabla g\,\diff\mass=-\int_\Omega g\Delta f\,\diff\mass,\quad\text{ for every }g\in W^{1,2}_0(\Omega).
\end{equation}
No ambiguity may arise, since \(\Delta f\) is uniquely determined by \eqref{eq:def_Laplacian}. The set \(D(\Delta,\Omega)\) is a
linear subspace of \(W^{1,2}(\Omega)\) and the resulting operator \(\Delta\colon D(\Delta,\Omega)\to L^2(\Omega,\mass)\) is linear.
For the sake of brevity, we shorten \(D(\Delta,\XX)\) to \(D(\Delta)\). By a \emph{harmonic} function on \(\Omega\) we mean an
element \(f\in D(\Delta,\Omega)\) such that \(\Delta f=0\).
\subsubsection{BV calculus}
We begin by recalling the notions of function of bounded variation and of set of finite perimeter in the context of metric measure spaces,
following \cite{MIRANDA2003}.
\begin{defn}[Function of bounded variation]
Let $(\XX,\dist,\mass)$ be a metric measure space. Let $f\in L^1_{\mathrm{loc}}(\XX,\mass)$ be given. Then we define
\[
|Df|(\Omega) \coloneqq \inf\bigg\{\varliminf_{i\to\infty}\int_\Omega\lip f_i\,\diff\mass\;\bigg|\;\text{$(f_i)_{i\in\NN}\subseteq{\rm LIP}_{\rm loc}(\Omega),
\,f_i \to f $ in $L^1_{\mathrm{loc}}(\Omega,\mass)$}\bigg\},
\]
for any open set $\Omega\subseteq\XX$. We declare that a function \(f\in L^1_{\rm loc}(\XX,\mass)\) is of \emph{local bounded variation},
briefly \(f\in{\rm BV}_{\rm loc}(\XX)\), if \(|Df|(\Omega)<+\infty\) for every \(\Omega\subseteq\XX\) open bounded. In this case, it is well known that $|\DIFF f|$ extends to a locally finite measure on $\XX$. Moreover,
a function $f \in L^1(\XX,\mass)$ is said to belong to the space of \emph{functions of bounded variation}
${\rm BV}(\XX)={\rm BV}(\XX,\dist,\mass)$ if $|Df|(\XX)<+\infty$. 
\end{defn}
\begin{defn}[Set of finite perimeter]
Let \((\XX,\dist,\mass)\) be a metric measure space. Let \(E\subseteq\XX\) be a Borel set and \(\Omega\subseteq\XX\) an open set.
Then we define the \emph{perimeter} of $E$ in $\Omega$ as
\[
P(E,\Omega) \coloneqq \inf\bigg\{\varliminf_{i\to\infty}\int_\Omega \lip f_i\,\diff\mass\;\bigg|\;\text{$(f_i)_{i\in\NN}\subseteq{\rm LIP}_{\rm loc}(\Omega),
\,f_i \to \chi_E $ in $L^1_{\rm loc}(\Omega,\mass)$} \bigg\},
\]
in other words \(P(E,\Omega)\coloneqq|D\chi_E|(\Omega)\).
We say that $E$ has \emph{locally finite perimeter} if $P(E,\Omega)<+\infty$ for every \(\Omega\subseteq\XX\) open bounded.
Moreover, we say that $E$ has \emph{finite perimeter} if $P(E,\XX)<+\infty$, and we denote $P(E)\coloneqq P(E,\XX)$.
\end{defn}
Given a uniformly locally doubling space \((\XX,\dist,\mass)\) and a Borel set \(E\subseteq\XX\), we define the \emph{essential boundary} of \(E\) as
\[
\partial^*E\coloneqq\bigg\{x\in\XX\;\bigg|\;\varlimsup_{r\searrow 0}\frac{\mass(E\cap B_r(x))}{\mass(B_r(x))}>0,
\,\varlimsup_{r\searrow 0}\frac{\mass(E^c\cap B_r(x))}{\mass(B_r(x))}>0\bigg\}.
\]
Then \(\partial^*E\) is a Borel subset of the topological boundary \(\partial E\) of \(E\). Moreover, if \((\XX,\dist,\mass)\) is a PI space,
then \(P(E,\cdot)\) is concentrated on \(\partial^*E\) (see \cite[Theorem 5.3]{amb01}).
\begin{defn}[Precise representative]
Let $(\XX,\dist,\mass)$ be a metric measure space and let $f\colon\XX\to\RR$ be a Borel function.		
Then we define the \emph{approximate lower} and \emph{upper limits} as 
\begin{alignat*}{3}
	f^{\wedge}(x)&\defeq \apliminf_{y\rightarrow x} f(y)&&\defeq\sup&&\left\{t\in\bar{\RR}\;:\;\lim_{r\searrow 0} \frac{\mass(B_r(x)\cap\{f<t\})}{\mass(B_r(x))}=0\right\}, \\
	f^{\vee}(x)&\defeq \aplimsup_{y\rightarrow x} f(y)&&\defeq\inf&&\left\{t\in\bar{\RR}\;:\;\lim_{r\searrow 0} \frac{\mass(B_r(x)\cap\{f>t\})}{\mass(B_r(x))}=0\right\},
\end{alignat*}
for every \(x\in\XX\). Here we adopt the convention that $$\inf\emptyset=+\infty,\quad \text{and}\quad\sup\emptyset=-\infty.$$
Moreover, we define the \emph{precise representative} \(\bar f\colon\XX\to\bar\RR\) of \(f\) as
\[
\bar f(x)\coloneqq\frac{f^\wedge(x)+f^\vee(x)}{2},\quad\text{ for every }x\in\XX,
\]
where we adopt the convention that \(+\infty-\infty=0\).
\end{defn}
We define the \emph{jump set} \(J_f\subseteq\XX\) of the function \(f\) as the Borel set
\[
J_f\coloneqq\{x\in\XX\,:\,f^\wedge(x)<f^\vee(x)\}.
\]
It is well-known that if \((\XX,\dist,\mass)\) is a PI space and \(f\in\BV(\XX)\), then \(J_f\) is a countable union of
essential boundaries of sets of finite perimeter, so that in particular \(\mass(J_f)=0\).
See \cite[Proposition 5.2]{ambmirpal04}.
Moreover, as proved in \cite[Lemma 3.2]{kinkorshatuo}, it holds that
\begin{equation}\label{eq:bar_f_finite}
|\DIFF f|(\XX\setminus\XX_f)=0,\quad\text{ where }\XX_f\coloneqq\big\{x\in\XX\;\big|\;-\infty<f^\wedge(x)\leq f^\vee(x)<+\infty\big\},
\end{equation}
thus in particular \(-\infty<\bar f(x)<+\infty\) holds for \(|\DIFF f|\)-a.e.\ \(x\in\XX\).
\begin{defn}
Let $(\XX,\dist,\mass)$ be a metric measure space and let $f:\XX\rightarrow\RR$ be Borel.
Then we define the \emph{subgraph} of $f$ as the Borel set \(\GG_f\subseteq\XX\times\RR\) as
$$\GG_f\defeq\{(x,t)\in\XX\times\RR\,:\,t<f(x)\}.$$
\end{defn}
\begin{lem}\label{lem:char_G_f}
Let $(\XX,\dist,\mass)$ be a locally uniformly doubling metric measure space and let $f\colon\XX\rightarrow\RR$ be a Borel function. Then it holds that
\begin{align*}
({x},{t})\in\partial^*\GG_f&\quad\Rightarrow\quad t\in [f^\wedge(x),f^\vee(x)],\\
t\in (f^\wedge(x),f^\vee(x))&\quad\Rightarrow\quad(x,t)\in\partial^*\GG_f.
\end{align*}
In particular, if \(x\in\XX_f\setminus J_f\), then { it holds that \(\partial^*\GG_f\cap(\{x\}\times\RR)\subseteq\{(x,\bar f(x))\}\).}
\end{lem}
\begin{proof}
In the proof the constant $C_D$ may change from line to line and it only depends on the doubling constant at scale $R=1$.
We can compute, for $r\in (0,\epsilon)$, using Fubini's Theorem,
\begin{align*}
	\frac{(\mass\otimes\mathcal L^1)(B_r(x,t)\cap\GG_f)}{(\mass\otimes\mathcal L^1)(B_r(x,t))}&\le \frac{(\mass\otimes\mathcal L^1)\big((B_r(x)\times B_r(t))\cap\GG_f\big)}{(\mass\otimes\mathcal L^1)(B_{r/2}(x)\times B_{r/2}(t))}\\&\le C_D\frac{(\mass\otimes\mathcal L^1)(\{(y,t)\in B_r(x)\times B_r(t):t<f(y)\})}{r\mass (B_r(x))}
	\\&\le C_D\frac{\dashint_{t-r}^{t+r}\mass(\{y\in B_r(x):s<f(y)\})\dd{s}}{\mass(B_r(x))}\\&\le C_D\frac{\mass(B_r(x)\cap \{f>t-\epsilon\})}{\mass (B_r(x))}.
\end{align*}
Therefore, if $(x,t)\in\partial^*\GG_f$, then $t\le f^{\vee}(x)$. Similarly, we can show that if $r\in (0,\epsilon)$,
$$\frac{(\mass\otimes\mathcal L^1)(B_r(x,t)\setminus\GG_f)}{(\mass\otimes\mathcal L^1)(B_r(x,t))}\le C_D\frac{\mass(B_r(x)\cap \{f<t+\epsilon\})}{\mass (B_r(x))},$$
which in turn shows that if $(x,t)\in\partial^*\GG_f$, then $t\ge f^{\wedge}(x)$.
Conversely, arguing as above, we can show that if $r\in (0,\epsilon)$,
$$ 	\frac{(\mass\otimes\mathcal L^1)(B_{2 r}(x,t)\cap\GG_f)}{(\mass\otimes\mathcal L^1)(B_{2 r}(x,t))}\ge C_D\frac{\mass(B_{r}(x)\cap \{f>t+\epsilon\})}{\mass (B_r(x))}$$
and that 
$$ 	\frac{(\mass\otimes\mathcal L^1)(B_r(x,t)\setminus\GG_f)}{(\mass\otimes\mathcal L^1)(B_r(x,t))}\ge C_D\frac{\mass(B_{r}(x)\cap \{f<t-\epsilon\})}{\mass (B_r(x))}$$
which yield the second claim.
\end{proof}
\begin{defn}[Decomposition of the total variation measure]\label{decomptv}
Let \((\XX,\dist,\mass)\) be a PI space and \(f\in\BV(\XX)\). Then we write \(|\DIFF f|\) as \(|\DIFF f|^a+|\DIFF f|^s\),
where \(|\DIFF f|^a\ll\mass\) and \(|\DIFF f|^s\perp\mass\). We can decompose the singular part \(|\DIFF f|^s\)
as \(|\DIFF f|^j+|\DIFF f|^c\), where the \emph{jump part} is given by \(|\DIFF f|^j\coloneqq|\DIFF f|\llcorner J_f\), while the
\emph{Cantor part} is given by \(|\DIFF f|^c\coloneqq|\DIFF f|^s\llcorner(\XX\setminus J_f)\).
\end{defn}
By \cite[Theorem 5.1]{Ambrosio-Pinamonti-Speight15} and its proof, taking into account the elementary inequality
$$ a\le \sqrt{1+a^2}\le 1+a,\quad\text{ for every }a>0,$$
(or see \cite[Proposition 4.2]{ambmirpal04}) we obtain the following proposition.
\begin{prop}\label{prop:behaviour_TV_PI}
Let $(\XX,\dist,\mass)$ be a $\PI$ space and $f\in\BVv$. Then $\GG_f$ is a set of locally finite perimeter in $\XX\times\mathbb R$ and, denoting with $\pi$ the projection map $\XX\times\RR\rightarrow\XX$, it holds that
$$ \abs{\DIFF f}\le \pi_*|{\DIFF{\chi_{\GG_f}}}|\le \abs{\DIFF f}+\mass.$$
In particular, if \(C\subseteq\XX\) is a Borel set satisfying \(|\DIFF f|^c=|\DIFF f|\llcorner C\), then it holds that
$$ \pi_*(|{\DIFF{\chi_{\GG_f}}}|\llcorner C\times\RR)= |\DIFF f|\llcorner C.$$
\end{prop}

\begin{defn}\label{def:TotalVariationBV}
	Let $(\XX,\dist,\mass)$ be a metric measure space and $F\in \BV_{\rm loc}(\XX)^k$. We define
$$
|\DIFF F|(\Omega)\defeq\inf\bigg\{\liminf_{i\rightarrow\infty}\int_\Omega \Big({\sum_{j=1}^k(\lip F^j_i)^2}\Big)^{1/2}\,\diff\mass\;\bigg|\;(F_i)_i\subseteq \LIP_{\mathrm{loc}}(\Omega)^k,\  F_i\rightarrow F\text{ in }L^1_{\mathrm{loc}}(\Omega)^k\bigg\}
$$
for any open set $\Omega\subseteq\XX$. Then we extend this definition to Borel subsets of $\XX$, as done in the scalar case (see \cite[
Section 2.3]{BGBV} and the references therein). We also define $$J_F\defeq\bigcup_{i=1}^k J_{F_i}$$
\end{defn}
It is clear that Definition \ref{decomptv} extends immediately to the vector valued case.
\subsection{\texorpdfstring{\(\RCD\)}{RCD} spaces}\label{sec:RCD}
We assume the reader is familiar with the language of \(\RCD(K,N)\) spaces. Recall that an \(\RCD(K,N)\) space is an infinitesimally Hilbertian
metric measure space verifying the Curvature-Dimension condition \({\rm CD}(K,N)\), in the sense of Lott--Villani--Sturm, for some \(K\in\RR\)
and \(N\in[1,\infty)\). In this paper we only consider finite-dimensional \(\RCD(K,N)\) spaces, namely we assume \(N<\infty\).
Finite-dimensional \(\RCD\) spaces are PI. As proven in \cite{Mondino-Naber14,GP16-2,DPMR16,KelMon16,bru2018constancy}, the following structure theorem holds.
\begin{thm}
Let \((\XX,\dist,\mass)\) be an \(\RCD(K,N)\) space. Then there exists a number \(n\in\NN\) with \(1\leq n\leq N\), called the
\emph{essential dimension} of \((\XX,\dist,\mass)\), such that \(\mass(\XX\setminus\mathcal R_n)=0\). Moreover, the regular set
\(\mathcal R_n\) is \((\mass,n)\)-rectifiable and it holds that \(\mass\ll\mathcal H^n\llcorner\mathcal R_n\).
\end{thm}
Recall that \(\mathcal R_n\) is said to be \emph{\((\mass,n)\)-rectifiable} provided there exist Borel subsets \((A_i)_{i\in\NN}\) of
\(\mathcal R_n\) such that each \(A_i\) is biLipschitz equivalent to a subset of \(\RR^n\) and \(\mass(\mathcal R_n\setminus\bigcup_i A_i)=0\).
\subsubsection{Sobolev calculus on \(\RCD\) spaces}
We assume the reader is familiar with the language of \emph{\(L^p(\mass)\)-normed \(L^\infty(\mass)\)-modules} \cite{Gigli14} and
\emph{\(L^0({\rm Cap})\)-normed \(L^0({\rm Cap})\)-modules} \cite{debin2019quasicontinuous}.
Let \((\XX,\dist,\mass)\) be a given \(\RCD(K,N)\) space. We denote by \(L^2(T^*\XX)\) and \(L^2(T\XX)\) the \emph{cotangent module} and the
\emph{tangent module} of \((\XX,\dist,\mass)\), respectively. Moreover, \(L^0(T\XX)\) stands for the \(L^0(\mass)\)-completion of \(L^2(T\XX)\),
in the sense of \cite[Theorem/Definition 2.7]{Gigli17}.
A fundamental class of Sobolev functions on \(\XX\) is the algebra of \emph{test functions} \cite{Savare13,Gigli14}:
\[
{\rm Test}^\infty(\XX)\coloneqq\Big\{f\in D(\Delta)\cap L^\infty(\mass)\;\Big|\;|\DIFF f|\in L^\infty(\mass),
\,\Delta f\in W^{1,2}(\XX)\cap L^\infty(\mass)\Big\}.
\]
Since \(\RCD\) spaces enjoy the \emph{Sobolev-to-Lipschitz property}, each function in \({\rm Test}^\infty(\XX)\) has a Lipschitz representative.
Moreover, \({\rm Test}^\infty(\XX)\) is dense in \(W^{1,2}(\XX)\) and \(\nabla f\cdot\nabla g\in W^{1,2}(\XX)\) holds for every
\(f,g\in{\rm Test}^\infty(\XX)\). The class of \emph{test vector fields} is then defined as
\[
{\rm TestV}(\XX)\coloneqq\bigg\{\sum_{i=1}^k f_i\nabla g_i\;\bigg|\;k\in\NN,\,(f_i)_{i=1}^k,(g_i)_{i=1}^k\subseteq{\rm Test}^\infty(\XX)\bigg\}
\subseteq L^2(T\XX).
\]
We denote by \(L^0_{\rm Cap}(T\XX)\) the \emph{capacitary tangent module} on \((\XX,\dist,\mass)\) introduced in \cite[Theorem 3.6]{debin2019quasicontinuous}
and by \(\bar\nabla\colon{\rm Test}^\infty(\XX)\to L^0_{\rm Cap}(T\XX)\) the capacitary gradient operator. Given any Borel measure \(\mu\) on \(\XX\)
such that \(\mu\ll{\rm Cap}\) (meaning that \(\mu(N)=0\) for every \(N\subseteq\XX\) Borel with \({\rm Cap}(N)=0\)), we denote by
\(\pi_\mu\colon L^0({\rm Cap})\to L^0(\mu)\) the canonical projection.

Letting \(L^0_\mu(T\XX)\) be the quotient of \(L^0_{\rm Cap}(T\XX)\)
up to \(\mu\)-a.e.\ equality (where we identity two elements \(v,w\in L^0_{\rm Cap}(T\XX)\) if \(\pi_\mu(|v-w|)=0\) holds \(\mu\)-a.e.),
we have a natural projection map \(\pi_\mu\colon L^0_{\rm Cap}(T\XX)\to L^0_\mu(T\XX)\), which satisfies
\(|\pi_\mu(v)|=\pi_\mu(|v|)\) \(\mu\)-a.e.\ for all \(v\in L^0_{\rm Cap}(T\XX)\). The space \(L^0_\mu(T\XX)\) is an \(L^0(\mu)\)-normed
\(L^0(\mu)\)-module. As pointed out in \cite[Proposition 3.9]{debin2019quasicontinuous}, the quotient \(L^0_\mass(T\XX)\) can be identified
with the tangent module \(L^0(T\XX)\) and the projection \(\pi_\mass\colon L^0_{\rm Cap}(T\XX)\to L^0(T\XX)\) satisfies
\(\nabla f=\pi_\mu(\bar\nabla f)\) for every \(f\in{\rm Test}^\infty(\XX)\). Due to this consistency, to ease the notation we will
indicate the capacitary gradient of a test function \(f\) with \(\nabla f\) instead of \(\bar\nabla f\).
\medskip

The \emph{Hessian} of \(f\in{\rm Test}^\infty(\XX)\) is the unique tensor \({\rm Hess}(f)\in L^2(T^*\XX)\otimes L^2(T^*\XX)\) with
\[\begin{split}
&2\int h\,{\rm Hess}(f)(\nabla g_1\otimes\nabla g_2)\,\diff\mass\\
=\,&-\int\nabla f\cdot\nabla g_1\,{\rm div}(h\nabla g_2)+\nabla f\cdot\nabla g_2\,{\rm div}(h\nabla g_1)
+h\,\nabla f\cdot\nabla(\nabla g_1\cdot\nabla g_2)\,\diff\mass,
\end{split}\]
for every \(g_1,g_2,h\in{\rm Test}^\infty(\XX)\). Recall that a vector field \(v\in L^2(T\XX)\) is said to have a \emph{divergence},
briefly \(v\in D({\rm div})\), provided there exists a function \({\rm div}(v)\in L^2(\mass)\) such that
\begin{equation}\label{eq:def_divergence}
\int\nabla f\cdot v\,\diff\mass=-\int f\,{\rm div}(v)\,\diff\mass,\quad\text{ for every }f\in W^{1,2}(\XX);
\end{equation}
note that \({\rm div}(v)\) is uniquely determined by \eqref{eq:def_divergence}. The Hessian above is a local object:
\begin{equation}\label{eq:locality_Hessian}
\chi_{\{f_1=f_2\}}\cdot{\rm Hess}(f_1)=\chi_{\{f_1=f_2\}}\cdot{\rm Hess}(f_2),\quad\text{ for every }f_1,f_2\in{\rm Test}^\infty(\XX).
\end{equation}
The validity of this property allows to define the Hessian of a harmonic function \(f\) defined on an open set
\(\Omega\subseteq\XX\), as we are going to discuss. As proven in \cite{Jiang13}, the harmonic function \(f\colon\Omega\to\RR\)
is locally Lipschitz. In particular, \(\eta f\in{\rm Test}^\infty(\XX)\) holds for every cut-off function \(\eta\in{\rm Test}^\infty(\XX)\)
such that \({\rm spt}(\eta)\Subset\Omega\). As shown in \cite{AmbrosioMondinoSavare13-2,Mondino-Naber14}, there are plenty of cut-off test
functions: given any \(x\in\XX\) and \(0<r<R\), there exists \(\eta\in{\rm Test}^\infty(\XX)\) with \(0\leq\eta\leq 1\) such that
\(\eta=1\) on \(B_r(x)\) and \({\rm spt}(\eta)\Subset B_R(x)\). Thanks to this fact and to \eqref{eq:locality_Hessian}, it makes
sense to \(\mass\)-a.e.\ define the measurable function \(|{\rm Hess}(f)|\colon\Omega\to[0,+\infty)\) as
\[
|{\rm Hess}(f)|\coloneqq|{\rm Hess}(\eta f)|,\quad\mass\text{-a.e.\ on }\{\eta=1\},
\]
for every \(\eta\in{\rm Test}^\infty(\XX)\) such that \({\rm spt}(\eta)\Subset\Omega\).
\subsubsection{BV calculus on \(\RCD\) spaces}
Now we focus on BV functions and sets of finite perimeter on \(\RCD(K,N)\) spaces.
The following notion was introduced in \cite[Definition 4.1]{ambrosio2018rigidity}.
\begin{defn}[Tangents to a set of finite perimeter]
Let \((\XX,\dist,\mass,p)\) be a pointed \(\RCD(K,N)\) space, \(E\subseteq\XX\) a set of locally finite perimeter. Then
we define \({\rm Tan}_p(\XX,\dist,\mass,E)\) as the family of all quintuplets \((\YY,\dist_\YY,\mass_\YY,q,F)\) that verify the following two conditions:
\begin{itemize}
\item[\(\rm i)\)] \((\YY,\dist_\YY,\mass_\YY,q)\in{\rm Tan}_p(\XX,\dist,\mass)\),
\item[\(\rm ii)\)] \(F\subseteq\YY\) is a set of locally finite perimeter with \(\mass_\YY(F)>0\) for which the following property holds:
along a sequence \(r_i\searrow 0\) such that \((\XX,r_i^{-1}\dist,\mass_p^{r_i},p)\to(\YY,\dist_\YY,\mass_\YY,q)\) in the pmGH sense, with
realisation \(\ZZ\), it holds that \(\chi_E^i\to\chi_F\) in \(L^1_{\rm loc}\), where by \(\chi_E^i\) we mean the characteristic function
of \(E\) intended in the rescaled space \((\XX,r_i^{-1}\dist)\). If this is the case, we write
$$
(\XX,r_i^{-1}\dist,\mass_p^{r_i},p,E)\to(\YY,\dist_\YY,\mass_\YY,q,F).
$$
\end{itemize}
\end{defn}

The following theorem is extracted from \cite[Theorem 3.13]{BGBV}, see also \cite[Theorem 2.4]{bru2019rectifiability}.
\begin{thm}
Let $(\XX,\dist,\mass)$ be an $\RCD(K,N)$ space and let $F\in \BV(\XX)^k$. Then there exists a unique, up to $|\DIFF F|$-a.e.\ equality, $\nu_F\in L^0_\capa(T\XX)^k$ such that $|\nu_F|=1\ |\DIFF F|$-a.e.\ and 
    $$
    \sum_{j=1}^k\int_\XX F_j{\rm div}(v_j)\dd{\mass}=-\int_\XX \pi_{|\DIFF F|}(v)\,\cdot\,\nu_F\dd{|\DIFF F|}\quad\text{ for every $v=(v_1,\dots,v_k)\in\TestV(\XX)^k$.}
    $$
\end{thm}

Notice that if $F\in \BV(\XX)^k$, we consider $\nu_F$ as an element of $L^0_\capa (T\XX)^k$ that is defined $|\DIFF F|$-a.e.. This allows us, via a standard localization procedure, to define $\nu_F$ even if $F$ is a vector valued function of locally bounded variation, or, in other words, if $F$ is a $k$-tuple of functions of locally bounded variation.
In particular, if $E$ is a set of locally finite perimeter, we naturally have a unique, up to $|\DIFF \chi_E|$-a.e.\ equality, $\nu_E\in L^0_\capa (T\XX)$, where we understand $\nu_E=\nu_{\chi_E}$.

Next we recall that, as proven in \cite{bru2019rectifiability},
each set of locally finite perimeter \(E\) in an \(\RCD(K,N)\) space \((\XX,\dist,\mass)\) satisfies \(|\DIFF\chi_E|\ll{\rm Cap}\). Notice however that the same result holds in every metric measure space, see \cite[Theorem 2.5]{BGBV}. By the coarea formula, this absolute continuity extends immediately to total variations, so that $$|\DIFF F|\ll\capa\qquad\text{for every }F\in\BV_{\rm loc}(\XX)^n.$$

The following proposition summarizes results about sets of finite perimeter that are now well-known in the context of $\PI$ spaces and are proved in \cite{amb01,erikssonbique2018asymptotic}, see also \cite{AmbrosioAhlfors}.
\begin{prop}\label{zuppa}
Let $(\XX,\dist,\mass)$ be a $\PI$ space and let $E\subseteq\XX$ be a set of locally finite perimeter. Then, for $|\DIFF\chi_E|$-a.e.\ $x\in\XX$ the following hold:
\begin{enumerate}[label=\roman*)]
    \item $E$ is \emph{asymptotically minimal} at $x$, in the sense that there exist $r_x>0$ and a function $\omega_x:(0,r_x)\rightarrow(0,\infty)$ with $\lim_{r\searrow 0}\omega_x(r)=0$ satisfying
    $$
    |\DIFF\chi_E|(B_r(x))\le(1+\omega_x(r))|\DIFF\chi_{E'}|(B_r(x)),\qquad\text{if $r\in (0,r_x)$ and $E'\Delta E\Subset B_r(x),$}
    $$
    \item $|\DIFF\chi_E|$ is \emph{asymptotically doubling} at $x$:
    $$\limsup_{r\searrow 0}\frac{|\DIFF\chi_E| (B_{2 r}(x))}{|\DIFF\chi_E| (B_{r}(x))}<\infty, $$
    \item we have the following estimates
    $$
    0<\liminf_{r\searrow 0} \frac{r |\DIFF\chi_E|(B_r(x))}{\mass(B_r(x))}\le \limsup_{r\searrow 0}\frac{r |\DIFF\chi_E|(B_r(x))}{\mass(B_r(x))}<\infty,
    $$
    \item the following density estimate holds
    $$ 
    \liminf_{r\searrow 0} \min\bigg\{\frac{\mass( B_r(x)\cap E)}{\mass(B_r(x))},\frac{\mass( B_r(x)\setminus E)}{\mass(B_r(x))}\bigg\}>0.$$
\end{enumerate}
\end{prop}
\begin{rem}\label{rmk:per_asympt_doubl}
It is well known (see \cite[Theorem 3.4.3 and page 77]{HKST15}) that for an asymptotically doubling measure the Lebesgue differentiation Theorem holds. In particular, if $E$ is a set of locally finite perimeter in a $\PI$ space and $f\in L^1(|\DIFF\chi_E|)$, then for $|\DIFF\chi_E|$-a.e.\ $x$ it holds
$$
\lim_{r\searrow 0}\int_{B_r(x)} |f(y)-f(x)|\dd{|\DIFF\chi_E|(y)}=0.
$$
\end{rem}

Let us now introduce the notion of reduced boundary of a set of locally finite perimeter. First, we introduce the set $\mathcal{R}_n^*$. Following \cite{AmbrosioTilli04}, given a metric measure space \((\XX,\dist,\mu)\) and a real number \(k\geq 0\),
we define the \emph{upper} and \emph{lower \(k\)-dimensional densities} of \(\mu\) as
\[
\overline\Theta_k(\mu,x)\coloneqq\varlimsup_{r\searrow 0}\frac{\mu(B_r(x))}{\omega_k r^k},\qquad
\underline\Theta_k(\mu,x)\coloneqq\varliminf_{r\searrow 0}\frac{\mu(B_r(x))}{\omega_k r^k},\qquad\text{ for every }x\in\XX,
\]
respectively. In the case where \(\overline\Theta_k(\mu,x)\) and \(\underline\Theta_k(\mu,x)\) coincide, we denote their
common value by \(\Theta_k(\mu,x)\in[0,+\infty]\) and we call it the \emph{\(k\)-dimensional density} of \(\mu\) at \(x\).
\begin{defn}
Let \((\XX,\dist,\mass)\) be an \(\RCD(K,N)\) space having essential dimension \(n\). Then we define the set
\(\mathcal R_n^*=\mathcal R_n^*(\XX)\subseteq\mathcal R_n\) as
\[
\mathcal R_n^*\coloneqq\big\{x\in\mathcal R_n\;\big|\;\exists\Theta_n(\mass,x)\in(0,+\infty)\big\}.
\]
\end{defn}
In the case in which $\mass=\mathcal{H}^N$, by Bishop--Gromov comparison one has that $\Theta_N(\mathcal{H}^N,x)$ exists and it is positive for every $x\in \XX$. Moreover, the volume convergence results in \cite{DPG17} and the lower semicontinuity of the density imply that $\Theta_N(\mathcal{H}^N,x)\leq 1$ for every $x\in\XX$.
Notice that the set \(\mathcal R_n^*\) is Borel, see Remark \ref{RnBorel}. As shown in \cite[Theorem 4.1]{AHT17}, it holds that \(\mass(\XX\setminus\mathcal R_n^*)=0\).

\begin{defn}[Reduced boundary]\label{def:ReducedBoundary}
Let \((\XX,\dist,\mass)\) be an \(\RCD(K,N)\) space. Let \(E\subseteq\XX\) be a set of locally finite perimeter.
Then we define the \emph{reduced boundary} \(\mathcal F E\subseteq\partial^* E\) of \(E\) as the set of all those
points \(x\in \mathcal R_n^*\) satisfying all the four conclusions of Proposition \ref{zuppa} and such that 
\begin{equation}\label{eqn:FnE}
\Tan_x(\XX,\dist,\mass,E)=\big\{(\RR^n,\dist_e,\underline{\mathcal L}^n,0,\{x_n>0\})\big\},
\end{equation}
where \(n\in\NN\), \(n\leq N\) stands for the essential dimension of \((\XX,\dist,\mass)\). We recall that the set of the points $x\in X$ that satisfy \eqref{eqn:FnE} is denoted by $\mathcal{F}_nE$. 
\end{defn}

As proven in \cite{bru2021constancy} after \cite{ambrosio2018rigidity,bru2019rectifiability}, taking into account the forthcoming Theorem \ref{thm:const_dim_cod1}, the perimeter measure
\(|\DIFF\chi_E|\) is concentrated on the reduced boundary \(\mathcal F E\).

\begin{rem}\label{rem:properties_FE}
By the proof of \cite[Corollary 4.10]{ambrosio2018rigidity}, by \cite[Corollary 3.4]{ambrosio2018rigidity}, and by the membership to $\mathcal R_n^*$,
we see that for any $x\in\FF E$ the following hold.
\begin{enumerate}[label=\roman*)]
\item If $r_i\searrow 0$ is such that \begin{equation}\label{convsemplice}(\XX, r_i^{-1}\dist,\mass_x^{r_i},x)\rightarrow (\RR^n,\dist_e, \underline{\mathcal L}^n,0)\end{equation}
in a realisation $(\ZZ,\dist_\ZZ)$, then, up to not relabelled subsequences and a change of coordinates in $\RR^n$,
$$(\XX, r_i^{-1}\dist,\mass_x^{r_i},x,E)\rightarrow (\RR^n,\dist_e, \underline{\mathcal L}^n,0,\{x_n>0\}),$$
in the same realisation $(\ZZ,\dist_\ZZ)$. Notice that, given a sequence $r_i\searrow 0$, it is always possible to find a subsequence satisfying \eqref{convsemplice}.
\item\label{conv}  If $r_i\searrow 0$ is such that $$(\XX, r_i^{-1}\dist,\mass_x^{r_i},x,E)\rightarrow (\RR^n,\dist_e, \underline{\mathcal L}^n,0,\{x_n>0\})$$ in a realisation $(\ZZ,\dist_\ZZ)$, then $|\DIFF\chi_E|$ weakly converge to $|\DIFF\chi_{\{x_n>0\}}|$ in duality with $C_{\mathrm{bs}}(Z)$.
\item We have
\begin{equation}\label{usefullimits}
\begin{split}
&\lim_{r\searrow 0}\frac{\mass(B_r(x))}{r^n}=\omega_n\Theta_n(\mass,x)\in (0,+\infty),\\
&\lim_{r\searrow 0}\frac{C_x^r}{r^n}=\frac{\omega_n}{n+1}\Theta_n(\mass,x),\\
&\lim_{r\searrow 0}\frac{|\DIFF\chi_E|(B_r(x))}{r^{n-1}}={\omega_{n-1}}\Theta_n(\mass,x).
\end{split}
\end{equation}
\end{enumerate}

\end{rem}

\begin{defn}[Good coordinates]\label{defn:good_coordinates}
Let \((\XX,\dist,\mass)\) be an \(\RCD(K,N)\) space of essential dimension \(n\). Let \(E\subseteq\XX\) be a set of locally finite perimeter
and \(x\in\mathcal F E\) be given. Then we say that an \(n\)-tuple \(u=(u^1,\ldots,u^n)\) of harmonic functions \(u^\ell\colon B_{r_x}(x)\to\RR\)
is a \emph{system of good coordinates} for \(E\) at \(x\) provided the following properties are satisfied:
\begin{itemize}
\item[\(\rm i)\)] For any \(\ell,j=1,\ldots,n\), it holds that
\[
\lim_{r\searrow 0}\fint_{B_r(x)}|\nabla u^\ell\cdot\nabla u^j-\delta_{\ell j}|\,\diff\mass=
\lim_{r\searrow 0}\fint_{B_r(x)}|\nabla u^\ell\cdot\nabla u^j-\delta_{\ell j}|\,\diff|\DIFF\chi_E|=0.
\]
\item[\(\rm ii)\)] For any \(\ell=1,\ldots,n\), it holds that
\begin{equation}\label{eq:def_nu_in_good_coord}
\exists\,\nu_\ell(x)\coloneqq\lim_{r\searrow 0}\fint_{B_r(x)}\nu_E\cdot\nabla u^\ell\,\diff|\DIFF\chi_E|,
\quad\lim_{r\searrow 0}\fint_{B_r(x)}|\nu_\ell(x)-\nu_E\cdot\nabla u^\ell|\,\diff|\DIFF\chi_E|=0.
\end{equation}
\item[\(\rm iii)\)] The resulting vector \(\nu(x)\coloneqq(\nu_1(x),\ldots,\nu_n(x))\in\RR^n\) satisfies \(|\nu(x)|=1\).
\end{itemize}
\end{defn}
{
The following theorem is proved in \cite[Theorem 3.6]{bru2021constancy}.
\begin{thm}
Let \((\XX,\dist,\mass)\) be an \(\RCD(K,N)\) space of essential dimension \(n\). Let \(E\subseteq\XX\) be a set of locally finite perimeter
and \(x\in\mathcal F E\) be given. Then, good coordinates exist at \(|\DIFF\chi_E|\)-a.e.\ point \(x\in\mathcal F E\).
\end{thm}
}

\begin{rem}\label{convdelta}
Let $(\XX,\dist,\mass)$ be an $\RCD(K,N)$ space of essential dimension $n$, let $x\in\XX$ and let $u=(u^1,\dots,u^n)$ be an $n$-tuple of harmonic functions satisfying 
$$
\lim_{r\searrow 0}\fint_{B_r(x)}|\nabla u^\ell\cdot\nabla u^j-\delta_{\ell j}|\,\diff\mass=0.
$$
Given a sequence of radii $r_i\searrow 0$ such that $$(\XX,r_i^{-1}\dist,\mass_x^{r_i},x)\rightarrow(\RR^n,\dist_e,\underline{\mathcal L}^n,0)$$ and fixed a realization of such convergence, it follows from the results recalled in \cite[Section 1.2.3]{bru2019rectifiability} (see the references therein, see also \cite[(1.22)]{bru2019rectifiability},
consequence of the improved Bochner inequality in \cite{Han14}) that, up to extracting a not relabelled subsequence, the functions in
$$\{r_i^{-1} u^j\}_i\qquad\text{for }j=1,\dots,n $$
converge locally uniformly to orthogonal coordinate functions of $\RR^n$.
\end{rem}

The ensuing result is taken from \cite[Proposition 4.8]{bru2021constancy}.
\begin{prop}\label{prop:conv_good_coord}
Let \((\XX,\dist,\mass)\) be an \(\RCD(K,N)\) space of essential dimension \(n\). Let \(E\subseteq\XX\) be a set of locally finite perimeter.
Then for \(|\DIFF\chi_E|\)-a.e.\ \(x\in\XX\) the following property holds. Suppose that \(u=(u^1,\dots,u^n)\colon B_r(x)\to\RR^n\) is a system
of good coordinates for \(E\) at \(x\). Let \(\nu(x)\in\RR^n\) be as in Definition \ref{defn:good_coordinates}. If the coordinates \((x_\ell)\)
on the (Euclidean) tangent space to \(\XX\) at \(x\) are chosen so that the maps \((u^\ell)\) converge to \((x_\ell)\colon\RR^n\to\RR^n\)
when properly rescaled, then the blow-up \(H\) of \(E\) at \(x\) (in the sense of finite perimeter sets) is
\[
H=\big\{y\in\RR^n\;\big|\;y\cdot\nu(x)\geq 0\big\}.
\]
\end{prop}
\subsubsection{Splitting maps}
Let us now present the notion of $\delta$-splitting map. We follow closely the presentation in \cite{bru2019rectifiability}, compare with \cite[Definition 3.4]{bru2019rectifiability}.
\begin{defn}[Splitting map]
Let \((\XX,\dist,\mass)\) be an \(\RCD(K,N)\) space. Let \(x\in\XX\), \(k\in\NN\), and \(r,\delta>0\) be given.
Then a map \(u=(u_1,\ldots,u_k)\colon B_r(x)\to\RR^k\) is said to be a \emph{\(\delta\)-splitting map} provided
the following properties hold:
\begin{itemize}
\item[\(\rm i)\)] \(u_\ell\) is harmonic, meaning that, for every $\ell=1,\dots,k$, $u_\ell\in D(\Delta,B_r(x))$ and $\Delta u_\ell=0$, and $u_\ell$ is $C_N$-Lipschitz for every $\ell=1,\dots,k$,
\item[\(\rm ii)\)] \(r^2\dashint_{B_r(x)}|{\rm Hess}(u_\ell)|^2\,\dd\mass\leq\delta\) for every \(\ell=1,\ldots,k\),
\item[\(\rm iii)\)] \(\dashint_{B_r(x)}|\nabla u_\ell\cdot\nabla u_j-\delta_{\ell j}|\,\dd\mass\leq\delta\) for every \(\ell,j=1,\ldots,k\).
\end{itemize}
\end{defn}
As already noticed in \cite[Remark 3.6]{bru2019rectifiability}, in the classical definition of $\delta$-splitting map in the smooth setting, in item i) above the stronger condition $|\nabla u|\leq 1+\delta$ is required. Anyway we stress that when $(\XX,\dist,\mass)$ is an $\RCD(-\delta,N)$ space and $u$ is a $\delta$-splitting map as above, we have that $\sup_{y\in B_{r/2}(x)}|\nabla u|(y)\leq 1+C_N\delta^{1/2}$, see \cite[Remark 3.3]{BrueNaberSemola20}, and compare with \cite[Equations (3.42)--(3.46)]{ChNa15}. This means that, for $\delta$ small enough, if $u$ is a $\delta$-splitting map on $B_r(x)$ on an $\RCD(-\delta,N)$ space as above, then it is a $C_N\delta^{1/2}$-splitting map on $B_{r/2}(x)$ in the classical smooth sense.

In the following Lemma we slightly improve previous results obtained in \cite{bru2019rectifiability, bru2021constancy}, and we show that we can find good coordinates with respect to \textit{every} $\mathrm{BV}_{\mathrm{loc}}$ function.
\begin{lem}\label{lemsplitting}
Let $(\XX,\dist,\mass)$ be an $\RCD(K,N)$ space of essential dimension $n$ and $\eta\in(0,1)$. Then there exists a sequence of $n$-tuples
of harmonic \(C_{K,N}\)-Lipschitz maps $\{u_k\}_k$,
$$
u_k=(u_k^1,\dots,u_k^n):B_{2r_k}(x_k)\rightarrow\RR^n,
$$
and a sequence of pairwise disjoint Borel sets $\{D_k\}_k$ with $D_k\subseteq B_{r_k}(x_k)$ such that 
\begin{enumerate}[label=\roman*)]
\item for every $f\in\BV_{\mathrm{loc}}(\XX)$, $$\abs{\DIFF f}\left(\XX\setminus\bigcup_k D_k\right)=0,$$
\item for every $x\in D_k$, $u_k$ is an $\eta$-splitting map on $B_r(x)$, for any $r\in (0,r_k)$.
\item there exists a Borel matrix-valued map
$M=(M_{\ell,j}):D_k\rightarrow\RR^{n\times n}$ satisfying
\begin{equation}\label{eqn:ELEBESGUE}
\lim_{r\searrow 0}\dashint_{B_r(x)} |\nabla u^\ell_k\,\cdot\,\nabla u^j_k-M(x)_{\ell,j}|\dd{\mass}=0.
\end{equation}

\end{enumerate} 
\end{lem}
To any such collection of \(\eta\)-splitting maps, we can therefore associate a natural map $$\bigcup_k D_k\rightarrow \NN\qquad x\mapsto k(x).$$
\begin{proof}
The proof follows the arguments given in the proof of \cite[Theorem 3.2]{bru2019rectifiability}. However, as we need a slightly stronger statement,
we include the details of the proof.

Fix a countable dense set $S\subseteq\Reg_n$. Let $y\in S$ be given. If $\epsilon>0$ is small enough and $r\in (0,\sqrt{\epsilon/\abs{K}})\cap \QQ$
is such that 
$$
\dist_{\rm pmGH}\big((\XX,r^{-1}\dist,\mass_y^r,y),(\RR^n,\dist_e,\underline{\mathcal L}^n,0)\big) <\epsilon,
$$
then, by \cite[Corollary 3.10]{bru2019rectifiability},
we obtain a $\delta$-splitting map $u_{y,r}:B_{5 r}(y)\rightarrow\RR^n$ for some $\delta$ (which can be made arbitrarily small, taking $\epsilon$ small enough). 
Let $$D_{y,r}\defeq\big\{x\in B_{5/4 r}(y)\;\big|\;u_{y,r}\text{ is an $\eta$-splitting map on $B_s(x)$ for every $s\in (0,5/4 r)$}\big\}.$$
The claim of the lemma will be proved with the sequence of sets $\{D_{y,r}\}_{y,r}$ and maps $\{u_{y,r}\}_{y,r}$, after having made the sets disjoint and restricted the maps.

Assume now, by contradiction, that the claim is false. Then, using a locality argument and the coarea formula, we find a set of finite perimeter $E\subseteq\XX$ such that \begin{equation}\label{contra}
	\abs{\DIFF \chi_E}\bigg(\XX\setminus\bigcup_{y,r}D_{y,r}\bigg)>0.
\end{equation}
Fix $\epsilon>0$ to be determined later. 
  If $x\in \FF E$, then there exists $r=r(x)\in\QQ\cap (0,1)$ such that $\abs{K} r^2< \epsilon<4$ and
$$
\dist_{\rm pmGH}\big((\XX,r^{-1}\dist,\mass_x^r,x),(\RR^n,\dist_e,\underline{\mathcal L}^n,0)\big) <\epsilon
\qquad\text{and}\qquad
\frac{r\abs{\DIFF \chi_{E}}(B_{r/4}(x))}{\mass(B_{r/4}(x))}>2\frac{\omega_{n-1}}{\omega_n}.
$$
By density of \(S\) and thanks to an easy continuity argument, we deduce that for some point $y=y(x)\in S\cap B_{r/2}(x)$,
\begin{equation}\label{eqn:AccorciamoLeDistanze}
\dist_{\rm pmGH}\big((\XX,r^{-1}\dist,\mass_{y}^r,y),(\RR^n,\dist_e,\underline{\mathcal L}^n,0)\big) <\epsilon,
\qquad\text{and}\qquad
\frac{r\abs{\DIFF \chi_{E}}(B_{r/4}(y))}{\mass(B_{r/4}(y))}>2\frac{\omega_{n-1}}{ \omega_n}.
\end{equation}
By the discussion above (that is, \cite[Corollary 3.10]{bru2019rectifiability}),
we obtain a $\delta$-splitting map $u_{y,r}:B_{5 r}(y)\rightarrow\RR^n$ for some $\delta=\delta(\epsilon)$ (which can be made arbitrarily small, taking $\epsilon$ small enough). By \cite[Corollary 3.12]{bru2019rectifiability}, $u_{y,r}$ is a $C_N \delta^{1/4}$-splitting map on $B_s(x)$ for any $x\in D_{y,r}^\epsilon\subseteq B_{5/4 r }(y)$ and $s\in(0,5/4 r)$, where $$\HH^h_5(B_{5/4 r}(y)\setminus D^\epsilon_{y,r})\le C_N\delta^{1/2}\frac{\mass(B_{5/2 r}(x))}{5/2 r}.$$
Therefore, $D^\epsilon_{y,r}\subseteq D_{y,r}$ if $C_N\delta^{1/4}<\eta$.

We apply Vitali covering lemma to the family $\{B_{r(x)/4}(y(x))\}_{x\in \FF E}$ constructed arguing as above and we obtain a sequence of disjoint balls $\{B_{r(x_i)/4}(y(x_i))\}_i$ such that $$\FF E\subseteq \bigcup_i B_{5/4 r(x_i)}(y(x_i)).$$ Set $$D^\epsilon\defeq \bigcup_{i} D^\epsilon_{y(x_i),r(x_i)}.$$
Following the computations in the proof of \cite[Theorem 3.2]{bru2019rectifiability},
we obtain that 
\begin{equation}
\begin{split}
\HH^h_5(\FF E\setminus D^\epsilon)&\leq \sum_{i\in\mathbb N}\mathcal{H}^h_5\left(B_{5/4r(x_i)}(y(x_i))\setminus D^{\epsilon}_{y(x_i),r(x_i)}\right)\\
&\leq C_N\delta^{1/2}\sum_{i\in\mathbb N}\frac{\mass(B_{5/2r(x_i)}(y(x_i)))}{5/2r(x_i)}\\
&\leq C_{N}\delta^{1/2}\sum_{i\in\mathbb N}\frac{\mass(B_{r(x_i)/4}(y(x_i)))}{r(x_i)/4} \\
&\le C_N \delta^{1/2}|\DIFF\chi_E|(\XX),
\end{split}
\end{equation}
where the constants $C_N$ may change from line to line,
in the third inequality we are using the doubling property together with the fact that $r(x_i)$ is sufficiently small, and in the last inequality we are using \eqref{eqn:AccorciamoLeDistanze} together with the fact that $\{B_{r(x_i)/4}(y(x_i))\}$ are disjoint.
Let now $\{\epsilon_i\}_i$ with $\epsilon_i\searrow 0$ be such that the corresponding $\{\delta_i\}_i$ satisfy both $\delta_i^{1/2} \le 2^{-i}$ and $C_N\delta_i^{1/4}<\eta$, and set 
$$ G\defeq \bigcup_{i} D^{\epsilon_i}\subseteq D_{y,r}.$$
Then $\HH_5^h( \FF E\setminus G) =0$, which contradicts \eqref{contra}.

Finally, item iii) is a direct consequence of the fact that, since $u^\ell_k$ is harmonic for every $\ell=1,\dots,n$ and $k\in\mathbb N$, one can give a pointwise meaning to $\nabla u^\ell_k(x)\cdot \nabla u^j_k(x)$, compare with \cite[Remark 2.10]{bru2021constancy}. 
\end{proof}
\begin{defn}\label{defn:GoodCollection}
Let \((\XX,\dist,\mass)\) be an \(\RCD(K,N)\) space having essential dimension \(n\). Then by a \emph{good collection of splitting maps} on \(\XX\)
we mean a family \(\big\{\boldsymbol u_\eta\,:\,\eta\in(0,n^{-1})\cap\QQ\big\}\) of sequences \(\boldsymbol u_\eta=(u_{\eta,k})_{k\in\NN}\) of
maps \[u_{\eta,k}=(u_{\eta,k}^1,\dots,u_{\eta,k}^n)\colon B_{r_{\eta,k}}(x_{\eta,k})\to\RR^n\] as in Lemma \ref{lemsplitting}. We will denote by
\(D_{\eta,k}\subseteq B_{r_{\eta,k}}(x_{\eta,k})\) the sets associated to \(\boldsymbol u_\eta\) as in Lemma \ref{lemsplitting}. We denote
\[
D_\eta := \bigcup_{k=1}^\infty D_{\eta,k}
\]
and by \(k_\eta(x):D_\eta\rightarrow\NN\) the unique index satisfying \(x\in D_{\eta,k_\eta(x)}\). For every $x\in D_{\eta,k}$ we denote by $A_\eta(x)$ a matrix $A_\eta(x)\in\RR^{n\times n}$ such that, with the same notation of Lemma \ref{lemsplitting}, $A_\eta(x) M_\eta(x) A_\eta(x)^{T}={{\rm Id}_{n\times n}}.$
The existence of such a matrix follows from the choice of $\bar{\eta}_n$. Indeed, from the construction of the symmetric matrix $B_\eta(x)$
it follows that $ \Vert {\rm Id} - M_\eta(x)\Vert_{L^\infty}<n^{-1}$, thus $ \Vert {\rm Id} - M_\eta(x)\Vert_{\rm {op}}<1$ so that the conclusion follows from the spectral theorem.
\end{defn}

Notice that for every $f\in\mathrm{BV}(\XX)$, we have that $|\DIFF f|(\XX\setminus D_\eta)=0$. Let us fix $\eta\in (0,n^{-1})\cap\mathbb Q$. Since for every $x\in D_\eta$ there exists a unique $k_\eta(x)$ such that $x\in D_{\eta,k_\eta(x)}$, and since there exists also a splitting map $u_{\eta,k_\eta(x)}$ on some ball around $x$, one has that the limit 
\[
\lim_{r\to 0}\fint_{B_r(x)}\nabla u^\ell_{\eta,k_{\eta(x)}}\cdot \nabla u^j_{\eta,k_{\eta(x)}}\de \mass,
\]
exists for every $\ell,j\in \{1,\dots,n\}$, compare the end of the proof of \Cref{lemsplitting} and \cite[Remark 2.10]{bru2021constancy}. Hence, for every $\eta\in (0,n^{-1})\cap\mathbb Q$, one can give a pointwise meaning to the $\mathbb R^{n\times n}$-valued map 
\begin{equation}\label{eqn:MAPPAGRADIENTI}
M_\eta:x\in D_\eta\mapsto (\nabla u^\ell_{\eta,k_{\eta(x)}}\cdot \nabla u^j_{\eta,k_{\eta(x)}})_{\ell,j\in \{1,\dots,n\}}(x),
\end{equation}
such that \eqref{eqn:ELEBESGUE} holds.
\section{Main results}\label{sec:Main}
\subsection{Representation formula for the perimeter}\label{sec:Representation}
In this section we prove, by exploiting \cite{deng2020holder} and the same argument of \cite{bru2021constancy}, that the total variation of every $\mathrm{BV}$ function is concentrated on $\mathcal{R}_n^*$. We use the latter information to deduce that perimeter measure of every set of locally finite perimeter is mutually absolutely continuous with respect to $\mathcal{H}^{n-1}$. 
{ We will be using the following theorem, which is proved in \cite[Theorem 1.3]{deng2020holder}.
\begin{thm}\label{deng2020holderthm}
Let $(\XX,\dist,\mass)$ be an $\RCD(K,N)$ space with $K\in\mathbb R$ and $N\geq 1$, and $\mathrm{supp}(\mass)=\XX$. Then $(\XX,\dist,\mass)$ is non-branching, i.e.\ if $\gamma,\sigma:[0,L]\to\XX$ are two unit speed geodesics satisfying $\gamma(0)=\sigma(0)$, and $\gamma(t_0)=\sigma(t_0)$ for some $t_0\in (0,L)$, then $\gamma=\sigma$.
\end{thm}
}
\begin{prop}\label{prop:key_prop_const_dim}
Let \((\XX,\dist,\mass)\) be an \(\RCD(K,N)\) space having essential dimension \(n\). Suppose that \(\gamma\colon[0,1]\to\XX\) is a geodesic
satisfying \(\gamma_t\in\mathcal R_n^*\) for a dense family of \(t\in(0,1)\). Then it holds that \(\gamma_t\in\mathcal R_n^*\) for every \(t\in(0,1)\).
\end{prop}
\begin{proof}
Let \(\delta\in(0,1/20)\) be fixed. { \Cref{deng2020holderthm}
ensures that the constant-speed reparametrisation of \(\gamma|_{[\delta/2,1-\delta/2]}\) on \([0,1]\) is the unique geodesic between its endpoints.
Then \cite[Eq.\ (166)]{deng2020holder}} gives \(\varepsilon=\varepsilon(N,\delta)>0\),
\(\bar r=\bar r(N,\delta)>0\), and \(C=C(N,\delta)>0\) such that
\[
\bigg|\frac{\mass(B_r(\gamma_s))}{\mass(B_r(\gamma_{ s'}))}-1\bigg|\leq C|s-s'|^\frac{1}{2(1+2N)},\quad\text{ for every }r\in(0,\bar r)
\text{ and }s, s'\in[\delta,1-\delta]\text{ with }|s- s'|<\varepsilon.
\]
In particular, for any \(s,s'\in[\delta,1-\delta]\) with \(|s- s'|<\varepsilon\) we have that
\begin{equation}\label{eq:Hold_cont}
\bigg|\frac{\mass(B_r(\gamma_s))}{\omega_n r^n}\bigg(\frac{\mass(B_r(\gamma_{ s'}))}{\omega_n r^n}\bigg)^{-1}-1\bigg|\leq C|s- s'|^{\frac{1}{2(1+2N)}},
\quad\text{ for every }r\in(0,\bar r).
\end{equation}
Now let \(t\in[\delta,1-\delta]\) be fixed and choose a sequence
\((t_i)_{i\in\NN}\subseteq\gamma^{-1}(\mathcal R_n^*)\cap[\delta,1-\delta]\cap(t-\varepsilon,t+\varepsilon)\) such that \(t_i\to t\).
Up to a not relabelled subsequence, we can assume that \(\Theta_n(\mass,\gamma_{t_i})\to\lambda\) for some \(\lambda\in[0,+\infty]\).
Pick sequences \((r_j)_{j\in\NN},(\tilde r_j)_{j\in\NN}\subseteq(0,\bar r)\) such that \(\frac{\mass(B_{r_j}(\gamma_t))}{\omega_n r_j^n}\to
\overline\Theta_n(\mass,\gamma_t)\) and \(\frac{\mass(B_{\tilde r_j}(\gamma_t))}{\omega_n\tilde r_j^n}\to\underline\Theta_n(\mass,\gamma_t)\).
Plugging \((s, s',r)=(t,t_i,r_j)\) or \((s, s',r)=(t,t_i,\tilde r_j)\) in \eqref{eq:Hold_cont}, and letting \(j\to\infty\), we deduce that \(\overline\Theta_n(\mass,\gamma_t)<+\infty\) and
\begin{equation}\label{eq:conseq_Hold_1}
\bigg|\frac{\overline\Theta_n(\mass,\gamma_t)}{\Theta_n(\mass,\gamma_{t_i})}-1\bigg|,
\bigg|\frac{\underline\Theta_n(\mass,\gamma_t)}{\Theta_n(\mass,\gamma_{t_i})}-1\bigg|\leq C|t-t_i|^{\frac{1}{2(1+2N)}},\quad\text{ for every }i\in\NN.
\end{equation}
Similarly, plugging \((s, s',r)=(t_i,t,r_j)\) or \((s, s',r)=(t_i,t,\tilde r_j)\) in \eqref{eq:Hold_cont}, and letting \(j\to\infty\),
we deduce that \(\underline\Theta_n(\mass,\gamma_t)>0\) and
\begin{equation}\label{eq:conseq_Hold_2}
\bigg|\frac{\Theta_n(\mass,\gamma_{t_i})}{\overline\Theta_n(\mass,\gamma_t)}-1\bigg|,
\bigg|\frac{\Theta_n(\mass,\gamma_{t_i})}{\underline\Theta_n(\mass,\gamma_t)}-1\bigg|\leq C|t-t_i|^{\frac{1}{2(1+2N)}},\quad\text{ for every }i\in\NN.
\end{equation}
Observe that \eqref{eq:conseq_Hold_1} and \eqref{eq:conseq_Hold_2} imply, respectively, that for every \(i\in\NN\) it holds that
\begin{subequations}\begin{align}
\label{eq:conseq_Hold_3}
\big|\overline\Theta_n(\mass,\gamma_t)-\underline\Theta_n(\mass,\gamma_t)\big|&\leq 2 C|t-t_i|^{\frac{1}{2(1+2N)}}\Theta_n(\mass,\gamma_{t_i}),\\
\label{eq:conseq_Hold_4}
\big|\overline\Theta_n(\mass,\gamma_t)-\underline\Theta_n(\mass,\gamma_t)\big|&\leq 2 C|t-t_i|^{\frac{1}{2(1+2N)}}
\frac{\overline\Theta_n(\mass,\gamma_t)\underline\Theta_n(\mass,\gamma_t)}{\Theta_n(\mass,\gamma_{t_i})}.
\end{align}\end{subequations}
Hence, we can conclude that \(\overline\Theta_n(\mass,\gamma_t)=\underline\Theta_n(\mass,\gamma_t)\) by letting \(i\to\infty\) in \eqref{eq:conseq_Hold_3}
if \(\lambda<+\infty\), or in \eqref{eq:conseq_Hold_4} if \(\lambda=+\infty\). This shows that \(\gamma_t\in\mathcal R_n^*\) for every \(t\in[\delta,1-\delta]\).
Thanks to the arbitrariness of \(\delta\), we proved that \(\gamma_t\in\mathcal R_n^*\) for every \(t\in(0,1)\), as desired.
\end{proof}
\begin{thm}\label{thm:const_dim_cod1}
Let \((\XX,\dist,\mass)\) be an \(\RCD(K,N)\) space having essential dimension \(n\). Then
\[
|\DIFF f|(\XX\setminus\mathcal R_n^*)=0,\quad\text{ for every }f\in\BV_{\mathrm{loc}}(\XX).
\]
\end{thm}
\begin{proof}
The statement can be achieved by repeating verbatim the proof of \cite[Theorem 3.1]{bru2021constancy}, using \(\mathcal R_n^*\) instead of
\(\mathcal R_n\), and Proposition \ref{prop:key_prop_const_dim} instead of \cite[Proposition 2.14]{bru2021constancy}.
\end{proof}
The following theorem answers in the affirmative to \cite[Conjecture 5.32]{Sem20}.
\begin{thm}[Representation formula for the perimeter]\label{thm:repr_per}
Let \((\XX,\dist,\mass)\) be an \(\RCD(K,N)\) space having essential dimension \(n\). Let \(E\subseteq\XX\) be a set of locally finite perimeter. Then
\begin{equation}\label{eq:repr_per}
|\DIFF\chi_E|=\Theta_n(\mass,\cdot)\HH^{n-1}\llcorner\mathcal F E.
\end{equation}
In particular, it holds that \(\Theta_{n-1}(|\DIFF\chi_E|,x)=\Theta_n(\mass,x)\) for \(\HH^{n-1}\)-a.e. \(x\in\mathcal \mathcal{F} E\).
\end{thm}
\begin{proof}
Up to a standard localization argument, we can suppose that \(E\) is of finite perimeter.
Define \(R_j\coloneqq\big\{x\in\mathcal R_n^*\,:\,2^j\leq\Theta_n(\mass,x)<2^{j+1}\big\}\) for all \(j\in\mathbb Z\). Notice that
\(\{R_j\}_{j\in\mathbb Z}\) is a measurable partition of \(\mathcal R_n^*\). Given \(j\in\mathbb Z\) and \(B\subseteq\XX\) Borel,
for any \(x\in B\cap R_j\cap\mathcal F E\) we have that
\begin{equation}\label{densitcomp}
\exists\Theta_{n-1}(|\DIFF\chi_E|,x)=\frac{\omega_n}{\omega_{n-1}}\lim_{r\searrow 0}\frac{r|\DIFF\chi_E|(B_r(x))}{\mass(B_r(x))}\frac{\mass(B_r(x))}{\omega_n r^n}
=\Theta_n(\mass,x)\in[2^j,2^{j+1}).
\end{equation}
Therefore, an application of \cite[Theorem 2.4.3]{AmbrosioTilli04} yields, for any \(B\subseteq\XX\) Borel,
\begin{equation}\notag
    2^j\HH^{n-1}(B\cap R_j\cap\mathcal F E)\leq|\DIFF\chi_E|(B\cap R_j)\leq 2^{j+n}\HH^{n-1}(B\cap R_j\cap\mathcal F E),
\end{equation}
whence it follows that \(2^j\HH^{n-1}\llcorner(R_j\cap\mathcal F E)\leq|\DIFF\chi_E|\llcorner R_j\leq 2^{j+n}\HH^{n-1}\llcorner(R_j\cap\mathcal F E)\).
Thanks to Theorem \ref{thm:const_dim_cod1}, we deduce that \(\mu_E\coloneqq\HH^{n-1}\llcorner\mathcal F E\) is a \(\sigma\)-finite
Borel measure on \(\XX\) satisfying \(|\DIFF\chi_E|\ll\mu_E\ll|\DIFF\chi_E|\). In particular, we know from \cite[Theorem 4.1]{bru2019rectifiability} that
the set \(\mathcal F E\) is countably \(\HH^{n-1}\)-rectifiable, so that \cite[Theorem 5.4]{AmbKir00} and the computation in \eqref{densitcomp} ensure that
\[
\frac{\diff|\DIFF\chi_E|}{\diff\mu_E}(x)=\lim_{r\searrow 0}\frac{|\DIFF\chi_E|(B_r(x))}{\omega_{n-1}r^{n-1}}
=\Theta_n(\mass,x)
\]
is verified for \(\HH^{n-1}\)-a.e.\ \(x\in\mathcal F E\). Therefore, the identity stated in \eqref{eq:repr_per} is achieved.
\end{proof}
\begin{rem}
Notice that, as a consequence of \cite[Corollary 3.2]{bru2021constancy}, for any set $E$ of locally finite perimeter in an $\RCD(K,N)$ space $(\XX,\dist,\mass)$ of essential dimension $n$, we have that 
\[
|D\chi_E|=\frac{\omega_{n-1}}{\omega_n}\mathcal{H}^h\llcorner\mathcal{F}E.
\]
Hence, taking also \eqref{eq:repr_per} into account, we conclude that the measure $\mathcal{H}^h$ and $\mathcal{H}^{n-1}$ are mutually absolutely continuous on the reduced boundary $\mathcal{F}E$.
\end{rem}

\subsection{Auxiliary results}\label{sec:Auxiliary}
Let \((\XX,\dist,\mass)\) be an \(\RCD(K,N)\) space of essential dimension $n$. Notice that if a given function $u:B_r(\bar{x})\rightarrow\RR$ is harmonic,
then $\nabla u$ admits a quasi-continuous representative in a localization of $\tanXcap$. Also, by tensorization of the energy, if $k\in\NN$, then the function
$$\XX\times\RR^k\supseteq B_{r}(\bar{x})\times\RR^k\ni(x,y)\mapsto u(x)$$ is harmonic, hence it admits a quasi-continuous representative in a localization of
$L^0_\capa(T(\XX\times\RR^k))$ with respect to the relevant capacity. Therefore, the following definition is meaningful:
\begin{defn}\label{real_normal}
Let \((\XX,\dist,\mass)\) be an \(\RCD(K,N)\) space having essential dimension \(n\). Let \(f\in\BV(\XX)\) be given.
Fix a good collection \(\{\boldsymbol u_\eta\}_\eta\) of splitting maps on \(\XX\).
Then, given any \(\eta\in(0,n^{-1})\cap\QQ\), the \(|\DIFF f|\)-measurable map \(\nu_f^{\boldsymbol u_\eta}\colon\XX\to\RR^n\)
is defined at \(|\DIFF f|\)-a.e.\ \(x\in\XX\) as
\[
\nu_f^{\boldsymbol u_\eta}(x)\coloneqq\big((\nu_f\cdot\nabla u^1_{\eta,k_\eta(x)})(x),\dots,(\nu_f\cdot\nabla u^n_{\eta,k_\eta(x)})(x)\big).
\]
The \(|\DIFF\chi_{\GG_f}|\)-measurable map \(\nu_{\GG_f}^{\boldsymbol u_\eta}\colon\XX\times\RR\to\RR^{n+1}\)
is defined at \(|\DIFF\chi_{\GG_f}|\)-a.e.\ \(p=(x,t)\in\XX\times\RR\) as
\[
\nu_{\GG_f}^{\boldsymbol u_\eta}(p)\coloneqq
\big((\nu_{\GG_f}\cdot\nabla u^1_{\eta,k_\eta(x)})(p),\dots,(\nu_{\GG_f}\cdot\nabla u^n_{\eta,k_\eta(x)})(p),(\nu_{\GG_f}\cdot\nabla\pi^2)(p)\big);
\]
notice that $|\DIFF \chi_{\GG_f}|$-a.e.\ $p=(x,t)$ satisfies $x\in D_\eta$, as a consequence of item $i)$ of \Cref{lemsplitting}, \Cref{prop:behaviour_TV_PI} and the existence of functions of locally bounded variation whose total variation equals $\mass$.
\end{defn}
In view of the following proposition, recall the definition of the reduced boundary in use in this note, Definition \ref{def:ReducedBoundary}. In particular, notice that, by definition, $\FF\GG_f\subseteq\mathcal R_{n+1}^*(\XX\times\RR)$, and we will use this inclusion throughout (in particular, recall the properties stated in \Cref{rem:properties_FE}). Notice finally that the matrix valued maps $C_f\ni x\mapsto A_\eta(x)$ in the proposition below are independent of $f$ (up the choice of their domain).
\begin{prop}\label{prop:def_set_C_f}
Let \((\XX,\dist,\mass)\) be an \(\RCD(K,N)\) space having essential dimension \(n\). Let \(f\in\BV(\XX)\) be given.
Let \(\{\boldsymbol u_\eta\}_\eta\) be a good collection of splitting maps on \(\XX\).
Then there exists a Borel set \(C_f\subseteq\XX\) satisfying the following properties:
\begin{itemize}
\item[\(\rm i)\)] \(|\DIFF f|^c=|\DIFF f|\llcorner C_f\) and \(\mass(C_f)=0\).
\item[\(\rm ii)\)] \(C_f\subseteq\mathcal R_n^*(\XX)\setminus J_f\) and \(\mathcal F\GG_f\cap(C_f\times\RR)
=({\rm id}_\XX,\bar f)(C_f)\).
\item[\(\rm iii)\)] Given any \(\eta\in(0,n^{-1})\cap\QQ\) and \(x\in C_f\), for \(A_\eta(x)\in\RR^{n\times n}\) as in Definition \ref{defn:GoodCollection},
\((A_\eta(x) u_{\eta,k(x)},\pi^2)\) is a set of good coordinates for \(\GG_f\) at \((x,\bar f(x))\).
\item[\(\rm iv)\)] If \(u=(u^1,\ldots,u^{n+1})\colon B_{r_x}(x,\bar f(x))\to\RR^{n+1}\) is a system of good coordinates for \(\GG_f\) at \((x,\bar f(x))\)
for some \(x\in C_f\), and the coordinates \((x_\ell)\) on the (Euclidean) tangent space to \(\XX\times\RR\) at \((x,\bar f(x))\) are chosen so that the maps
\((u^\ell)\) converge to \((x_\ell)\colon\RR^{n+1}\to\RR^{n+1}\) (when properly rescaled, see Remark \ref{convdelta}), then the blow-up of \(\GG_f\) at \((x,\bar f(x))\) can be written as
\[
H\coloneqq\big\{y\in\RR^{n+1}\;\big|\;y\cdot\nu(x,\bar f(x))\geq 0\big\},
\]
where the unit vector \(\nu(x,\bar f(x))\coloneqq\big(\nu^1(x,\bar f(x)),\ldots,\nu^{n+1}(x,\bar f(x))\big)\) is given by \eqref{eq:def_nu_in_good_coord}.

\item[\(\rm v)\)] If $p=(x,\bar f(x))\in  C_f\times \mathbb R$, then, for every $\eta\in (0,n^{-1})\cap\mathbb Q$, $x\in D_{\eta,k_\eta(x)}$ for some $k_\eta(x)$, and $p$ is a point of density $1$ of $D_{\eta,k_\eta(x)}\times\RR$ for $|\DIFF\chi_{\GG_f}|$.
\end{itemize}
\end{prop}
\begin{proof}
Let us start this proof by defining several sets whose intersection will define $C_f$. Hence we will define $C_f$ in \eqref{eqn:DefiniamoCf}, and we will verify each item separately.
\smallskip

For every $\eta\in (0,n^{-1})\cap\mathbb Q$, and every $k\in\mathbb N$, take $\mathcal{D}_{\eta,k}$ to be the set of points of density 1 in $(D_{\eta,k}\times\mathbb R)\cap \FF\GG_f$ with respect to $|\DIFF\chi_{\mathcal{G}_f}|$. We thus have that $\bigcup_{k\in\mathbb N} \mathcal{D}_{\eta,k}$ covers $|\DIFF\chi_{\mathcal{G}_f}|$-almost all $D_\eta\times\mathbb R$. Hence, by \Cref{prop:behaviour_TV_PI} and \Cref{lemsplitting}, the set $\pi^1(\bigcup_{k\in\mathbb N}\mathcal{D}_{\eta,k})$ covers $|\DIFF f|$-almost all $\XX$ for every $\eta\in (0,n^{-1})\cap\mathbb Q$. As a consequence, if we denote $\mathcal{D}\coloneqq\bigcap_{\eta\in (0,n^{-1})\cap\mathbb Q}\pi^1(\bigcup_{k\in\mathbb N}\mathcal{D}_{\eta,k})$, then 
\begin{equation}\label{eqn:Set1}
|\DIFF f|(\XX\setminus \mathcal{D})=0.
\end{equation}

Let $\mathcal A\subseteq \XX\times\mathbb R$ be the set of the points $(x,t)\in \XX\times\mathbb R$ such that if $u=(u^1,\dots,u^{n+1}):B_{r_{(x,t)}}(x,t)\to\mathbb R^{n+1}$ is a system of good coordinates for $\mathcal{G}_f$ at $(x,t)$, and the coordinates \((x_\ell)\) on the (Euclidean) tangent space to \(\XX\times\RR\) at \((x,t)\) are chosen so that the maps
\((u^\ell)\) converge to \((x_\ell)\colon\RR^{n+1}\to\RR^{n+1}\) (when properly rescaled), then the blow-up of \(\GG_f\) at \((x,t)\) can be written as
\[
\big\{y\in\RR^{n+1}\;\big|\;y\cdot\nu(x,t)\geq 0\big\},
\]
where the unit vector \(\nu(x,t)\coloneqq\big(\nu^1(x,t),\ldots,\nu^{n+1}(x,t)\big)\) is given by \eqref{eq:def_nu_in_good_coord}. Then, by \cref{prop:conv_good_coord}, we have also that 
\begin{equation}\label{eqn:Set3}
|\DIFF\chi_{\mathcal{G}_f}|\big((\XX\times\mathbb R)\setminus \mathcal A\big)=0.
\end{equation}

Let $\eta\in (0,n^{-1})\cap\mathbb Q$, and let $\mathcal{T}_\eta$ be the Lebesgue points of $\nu_{\GG_f}^{\boldsymbol u_\eta}$ (defined in \Cref{real_normal}) with respect to $|\DIFF \chi_{\mathcal{G}_f}|$. Let $\mathcal{T}\coloneqq\bigcap_{\eta\in (0,n^{-1})\cap\mathbb Q}\mathcal{T}_\eta$, and notice that \begin{equation}\label{eqn:Set4}
|\DIFF \chi_{\mathcal{G}_f}|\big((\XX\times\mathbb R)\setminus\mathcal{T}\big)=0.
\end{equation}

Let us fix $\eta\in (0,n^{-1})\cap\mathbb Q$, and $k\in\mathbb N$. Let $\widetilde M_\eta:=(M_\eta,\pi^2)$ be defined on $\XX\times\mathbb R$, where $M_\eta$ is defined in \eqref{eqn:MAPPAGRADIENTI}. Notice that $\widetilde M_\eta$ is $|\DIFF \chi_{\mathcal{G}_f}|$-measurable. Let $\mathcal{S}_\eta$ be the Lebesgue points of $\widetilde M_\eta$ with respect to $|\DIFF \chi_{\mathcal{G}_f}|$, and let $\mathcal{S}\coloneqq\bigcap_{\eta\in (0,n^{-1})}\mathcal{S}_\eta$. Notice that 
\begin{equation}\label{eqn:Set5}
|\DIFF \chi_{\mathcal{G}_f}|\big((\XX\times\mathbb R)\setminus\mathcal{S}\big)=0.
\end{equation}

Let $S\subseteq\XX_f$ with $\mass(S)=0$ be such that $|\DIFF f|^s$ is concentrated on $S$ (recall \eqref{eq:bar_f_finite}). Let us now define
\begin{equation}\label{eqn:DefiniamoCf}
C_f\coloneqq S\cap\left(\mathcal{R}_n^*(\XX)\setminus J_f\right)\cap\bigg(\bigcap_{\eta\in(0,n^{-1})\cap\mathbb Q}D_\eta\bigg)
\cap \pi^1(\mathcal A\cap \mathcal{T}\cap\mathcal{S}\cap\mathcal{F}\mathcal{G}_f)\cap\mathcal{D},
\end{equation}
where $D_\eta$ is defined in \Cref{defn:GoodCollection}, $J_f$ is the jump set of $f$, $\mathcal{F}\mathcal{G}_f$ is the reduced boundary of $\mathcal{G}_f$, and $\mathcal A,\mathcal{T},\mathcal{S}$ are defined above. Let us verify each item separately. 
\smallskip

\textbf{Item (i)}. Notice that $|\DIFF f|^c$ is concentrated on $S$. Moreover, $|\DIFF f|^c$ is concentrated on $\XX\setminus J_f$, and due to \Cref{lemsplitting}, $|\DIFF f|^c$ is concentrated on $\bigcap_{\eta\in (0,n^{-1})\cap\mathbb Q}D_\eta$ as well. Due to \eqref{eqn:Set1}, $|\DIFF f|$ is concentrated on $\mathcal{D}$. Furthermore, $|\DIFF \chi_{\mathcal{G}_f}|$ is concentrated on $\mathcal A\cap\mathcal{T}\cap\mathcal{S}\cap\mathcal{F}\mathcal{G}_f$ due \eqref{eqn:Set3}, \eqref{eqn:Set4}, \eqref{eqn:Set5}, and to the definition of reduced boundary, see \Cref{def:ReducedBoundary}. Thus, due to \Cref{prop:behaviour_TV_PI}, $|\DIFF f|$ is concentrated on $\pi^1(\mathcal A\cap\mathcal{T}\cap\mathcal{S}\cap\mathcal{F}\mathcal{G}_f)$. Putting all together, we get that $|\DIFF f|^c$ is concentrated on $C_f$.
\smallskip

\textbf{Item (ii)}. { By \Cref{lem:char_G_f}, one has that if $x\in C_f\setminus J_f$, then
$\mathcal{F}\mathcal{G}_f\cap(\{x\}\times \mathbb R)=\{(x,\bar f(x))\}$. Indeed, $x\in C_f\subseteq \pi^1(\mathcal{F}\mathcal{G}_f)$, and then $\mathcal{F}\mathcal{G}_f\cap(\{x\}\times \mathbb R)$ is nonempty}. Hence $\mathcal{F}\mathcal{G}_f\cap(C_f\times\mathbb R)=(\mathrm{id}_\XX,\bar f)(C_f)$. 
\smallskip

\textbf{Item (iii)}. Let $x\in C_f$. Hence, by Item (ii) and by definition of $C_f$, we have that $x=\pi^1(x,\bar f(x))$, and $(x,\bar f(x))\in\mathcal{T}\cap\mathcal{S}$.

Let $\eta\in (0,n^{-1})\cap\mathbb Q$. We have that there exists $k_\eta(x)$ such that $x\in D_{\eta,k_\eta(x)}$. By Item iii) of \Cref{lemsplitting}, compare with \eqref{eqn:MAPPAGRADIENTI}, we get the existence of a matrix $M(x)\in\mathbb R^{n\times n}$ such that, for every $\ell,j=1,\dots,n$,
$$
\lim_{r\searrow 0}\dashint_{B_r(x)}\big|\nabla u^\ell_{\eta,k_\eta(x)}\,\cdot\,\nabla u^j_{\eta,k_\eta(x)}-M(x)_{\ell,j}\big|\dd{\mass}=0.
$$
Hence, taking $A_\eta(x)$ to be the matrix as in the end of \Cref{defn:GoodCollection}, we conclude that, calling $v_{\eta,k_\eta(x)}\coloneqq A_\eta(x)u_{\eta,k_\eta(x)}$, we have, for every $\ell,j=1,\dots,n$,
$$
\lim_{r\searrow 0}\dashint_{B_r(x)} |\nabla v^\ell_{\eta,k_\eta(x)}\,\cdot\,\nabla v^j_{\eta,k_\eta(x)}-\delta_{\ell j}|\dd{\mass}=0.
$$
Hence, as a consequence of the previous equality, the independence of the coordinates in $\XX\times\mathbb R$, and Fubini theorem, calling $\widetilde v_{\eta,k_\eta(x)}:=(v_{\eta,k_\eta(x)},\pi^2)$, we get that the following holds for every $\ell,j=1,\dots,n+1$,
\begin{equation}\label{eqn:Good1}
\lim_{r\searrow 0}\dashint_{B_r(x,\bar f(x))} |\nabla \widetilde v^\ell_{\eta,k_\eta(x)}\,\cdot\,\nabla \widetilde v^j_{\eta,k_\eta(x)}-\delta_{\ell j}|\dd(\mass\otimes\mathcal{H}^1)=0.
\end{equation}
Now, since $(x,\bar f(x))\in\mathcal{S}_\eta$, and $\mathcal{S}_\eta$ is the set of the Lebesgue points of $(M_\eta,\pi^2)$ (see \eqref{eqn:MAPPAGRADIENTI}) with respect to $|\DIFF \chi_{\mathcal{G}_f}|$, we get that also, for every $\ell,j=1,\dots,n+1$, the following holds
\begin{equation}\label{eqn:Good2}
\lim_{r\searrow 0}\dashint_{B_r(x,\bar f(x))} |\nabla \widetilde v^\ell_{\eta,k_\eta(x)}\,\cdot\,\nabla \widetilde v^j_{\eta,k_\eta(x)}-\delta_{\ell j}|\dd|\DIFF \chi_{\mathcal{G}_f}|=0.
\end{equation}
Finally, notice that $(x,\bar f(x))\in \mathcal{T}_\eta$, and $\mathcal{T}_\eta$ are the Lebesgue points of $\nu_{\GG_f}^{\boldsymbol u_\eta}$ with respect to $|\DIFF \chi_{\mathcal{G}_f}|$. Hence, $(x,\bar f(x))$ is also a Lebesgue point of the $|\DIFF \chi_{\mathcal{G}_f}|$-measurable map defined for $p=(y,t)$
\begin{equation}\label{eqn:NuAumentata}
\widetilde\nu_{\GG_f}^{\boldsymbol u_\eta}(p)\coloneqq
\big((\nu_{\GG_f}\cdot\nabla A_\eta(x)u^1_{\eta,k_\eta(x)})(p),\dots,(\nu_{\GG_f}\cdot\nabla A_\eta(x)u^n_{\eta,k_\eta(x)})(p),(\nu_{\GG_f}\cdot\nabla\pi^2)(p)\big).
\end{equation}
Arguing as in the last part of \cite[Proposition 3.6]{bru2021constancy}, we get that the norm of the $|\DIFF \chi_{\mathcal{G}_f}|$-Lebesgue representative of $\widetilde\nu_{\GG_f}^{\boldsymbol u_\eta}$ at $(x,\bar f(x))$ is 1. Hence the last information, together with \eqref{eqn:Good1}, and \eqref{eqn:Good2}, give that $\widetilde v_{\eta,k_\eta(x)}$ is a set of good coordinates for $\mathcal{G}_f$ at $(x,\bar f(x))$.
\smallskip

\textbf{Item (iv)}. It follows from Item (ii) and the definition of $\mathcal A$.
\smallskip

\textbf{Item (v)}. It follows from Item (ii) and the definition of $\mathcal{D}$.
\end{proof}

\begin{thm}\label{thmcantor}
Let $(\XX,\dist,\mass)$ be an $\RCD(K,N)$ space having essential dimension \(n\). Fix a function $f\in\BV(\XX)$ and a good collection
\(\{\boldsymbol u_\eta\}_\eta\) of splitting maps on \(\XX\). Let \(C_f\subseteq\XX\) be as in Proposition \ref{prop:def_set_C_f}. Then
for any given \(\eta\in(0,n^{-1})\cap\QQ\) it holds that
$$(\nu_{\GG_f}^{\boldsymbol u_\eta})_{n+1}(p)= 0,\quad\text{for }\HH^n\text{-a.e.\ }p\in\mathcal F\GG_f\cap(C_f\times\RR).$$
\end{thm}
\begin{proof}
We recall from Proposition \ref{prop:behaviour_TV_PI} that \(\pi^1_*\big(|\DIFF\chi_{\GG_f}|\llcorner(\mathcal F\GG_f\cap(C_f\times\RR))\big)=|\DIFF f|^c\).
Moreover, Lemma \ref{lem:char_G_f} ensures that the mapping \(\pi^1\colon\mathcal F\GG_f\cap(C_f\times\RR)\to C_f\) is the inverse of
\(({\rm id}_\XX,\bar f)\colon C_f\to\mathcal F\GG_f\cap(C_f\times\RR)\). Given any \(k,j\in\NN\) and \(\alpha\in(0,1)\cap\QQ\), we define
\[
C_f^{k,\alpha,j}\coloneqq\Big\{x\in C_f\cap D_{\eta,k}\;\Big|\;\big|(\nu_{\GG_f}^{\boldsymbol u_\eta})_{n+1}(x,\bar f(x))\big|\geq\alpha,
\,j^{-1}\leq\Theta_{n+1}\big(\mass\otimes\mathcal L^1,(x,\bar f(x))\big)\leq j\Big\}.
\]
{Notice that the sets $C_f^{k,\alpha,j}$ obviously depend on $\eta$ but, as we are working with a fixed $\eta\in (0,n^{-1})\cap\QQ$, we do not make this dependence explicit.}
Recalling Theorem \ref{thm:const_dim_cod1} we see that
\[
\big\{x\in C_f\;\big|\;(\nu_{\GG_f}^{\boldsymbol u_\eta})_{n+1}(x,\bar f(x))\neq 0\big\}=\bigcup_{k,\alpha,j}C_f^{k,\alpha,j},
\quad\text{ up to }|\DIFF f|\text{-null sets.}
\]
Hence, to prove the statement amounts to showing that each set \(\mathcal F\GG_f\cap(C_f^{k,\alpha,j}\times\RR)\) is \(\HH^n\)-negligible.
Given any \(\varepsilon>0\), by Lusin's Theorem we can find \(\Sigma\subseteq\mathcal F\GG_f\cap(C_f^{k,\alpha,j}\times\RR)\) Borel such that
\(\bar f\) is continuous on \(\pi^1(\Sigma)\) and \(\HH^n\big((\mathcal F\GG_f\cap(C_f^{k,\alpha,j}\times\RR))\setminus\Sigma\big)<\varepsilon\).

Our aim is to show that
\begin{equation}\label{thmcantor_claim1}
\HH^n(\Sigma)=0,
\end{equation}
since this would imply \(\HH^n\big(\mathcal F\GG_f\cap(C_f^{k,\alpha,j}\times\RR)\big)=0\) by the arbitrariness of \(\varepsilon>0\).
Up to discarding an \(\HH^n\)-null set from \(\Sigma\), we can also assume (thanks to Remark \ref{rmk:per_asympt_doubl} and Theorem \ref{thm:repr_per})
that \(\Theta_n(|\DIFF\chi_{\GG_f}|\llcorner\Sigma,p)=\Theta_{n+1}(\mass\otimes\mathcal L^1,p)\) holds for every \(p\in\Sigma\). Now we claim that
\begin{equation}\label{thmcantor_claim2}
\lim_{r\searrow 0}\frac{|\DIFF\chi_{\GG_f}|\big((\Sigma\cap B_r(p))\setminus(\XX\times B_{\beta r}(t))\big)}{r^n}=0,\quad\text{ for every }p=(x,t)\in\Sigma,
\end{equation}
where we set \(\beta=\beta(\alpha)\coloneqq\sqrt{1-\alpha^2}\in(0,1)\). { The role played by $\alpha$ will be made clear in  what follows}. To show the claim, fix \(p=(x,t)\in\Sigma\) and take any sequence
\(\{r_i\}_i\subseteq(0,+\infty)\) with \(r_i\searrow 0\). Since \(x\in\mathcal R_n(\XX)\), one has that
\[
(\XX,r_i^{-1}\dist,\mass_x^{r_i},x)\to(\RR^n,\dist_e,\underline{\mathcal L}^n,0),\quad\text{ in the pmGH topology.}
\]
Let \((\ZZ,\dist_\XX)\) be a realization of such convergence. Then \((\ZZ\times\RR,\dist_\ZZ\times\dist_e)\) is a realization of
\[
\big(\XX\times\RR,r_i^{-1}(\dist\times\dist_e),(\mass\otimes\mathcal L^1)_p^{r_i},p,\GG_f\big)\to(\RR^{n+1},\dist_e,\underline{\mathcal L}^{n+1},0,H),
\]
where \(H\subseteq\RR^{n+1}\) is a halfspace. We also know from item v) of Proposition \ref{prop:def_set_C_f} that, up to passing to a not relabelled subsequence,
the rescaled perimeters \(|\DIFF\chi_{\GG_f}|\) weakly converge to \(\HH^n\llcorner\partial H\) in duality with \(C_{bs}(\ZZ)\).
Moreover, by item iv) of Proposition \ref{prop:def_set_C_f}, \(\partial H\) is normal to \(\nu_{\GG_f}^{\boldsymbol u_\eta}(p)\). { Thus, since $\big(\nu_{\GG_f}^{\boldsymbol u_\eta}\big)_{n+1}(p)\ge \alpha$, we have \(\partial H\cap B_1(0)\subseteq B_1(0)\times B_\beta(0)\), by our choice of $\beta$.}
From the latter the claim \eqref{thmcantor_claim2} follows, taking into account also \eqref{usefullimits}.
For \(\gamma\in(0,+\infty)\) and \((x,t)\in\XX\times\RR\), we denote the cone
\[
C_\gamma(x,t)\coloneqq\big\{(y,s)\in\XX\times\RR\;\big|\;\gamma\,\dist(y,x)\geq|s-t|\big\}.
\]
Now take \(\gamma=\gamma(\beta)=\sqrt{\frac{1+\beta}{1-\beta}}\in(1,+\infty)\). Notice that \(\gamma^2>\frac{\beta}{1-\beta}\). Next we claim that
\begin{equation}\label{thmcantor_claim3}
\lim_{r\searrow 0}\frac{|\DIFF\chi_{\GG_f}|\big((\Sigma\cap B_r(p))\setminus C_\gamma(p)\big)}{r^n}=0,\quad\text{ for every }p=(x,t)\in\Sigma.
\end{equation}
In order to prove it, fix \(\delta>0\). By virtue of \eqref{thmcantor_claim2}, we can take \(r_0>0\) small enough so that
\begin{equation}\label{thmcantor_aux1}
\sup_{r\in(0,r_0)} \frac{|\DIFF\chi_{\GG_f}|\big((\Sigma\cap B_r(p))\setminus(\XX\times B_{\beta r}(t))\big)}{r^n}\leq\delta.
\end{equation}
Notice that it holds that
\begin{equation}\label{thmcantor_aux2}
B_{r_0}(p)\setminus C_\gamma(p)\subseteq\bigcup_i B_{r_i}(p)\setminus(\XX\times B_{\beta r_i}(t)),
\end{equation}
where for any $i\in\NN$ with $i\ge 1$ we define $r_i\defeq\beta\sqrt{\frac{\gamma ^2+1}{\gamma^2}}r_{i-1}=\bigg(\beta\sqrt{\frac{\gamma ^2+1}{\gamma^2}}\bigg)^i r_0$. Given that
$$
|\DIFF\chi_{\GG_f}|\big((\Sigma\cap B_{r_i}(p))\setminus(\XX\times B_{\beta r_i}(p))\big)\overset{\eqref{thmcantor_aux1}}\le \delta r_i^n=
\delta\Bigg(\beta\sqrt{\frac{\gamma^2 +1}{\gamma^2}}\Bigg)^{n i} r_0^n,
$$
it follows from the inclusion in \eqref{thmcantor_aux2} that
$$ 
\frac{|\DIFF\chi_{\GG_f}|\big((\Sigma\cap B_{r_0}(p))\setminus C_\gamma(p)\big)}{r_0^n}\le \delta\sum_i \left(\beta\sqrt{\frac{\gamma^2 +1}{\gamma^2}}\right)^{n i}.
$$
Thanks to the arbitrariness of \(\delta>0\) and the finiteness of \(\sum_i\big(\beta\sqrt{\frac{\gamma^2+1}{\gamma^2}}\big)^{ni}\),
\eqref{thmcantor_claim3} is proved.

Let now $\epsilon'>0$. We wish to show that there exists a set $\Sigma'\subseteq\Sigma$ with $\HH^n(\Sigma\setminus \Sigma')<\epsilon'$
and such that there exists $r_0\in (0,1)$ satisfying
\begin{equation}\label{simonclaim}
(\Sigma'\cap B_{r_0}(p))\setminus C_{2 \gamma}(p)=\emptyset,\quad\text{ for every }p\in\Sigma'.
\end{equation}
We do it using a standard argument, see e.g.\ the proof of \cite[Theorem 1.6]{Simon}. By Egorov's Theorem, we can choose $\Sigma'\subseteq\Sigma$ Borel with
$\HH^n(\Sigma\setminus \Sigma')<\epsilon'$ and such that for any given $\delta'>0$, there exists $r_0\in(0,1)$ such that for every $r\in(0,2 r_0)$ and $p\in\Sigma'$,
\begin{subequations}\begin{align}\label{thmcantor_aux3}
\frac{|\DIFF\chi_{\GG_f}|(\Sigma \cap B_r(p))}{\Theta_{n+1}(\mass\otimes\mathcal L^1,p)\omega_n r^n}&\ge 1-\delta',\\
\label{thmcantor_aux4}
\frac{|\DIFF\chi_{\GG_f}|\big((\Sigma\cap B_r(p))\setminus C_\gamma(p)\big)}{\Theta_{n+1}(\mass\otimes\mathcal L^1,p)\omega_n r^n}&\le\delta';
\end{align}\end{subequations}
the former follows from the fact that \(\Theta_n(|\DIFF\chi_{\GG_f}|\llcorner\Sigma,p)=\Theta_{n+1}(\mass\otimes\mathcal L^1,p)\), the latter from
\eqref{thmcantor_claim3}. We aim to show that if $\delta'>0$ is small enough, then this choice of $\Sigma'$ and $r_0$ satisfies \eqref{simonclaim}.
Assume now that there exists $q\in(\Sigma'\cap B_{r_0}(p))\setminus C_{2 \gamma}(p)$ for some $p\in\Sigma'$. Then
\begin{equation}\label{thmcantor_aux5}
B_{\rho}(q)\subseteq  B_{\tilde\dist(p,q)+\rho}(p)\setminus C_{ \gamma}(p),\quad\text{ where }
\rho\defeq\tilde\dist(p,q)\sin(\arctan(2\gamma)-\arctan(\gamma)),
\end{equation}
where we denoted \(\tilde\dist\coloneqq\dist\times\dist_e\) for brevity. Therefore, we can estimate
\begin{align*}
	\delta'&\overset{\eqref{thmcantor_aux4}}\ge\frac{|\DIFF\chi_{\GG_f}|\big((\Sigma\cap B_{\tilde\dist(p,q)+\rho}(p))\setminus C_\gamma(p)\big)}
	{\Theta_{n+1}(\mass\otimes\mathcal L^1,p)\omega_n(\tilde\dist(p,q)+\rho)^n} \overset{\eqref{thmcantor_aux5}}\ge
	\frac{|\DIFF\chi_{\GG_f}|(\Sigma\cap B_\rho(q))}{\Theta_{n+1}(\mass\otimes\mathcal L^1,p)\omega_n(\tilde\dist(p,q)+\rho)^n}\\
	& \overset{\eqref{thmcantor_aux4}}\ge (1-\delta')\frac{\rho^n}{(\tilde\dist(p,q)+\rho)^n} 
	=(1-\delta')\frac{\big(\sin(\arctan(2\gamma)-\arctan(\gamma))\big)^n}{\big(1+\sin(\arctan(2\gamma)-\arctan(\gamma))\big)^n},
\end{align*}
which leads to a contradiction provided $\delta'>0$ was chosen small enough, proving \eqref{simonclaim}.

Finally, our aim is to show that
\begin{equation}\label{thmcantor_claim4}
|\DIFF\chi_{\GG_f}|(\Sigma')=0,
\end{equation}
since this, by the arbitrariness of \(\varepsilon'>0\), would imply \eqref{thmcantor_claim2} and accordingly the statement.
Take $p=(x,t)\in\Sigma'$. Since $\bar{f}$ is continuous on $\pi^1(\Sigma')$, there exists $r_1\in (0,r_0/\sqrt 2)$ such that
$|{\bar{f}(y)-\bar{f}(x)}|<r_0/\sqrt 2$ for all $y\in B_{r_1}(x)\cap\pi^1(\Sigma')$. As $\Sigma'\subseteq\{(x,t)\in\XX\times\RR\,:\,t=\bar{f}(x)\}$,
we see that $\Sigma'\cap(B_{r_1}(x)\times\RR)\subseteq\Sigma'\cap B_{r_0}(p)\subseteq  C_{2\gamma}(p)$ by \eqref{simonclaim}, so that,
setting \(\lambda\coloneqq\sqrt{1+4\gamma^2}\),
\begin{equation}\label{thmcantor_aux6}
\Sigma'\cap(B_r(x)\times\RR)\subseteq \Sigma'\cap B_{\lambda r}(p),\quad\text{ for every }r\in(0,r_1).
\end{equation}
It follows that for every \(p=(x,t)\in\Sigma'\) we have
\[\begin{split}
\overline\Theta_n\big(\pi^1_*(|\DIFF\chi_{\GG_f}|\llcorner\Sigma'),x\big)&=
\varlimsup_{r\searrow 0}\frac{|\DIFF\chi_{\GG_f}|\big(\Sigma'\cap(B_r(x)\times\RR)\big)}{\omega_n r^n}
\overset{\eqref{thmcantor_aux6}}\leq\varlimsup_{r\searrow 0}\frac{|\DIFF\chi_{\GG_f}|(\Sigma\cap B_{\lambda r}(p))}{\omega_n r^n}\\
&=\lambda^n\Theta_n(|\DIFF\chi_{\GG_f}|\llcorner\Sigma,p)=\lambda^n\Theta_{n+1}(\mass\otimes\mathcal L^1,p)\leq\lambda^n j,
\end{split}\]
where the last inequality stems from the inclusion \(\Sigma'\subseteq\mathcal F\GG_f\cap(C_f^{k,\alpha,j}\times\RR)\).
Therefore, by applying \cite[Theorem 2.4.3]{AmbrosioTilli04} and using the fact that \(\pi^1(\Sigma')\subseteq C_f\), we can conclude that
\[
|\DIFF\chi_{\GG_f}|(\Sigma')=\pi^1_*(|\DIFF\chi_{\GG_f}|\llcorner\Sigma')(\pi^1(\Sigma'))\leq(2\lambda)^n j\HH^n(\pi^1(\Sigma'))
\leq(2\lambda)^n j\HH^n(C_f)=0,
\]
thus obtaining \eqref{thmcantor_claim4}. Consequently, the statement is achieved.
\end{proof}
\begin{lem}\label{coincide}
Let $(\XX,\dist,\mass)$ be an $\RCD(K,N)$ space of essential dimension $n$. Fix a function $f\in\BVv$ and a good collection \(\{\boldsymbol u_\eta\}_\eta\)
of splitting maps on \(\XX\). Let $C_f$ be as in the statement of Proposition \ref{prop:def_set_C_f}. Then for any \(\eta\in(0,n^{-1})\cap\QQ\) it holds that
$$ \nu_f^{\boldsymbol u_\eta}(x)=\big(\nu_{\GG_f}^{\boldsymbol u_\eta}(x,\bar{f}(x))\big)_{1,\dots,n},	\qquad\text{for }\abs{\DIFF f}\mres C_f\text{-a.e.\ $x\in\XX$}.$$
\end{lem}
\begin{proof}
Recall that $\abs{\DIFF f}\mres C_f=\pi^1_*\big(|{\DIFF\chi_{\GG_f}}|\mres(C_f\times\RR)\big)$, so that the statement makes sense. By the coarea formula,
it is enough to show that for a.e.\ $t$ it holds $\nu_f^{\boldsymbol u_\eta}(x)=\big(\nu_{\GG_f}^{\boldsymbol u_\eta}(x,\bar{f}(x))\big)_{1,\dots,n}$ for
$\HH^{n-1}$-a.e.\ $x\in\mathcal F E_t\cap C_f$, where we define $E_t\defeq\{f>t\}$. Taking \cite[
Lemma 3.27]{BGBV}
into account, we see that it is sufficient to prove that for a.e.\ \(t\), and for every \(k\in\NN\) it holds
\begin{equation}\label{eq:coincide_claim}
\nu_{\chi_{E_t}}^{\boldsymbol u_\eta}(x)=\big(\nu_{\GG_f}^{\boldsymbol u_\eta}(x,\bar{f}(x))\big)_{1,\dots,n},
\quad\text{ for }\mathcal H^{n-1}\text{-a.e.\ }x\in\mathcal F E_t\cap C_f\cap D_{\eta,k}.
\end{equation}
Let $x\in \mathcal F E_t\cap C_f\cap D_{\eta,k}$ be a given point where the conclusions of Proposition \ref{prop:conv_good_coord}
hold with \(E=E_t\); notice that \(\mathcal H^{n-1}\)-a.e.\ point of \(\mathcal F E_t\cap C_f\cap D_{\eta,k}\) has this property.
We aim to show that the identity in \eqref{eq:coincide_claim} is verified at \(x\). Denote $p\defeq (x,\bar{f}(x))$ for brevity. Thanks to item
(i) of Remark \ref{rem:properties_FE} and item (v) of Proposition \ref{prop:def_set_C_f}, we can find a sequence \(r_i\searrow 0\),
halfspaces \(H\subseteq\RR^{n+1}\) and \(H'\subseteq\RR^n\), and a proper metric space \((\ZZ,\dist_\ZZ)\) such that
\begin{subequations}\begin{align}\label{eq:conv_halfs_1}
\big(\XX,r_i^{-1}\dist,\mass_x^{r_i},x,E_t\big)&\rightarrow(\RR^n,\dist_e,\underline{\mathcal L}^n,0,H'),\\ \label{eq:conv_halfs_2}
\big(\XX\times\RR,r_i^{-1}\dist_{\XX\times\RR},(\mass\otimes\HH^1)_p^{r_i},p,\GG_f\big)&\rightarrow(\RR^{n+1},\dist_e,\underline{\mathcal L}^{n+1},0,H),
\end{align}\end{subequations}
in the realizations \(\ZZ\) and \(\ZZ\times\RR\), respectively. Notice also that
\begin{equation}\label{eq:conv_halfs_3}
\big\{(y,s)\in\XX\times\RR\;\big|\;s<t\big\}\rightarrow\big\{(y,s)\in\RR^n\times\RR\;\big|\;s<0\big\}\qquad\text{in }L^1_{\mathrm{loc}},
\end{equation}
in the realization \(\ZZ\times\RR\). Therefore, by stability, we deduce from \eqref{eq:conv_halfs_2} and \eqref{eq:conv_halfs_3} that
\[
\big\{(y,s)\in\XX\times\RR\;\big|\;s<f(y),s<t\big\}\rightarrow H\cap\big\{(y,s)\in\RR^n\times\RR\;\big|\;s<0\big\} \qquad\text{in }L^1_{\mathrm{loc}}.
\]
Recalling \eqref{eq:conv_halfs_1}, and using Fubini's Theorem and dominated convergence, we see that 
\[
E_t\times (-\infty,t)\rightarrow H' \times(-\infty,0) \qquad\text{in $L^1_{\mathrm{loc}}$}.
\]
Given that \(E_t\times(-\infty,t)\subseteq\big\{(y,s)\in\XX\times\RR\,:\,s<f(y),s<t\big\}\), we obtain that 
\[
H'\times(-\infty,0)\subseteq H\cap\big\{(y,s)\in\RR^n\times\RR\;\big|\;s<0\big\}.
\]
Thanks to our choice of \(x\) and to item (iv) and item (v) of Proposition \ref{prop:def_set_C_f}, we can see that
 \(\nu_{\chi_{E_t}}^{\boldsymbol u_\eta}(x)\) and \((\nu_{\GG_f}^{\boldsymbol u_\eta}(p))_{1,\dots,n}\) have the same direction, namely there exists $\lambda(x)\in[0,1]$ such that \(\nu_{\chi_{E_t}}^{\boldsymbol u_\eta}(x)=\lambda(x)(\nu_{\GG_f}^{\boldsymbol u_\eta}(p))_{1,\dots,n}\). Now notice that the conclusion of Theorem \ref{thmcantor} forces $\lambda(x)=1$, up to discarding a $|\DIFF f|\mres C_f$-negligible set.
\end{proof}
\subsection{Rank-One Theorem}\label{sec:Final}
In this final subsection we prove \Cref{thm:RankOne}.
{We first start with an auxiliary definition and a technical result taken from \cite{bru2019rectifiability}.
\medskip

Let \((\XX,\dist,\mass)\) be an \(\RCD(K,N)\) space of essential dimension \(n\) and \(E\subseteq\XX\) a set of
locally finite perimeter.
Let \(\varepsilon>0\) and \(r>0\) be given. Then, following \cite[Definition 4.6]{bru2019rectifiability}, we define
\((\mathcal F_n E)_{r,\varepsilon}\) as the set of all those points \(x\in\mathcal F_n E\)
such that
\[\begin{split}
&\dist_{\rm pmGH}\big((\XX,s^{-1}\dist,\mass_x^s,x),(\RR^n,\dist_e,\underline{\mathcal L}^n,0)\big)<\varepsilon,\\
&\bigg|\frac{\mass(E\cap B_s(x))}{\mass(B_s(x)}-\frac{1}{2}\bigg|+
\bigg|\frac{s|\DIFF\chi_E|(B_s(x))}{\mass(B_s(x))}-\frac{\omega_{n-1}}{\omega_n}\bigg|<\varepsilon
\end{split}\]
for every \(s\in(0,r)\). We remark that, for every $x\in \FF_n E$, and for every $\epsilon>0$, there exists $r>0$ such that $x\in (\FF_n E)_{r,\epsilon}$.
We now recall the following result, which was proved in \cite[Proposition 4.7]{bru2019rectifiability}.
\begin{prop}\label{prop:from_delta-splitt_to-biLip}
Let \((\XX,\dist,\mass)\) be an \(\RCD(K,N)\) space of essential dimension \(n\). Let \(E\subseteq\XX\) be a set
of locally finite perimeter. Then for any \(\eta>0\) there exists \(\varepsilon=\varepsilon(N,\eta)>0\) such that
the following property is verified: if \(p\in(\mathcal F_n E)_{2r,\varepsilon}\) for some \(0<r<|K|^{-1/2}\) and
there exists an \(\varepsilon\)-splitting map \(u\colon B_{2r}(p)\to\RR^{n-1}\) such that
\[
\frac{r}{\mass(B_{2r}(p))}\int_{B_{2r}(p)}|\nu_E\cdot\nabla u^\ell|\,\dd|\DIFF\chi_E|<\varepsilon,
\quad\text{ for every }\ell=1,\ldots,n-1,
\]
then there exists a Borel set \(G\subseteq B_r(p)\) with \(\HH^h_5(B_r(p)\setminus G)\leq C_N\eta\frac{\mass(B_r(p))}{r}\)
such that
\[
u\colon G\cap(\mathcal F_n E)_{2r,\varepsilon}\to\RR^{n-1}\quad\text{ is biLipschitz onto its image.}
\]
\end{prop}
}

We pass to the following Lemma, which is the technical core of the proof of \Cref{thm:RankOne}.
\begin{lem}\label{mainlemma}
Let $(\XX,\dist,\mass)$ be an $\RCD(K,N)$ space of essential dimension $n$. Fix any two functions $f,g\in\BVv$.
Let \(\{\boldsymbol u_\eta\}_\eta\) be a good collection of splitting maps on \(\XX\). Let us consider the sets
\(C_f,C_g\subseteq\XX\) given by Proposition \ref{prop:def_set_C_f}. {Let $\tau$ be the inversion map defined in \eqref{eqn:Tau}, and let}
\begin{equation*}
\begin{aligned}
	\Sigma_f&\defeq\mathcal F\GG_f\cap (C_f\times \RR),\quad&&\tilde{\Sigma}_f\defeq\Sigma_f\times\RR,\\
		\Sigma_g&\defeq\mathcal F\GG_g\cap (C_g\times \RR),&&\tilde{\Sigma}_g\defeq\tau(\Sigma_g\times\RR).
\end{aligned}
\end{equation*}
Moreover, let us set \(R\coloneqq\pi^1(\tilde R)\subseteq\XX\), where the set \(\tilde R\subseteq\XX\times\RR^2\) is defined as
\begin{equation}\label{eqn:SigmaFSigmag}
\bigcap_{\substack{\eta\in\QQ, \\ 0<\eta<n^{-1}}}\Big\{(x,t,s)\in \tilde{\Sigma}_f\cap \tilde{\Sigma}_g\,\Big|\,\nu_{\GG_f}^{\boldsymbol u_\eta}(x,t)\ne\pm\nu_{\GG_g}^{\boldsymbol u_\eta}(x,s),\,(\nu_{\GG_f}^{\boldsymbol u_\eta}(x,t))_{n+1}=(\nu_{\GG_g}^{\boldsymbol u_\eta}(x,s))_{n+1}=0\Big\}.
\end{equation}
Then it holds that $$(\abs{\DIFF f}\wedge\abs{\DIFF g})(R)=0.$$
\end{lem}

\begin{proof}
Let us fix a ball $\bar{B}$ in $\XX$, set 
$$
\Omega_f\defeq (C_f\times\RR)\cap( \bar{B}\times\RR)\cap \FF\GG_f,
$$
and define similarly $\Omega_g$. 

For $i\in\NN$, set $\eta_i\defeq 2^{-i}\eta_0$. Here $\eta_0\in(0,n^{-1})\cap\QQ$ satisfies $\eta_0 C_N<1$, where $C_N$ is given in {\Cref {prop:from_delta-splitt_to-biLip}}. 
We claim that for every $i$ there exists a decomposition of the kind
\[
\Omega_f=G_i(f)\cup M_i(f) \cup R_i(f),
\]
and similarly for $g$, for which the following hold.
\begin{itemize}
    \item We have the inequality 
\begin{equation}\label{firstbullet}
	    \HH^h_5 (M_i(f))+|\DIFF\chi_{\GG_f}|(R_i(f))\le C_{K,N}\eta_i{(|\DIFF\chi_{\GG_f}|(\bar B\times\RR)+1)}
\end{equation}
and similarly for $g$, where $C_{K,N}$ is, in particular, independent of $i$.
    \item 
Set $\hat{G}_i(f)\defeq\pi^1 (G_i(f))$ and $\hat{G}_i(g)\defeq\pi^1 (G_i(g))$. Define similarly $\hat{M}_i(f)$, $\hat{M}_i(g)$, $\hat{R}_i(f)$, and $\hat{R}_i(g)$.
Then
\begin{equation}\label{secondbullet}
	(\abs{\DIFF f}\wedge\abs{\DIFF g})(R\cap \hat{G}_i(f)\cap \hat{G}_i(g))=0.
\end{equation}
\end{itemize}

\medskip
We show now how this decomposition allows us to conclude the proof of the lemma. We set 
$$
\hat{G}\defeq\bigcup_{i\in\NN}\hat{G}_{i}(f)\cap \hat{G}_i(g).
$$
As \eqref{secondbullet} implies  that
$$(\abs{\DIFF f}\wedge\abs{\DIFF g})(R\cap \hat G)=0,$$
it suffices to show (recall that $R\subseteq C_f\cap C_g$) $$(\abs{\DIFF f}\wedge\abs{\DIFF g})((C_f\cap C_g\cap\bar{B})\setminus\hat G)=0,$$
as the ball $\bar{B}$ was arbitrary.

Let us go through the proof of the last equality. Notice that for every $i$,
\begin{align*}
	(\abs{\DIFF f}\wedge\abs{\DIFF g})((C_f\cap C_g\cap\bar{B})\setminus\hat{G})\le \abs{\DIFF f}(\hat{M}_i(f)\cup\hat{R}_i(f))+\abs{\DIFF g}(\hat{M}_i(g)\cup\hat{R}_i(g)).
\end{align*}
Therefore, it is enough to show that (as a similar statement will hold for $g$),
$$\lim_{i\to\infty}\abs{\DIFF f}(\hat{M}_i(f)\cup\hat{R}_i(f))=0,$$
so that, recalling Proposition \ref{prop:behaviour_TV_PI} and that $\pi^1|_{\mathcal F\GG_f}$ is injective on $C_f\times\RR$, we can just show that
$$\lim_{i\to\infty}{|\DIFF \chi_{\GG_f}|}\left(\bigcup_{j\ge i}{M}_j(f)\right)+{|\DIFF \chi_{\GG_f}|}({R}_i(f))=0,$$
which follows from \eqref{firstbullet}, since \eqref{firstbullet} again and the definition of $\eta_i$ imply that  $$\HH^h_5\left(\bigcap_{i\in\NN}\bigcup_{j\ge i}M^j(f)\right)=0.$$

For the sake of clarity, we subdivide the rest of the proof into five steps. In Step 1 we construct a candidate decomposition as above in such a way that \eqref{firstbullet} is satisfied. The remaining steps are to prove \eqref{secondbullet} for the decomposition obtained in Step 1. Step 2 and Step 4
are used to obtain technical estimates, whereas Step 3 is the most important and proves a $\sigma$-finiteness property via transverse intersection. With these results in mind, we conclude the proof in Step 5.  In the rest of the proof, we are going to use heavily all the conditions ensured by the membership to $C_f$ and $C_g$ without pointing it out every time. 
In other words, we are morally partitioning $\XX$ in good sets, up to an almost negligible set. These sets are good in the sense that $\tilde\Sigma_f$ and $\tilde\Sigma_g$, restricted to the preimage of these sets with respect to the projection onto $\XX$, are bi-Lipschitz equivalent to $(n+1)$-rectifiable subsets of $\RR^{n+2}$, via the same chart maps. Then, as explained in the introduction, the task is to prove transversality of these two subsets of $\RR^{n+2}$, and this is done via a blow-up argument, taking advantage of the fact that we are using the same chart maps.
\medskip

\textbf{Step 1: Construction of the decomposition.}  
{ Let $\epsilon_i\in (0,n^{-1})\cap (0,\omega_{n}/(2\omega_{n+1}))\cap\QQ$ be given by \Cref{prop:from_delta-splitt_to-biLip} applied to $E=\mathcal{G}_f$, with $\eta_i$ in place of $\eta$}. Using the good collection of splitting maps, consider 
\begin{equation*}
	\boldsymbol u_i=\{u_{i,k}\}_k\defeq\boldsymbol u_{{\epsilon}_i /(n+1)},\ \{D_{i,k}\}_k\defeq\{D_{{\epsilon}_i/(n+1),k}\}_k,\ k_i\defeq k_{{\epsilon_i/(n+1)}},\ A_{i}\defeq A_{{\epsilon}_i/(n+1)},
\end{equation*}
where we recall that $k$ and $A$ have been defined in \Cref{defn:GoodCollection}.

We only consider the case of the function $f$, the construction for $g$ being the same, and we concentrate on a fixed $i$. Therefore, we omit to write the dependence on $f$ for what remains of Step 1.

We refer to {the discussion at the beginning of \Cref{sec:Final}} for the the definition (and the basic properties) of the auxiliary set $(\FF_{n+1}\GG_f)_{r,\epsilon}$.
Let $$r_{i}\in (0,\abs{K}^{-1})$$ be small enough so that, setting
\begin{align*}
	R_i^1\defeq \Omega_f\setminus (\FF_{n+1}\GG_f)_{2 r_i,\epsilon_i},
\end{align*}
it holds that 
$$
|\DIFF\chi_{\GG_f}|(R_i^1)<\eta_i.
$$ 
Let also $c=c_i\in (0,1)$ be small enough so that, setting
\begin{align*}
	R_i^2\defeq \Omega_f\setminus \big\{p\in\mathcal F\GG_f\;\big|\;c<\Theta_n(|\DIFF\chi_{\GG_f}|,p)<c^{-1}\big\},
\end{align*}
it holds that $$|\DIFF\chi_{\GG_f}|(R_i^2)<\eta_i.$$

Take now $p=(x,\bar f(x))\in\Omega_f \setminus R_i^1$, so that $x\in D_{i,k}$ for $k=k_i(x)$, see item v) of \Cref{prop:def_set_C_f}, and we have an associated invertible matrix $A=A_i(x)$, compare with item iii) of \Cref{prop:def_set_C_f}, and the discussion in \Cref{defn:GoodCollection}.
Set $v\defeq(u_{i,k},\pi^2)$ and $z\defeq (A u_{i,k},\pi^2)$. Notice that by the fact that $x\in D_{i,k}$ we have that $u_{i,k}$ is $\varepsilon_i$-splitting on a small ball around $x$. Hence, by tensorization, $v$ is $\varepsilon_i$-splitting on a small ball around $p$. Recall, moreover, that by item iii) of \Cref{prop:def_set_C_f} we have that $z$ is a set of good coordinates at $(x,\bar f(x))$, see \Cref{defn:good_coordinates}. Hence, we have that for some ${\nu}\in\mathbb S^{n}$,
$$
\lim_{r\searrow 0}\dashint_{B_r(p)} |{\nu}^j-\nu_{\GG_f}\,\cdot\,\nabla z^j|\dd{|\DIFF\chi_{\GG_f}|}=0\qquad\text{for every }j=1,\dots,n+1,
$$ so that, for some $\mu\in\RR^{n+1}\setminus\{0\}$,
$$
\lim_{r\searrow 0}\dashint_{B_r(p)} |{\mu}^j-\nu_{\GG_f}\,\cdot\,\nabla v^j|\dd{|\DIFF\chi_{\GG_f}|}=0\qquad\text{for every }j=1,\dots,n+1.
$$
It follows that for some $B\in SO(n+1)$, setting $w= B v$, we have that
$$
\lim_{r\searrow 0}\dashint_{B_r(p)} |\nu_{\GG_f}\,\cdot\,\nabla w^j|\dd{|\DIFF\chi_{\GG_f}|}=0\qquad\text{for every }j=1,\dots,n.
$$
Indeed, it suffices to take $B\in SO(n+1)$ such that $B\mu = (0^n,\|\mu\|_{\mathbb R^{n+1}})$.
The equation above and the membership $p\in\mathcal{F} \GG_f$ imply that
$$
\lim_{r\searrow 0}\frac{r}{\mass{\otimes\HH^1}(B_{2 r}(p))}\int_{B_{2 r}(p)} |\nu_{\GG_f}\,\cdot\,\nabla w^j|\dd{|\DIFF\chi_{\GG_f}|}=0\qquad\text{for every }j=1,\dots,n.
$$

Take then $\tilde{r}=\tilde{r}_{i,p}\in(0,r_i)$ small enough so that $w$ is an $\epsilon_i$-splitting map
on $B_{2 \tilde{r}}(p)$ (this is possible thanks to our choice of $\boldsymbol u_i$, the fact that $v$ is $\varepsilon_i$-splitting on a small ball around $p$, and that $B\in SO(n+1)$\footnote{Notice that the operator norm of $B$ is bounded above by a function of $n$, hence the Lipschitz constant of $w$ might increase by at most such a factor, but this is clearly not a problem.}), moreover
$$
\frac{\tilde{r}}{\mass{\otimes\HH^1}(B_{2 \tilde{r}}(p))}\int_{B_{2 \tilde{r}}(p)} |\nu_{\GG_f}\,\cdot\,\nabla w^j|\dd{|\DIFF\chi_{\GG_f}|}<\epsilon_i\qquad\text{for every }j=1,\dots,n,
$$
and finally, using also that $|\DIFF\chi_{\GG_f}|$ is asymptotically doubling at \(p\), 
$$
|\DIFF\chi_{\GG_f}|(B_{\tilde{r}}(p)\setminus (D_{i,k}\times\RR))<\eta_i |\DIFF\chi_{\GG_f}|(B_{\tilde{r}/5}(p)),
$$
where we recall that for deducing the last information we are using that item v) of \Cref{prop:def_set_C_f}.
{We can also assume that $B_{\tilde r}(x)\subseteq\bar B$, and this will be useful below.}
Note that $p\in(\FF_{n+1}\GG_f)_{2 r_i,\epsilon_i}\subseteq (\FF_{n+1}\GG_f)_{2 \tilde{r},\epsilon_i}$.
{ We can thus apply \Cref{prop:from_delta-splitt_to-biLip} and obtain a set $G=G_{i,p}\subseteq B_{\tilde{r}}(p)$} such that 
$$ 
\HH_5^{h}(B_{\tilde{r}}(p)\setminus G)\le C_N\eta_i\frac{\mass\otimes\mathcal{H}^1(B_{\tilde{r}}(p))}{\tilde{r}},
$$
and $(w^1,\dots,w^n): G\cap(\FF_{n+1}\GG_f)_{2 \tilde{r},\epsilon_i}\rightarrow\RR^n$ is bi-Lipschitz onto its image. Here $C_N$ depends only on $N$. Clearly, also  $v:G\cap(\FF_{n+1}\GG_f)_{2 {\tilde r},\epsilon_i}\rightarrow\RR^{n+1}$ is bi-Lipschitz onto its image, so that the image of $v$ is $n$-rectifiable, due to the fact that $\mathcal{F}_{n+1}\mathcal{G}_f$ is $n$-rectifiable.

To sum up, for $i$ fixed, for every $ p=(x,t)\in\Omega_f\setminus R_i^1$ we have shown that 
$$
v_{i,p}\defeq (u_{i,k_i(x)},\pi^2):  G_{i,p}\cap (\FF_{n+1}\GG_f)_{2 {r}_{i},\epsilon_i}\rightarrow\RR^{n+1},
$$ 
is  bi-Lipschitz onto its image for some set $G_{i,p}\subseteq B_{\tilde{r}_{i,p}}(p)$, that
\begin{equation}\label{tmp111}
	 \HH_5^{h}(B_{\tilde{r}_{i,p}}(p)\setminus G_{i,p})\le C_N\eta_i\frac{\mass\otimes\mathcal{H}^1(B_{\tilde{r}_{i,p}}(p))}{{\tilde{r}_{i,p}}},
\end{equation}
and finally that
\begin{equation}\label{tmp112}
	|\DIFF\chi_{\GG_f}|(B_{\tilde{r}_{i,p}}(p)\setminus (D_{i,k_i(x)}\times\RR))<\eta_i |\DIFF\chi_{\GG_f}|(B_{\tilde{r}_{i,p}/5}(p)).
\end{equation}

We apply Vitali’s covering Lemma to find a sequence of balls $\{B^j_{i}\}_j$ where for every $j$, $B^j_{i}=B_{r_i^j}(p_i^j)=B_{\tilde{r}_{i,p}}(p)$ for some $p=p_i^j\in  \Omega_f\setminus R_i^1$, such that 
$$ \bigcup_{j\in\NN}B_{i}^j\supseteq \Omega_f\setminus R_{1}^i$$
and $\{5^{-1} B^j_{i}\}_j$ are pairwise disjoint; here \(5^{-1}B^j_i\) stands for the ball \(B_{r_i^j/5}(p_i^j)\).
Clearly, to each $B_i^j$ correspond in a natural way sets $G^j_i$, $D^j_i$, and maps $v^j_i:G_i^j\cap (\FF_{n+1}\GG_f)_{2 r_i,\epsilon_i}\rightarrow\RR^{n+1}$. 
We set then
\begin{align*}
M_i&\defeq\Omega_f\cap\bigcup_{j\in\NN}(B_{i}^j\setminus G^j_{i}),\\
R_i^3&\defeq\Omega_f\cap  \bigcup_{j\in\NN}(B_i^j\setminus (D_i^j\times\RR)).
\end{align*} 
Using \eqref{tmp111} for the first chain of inequalities, and \eqref{tmp112} for the second chain of inequalities, we have
\begin{align*}
\HH^h_5(M_i)&\le \sum_{j\in\NN}\HH^h_5(B_i^j\setminus G_i^j)\le C_N\eta_i\sum_{j\in\NN}\frac{ \mass\otimes\mathcal{H}^1(B_i^j)}{r_i^j} \le C_{K,N}\eta_i \sum_{j\in\NN} \frac{ \mass\otimes\mathcal{H}^1(5^{-1}B_i^j)}{r_i^j/5}\\&\le C_{K,N}\eta_i \sum_{j\in\NN} |\DIFF\chi_{\GG_f}|(5^{-1}B_i^j) \le C_{K,N}\eta_i{|\DIFF\chi_{\GG_f}|(\bar B\times\RR)}
\end{align*}
We stress that in the fourth inequality above we are using that $p_i^j\in (\mathcal{F}_{n+1}\mathcal{G}_f)_{2r_i,\varepsilon_i}$, and 
\begin{align*}
|\DIFF\chi_{\GG_f}|(R_i^3)\le \sum_{j\in\NN} |\DIFF\chi_{\GG_f}|(B_i^j\setminus (D_i^j\times\RR))\le  \eta_i\sum_{j\in\NN}|\DIFF\chi_{\GG_f}|(5^{-1}B_i^j)\le \eta_i{|\DIFF\chi_{\GG_f}(\bar B\times\RR)|}.
\end{align*}
Now set 
$$
S_i^j\coloneqq v^j_i\big((\Omega_f \cap G_i^j \cap (\FF_{n+1}\GG_f)_{2 r_i,\epsilon_i})\setminus( R_i^1 \cup R_i^2\cup R_i^3)\big)\subseteq\RR^{n+1},
$$
and recall that $S_i^j$ is $n$-rectifiable. For every $j\in\NN$, there exists a countable family $\{S_i^{j,\ell}\}_{\ell\in\NN}$ of $C^1$-hypersurfaces in $\RR^{n+1}$,  such that  $$\HH^n\bigg(S_i^j\setminus\bigcup_{\ell\in\NN} S_i^{j,\ell}\bigg)=0.$$
Define
$$
 \hat{S}_i^{j,\ell}\defeq\bigg\{y\in S_i^{j,\ell}\cap S_i^j\;\bigg|\;\lim_{r\searrow 0}\frac{\HH^n(B_r(y)\cap S_i^{j,\ell}\cap S_i^j)}{\omega_n r^n}=1 \bigg\},
$$
and 
$$
R^4_i\defeq \bigcup_{j\in\NN}\bigcap_{\ell\in\NN}\big(S_i^j\setminus (v_i^j)^{-1}(\hat{S}_i^{j,\ell}) \big)\subseteq\Omega_f,
$$ 
and notice that  $\HH^n(R^4_i)=0$, so that $|\DIFF\chi_{\GG_f}|(R_i^4)=0$.
We set also 
$$
Q_i^{j,\ell}\defeq (v_i^j)^{-1}(\hat{S}_i^{j,\ell})\subseteq\Omega_f,
$$
and notice that
\begin{equation}\label{compcharts}
	\text{if }v_i^j=(u_{i,k},\pi^2),\qquad\text{then }Q_i^{j,\ell}\subseteq D_{i,k}\times\RR\text{ for every }\ell\in\NN.
\end{equation}
Now define
\begin{align*}
	R_i^5&\defeq \bigcup_{j,\ell\in\NN}\bigg(Q_i^{j,\ell}\setminus\bigg\{p\in Q_i^{j,\ell}\;\bigg|\;\lim_{r\searrow 0}\frac{|\DIFF\chi_{\GG_f}|(B_r(p)\cap Q_i^{j,\ell})}{|\DIFF\chi_{\GG_f}|(B_r(p))}=1\bigg\}\bigg).
\end{align*}
We then set
$$
R_i\defeq R^{1}_i\cup R^{2}_i\cup R_i^3\cup R_i^4\cup R_i^5,
$$ 
and finally 
$$
G_i\defeq \Omega_f\setminus(M_i \cup R_i)\subseteq\bigcup_{j,\ell\in\NN}Q_i^{j,\ell}.
$$
It is immediate to check that the sets we constructed verify \eqref{firstbullet}.
The rest of the proof shows that they also verify \eqref{secondbullet}.

\medskip
\textbf{Step 2: Almost one-sided Kuratowski convergence.} For any $i$, let 
$$
p\in \Omega_f\setminus R_i^1(f),
$$ 
and let $\rho_k\searrow 0$ be such that 
$$
(\XX\times\RR,\rho_k^{-1}\dist_{\XX\times\RR},(\mass\otimes\mathcal{H}^1)_p^{\rho_k},p,\GG_f)\rightarrow (\RR^{n+1},\dist_e,\underline{\mathcal L}^{n+1},0,H),
$$
where $H\subseteq\RR^{n+1}$ is a halfspace. Fix also $\rho>0$. Assume the convergence is realized in a proper metric space $(\ZZ,\dist_\ZZ)$.
We show that for every $\epsilon>0$ there exists $k_0\in\NN$ such that 
$$
 B_\rho^\ZZ(p^k)\cap (\Omega_f\setminus R_i^1(f))^k\subseteq B^\ZZ_{\epsilon}(\partial H)\qquad\text{if $k\ge k_0$.}
$$
Here the superscript $k$ denotes the isometric image in $\ZZ$ through the embedding of the $\rho_k$-rescaled space. 

We argue by contradiction. Up to taking a not relabelled subsequence, by the contradiction assumption there exist $\{q^k\}_k$ such that for every $k$, 
$$
q^k\in  \big(B_\rho^\ZZ(p^k)\cap (\Omega_f\setminus R_i^1(f))^k\big)\setminus B_{\epsilon}^\ZZ(\partial H).
$$ 
Up to a not relabelled subsequence, $q^k\rightarrow q\in \ZZ$,  with $\dist_\ZZ(q,\partial H)\ge \epsilon/2$. It is easy to see that $q\in\RR^{n+1}$.
By weak convergence of measures,
$$
\lim_{k\to\infty}\frac{\rho_k|\DIFF\chi_{\GG_f}|(B_{\epsilon\rho_k/2}(q^k))}{C^{\rho_k}_p}=0.
$$
On the other hand, recalling that $\{q^k\}_k\subseteq (\FF_{n+1}\GG_f)_{2 r_i,\epsilon_i}$ and using again the weak convergence of measures,
\begin{align*}
	\liminf_{k\to\infty} \frac{\rho_k|\DIFF\chi_{\GG_f}|(B_{\epsilon\rho_k/2}(q^k))}{C^{\rho_k}_p}&=\liminf_{k\to\infty} \frac{\rho_k|\DIFF\chi_{\GG_f}|(B_{\epsilon\rho_k/2}(q^k))}{\mass(B_{\epsilon\rho_k/2}(q^k))}\frac{\mass(B_{\epsilon\rho_k/2}(q^k))}{C_p^{\rho_k}}\\
&\ge \frac{\omega_n}{2 \omega_{n+1}}\underline{\mathcal L}^{n+1}(B_{\epsilon/2}(q))>0,
\end{align*}
which is a contradiction.
\medskip

\textbf{Step 3: Proof of the \(\sigma\)-finiteness claim.} We use the same notation as in Step 1.
We claim that for every $i$, $$\HH^n\mres \big\{(x,t,s)\in \tilde R\;\big|\;x\in\hat G_i(f)\cap\hat{G}_i(g)\big\}$$ is $\sigma$-finite.
To show this, it is enough to prove that for every $i,j,k,\ell,m,\xi\in\NN$,
$$
\HH^n\mres \tilde{T}_{i,j,k,\ell,m,\xi}
$$
is $\sigma$-finite, where we set
\begin{align*}
\tilde{T}_{i,j,k,\ell,m,\xi}\defeq \Big\{(x,t,s)\in \tilde{R}\;\big|\;x\in  \hat{G}_i(f)\cap \hat{G}_{i}(g)\cap D_{i,k},\,(x,t)\in Q_i^{j,m}(f),\,(x,s)\in Q_i^{\ell,\xi}(g) \Big\}.
\end{align*}
Fix then $i,j,k,\ell,m,\xi\in\NN$ and set for simplicity $\tilde{T}=\tilde{T}_{i,j,k,\ell,m,\xi}$. Now define
$$
v\defeq(u_{i,k},\pi^2,\pi^3):(Q_i^{j,m}(f)\times\RR)\cup \tau(Q_i^{\ell,\xi}(g)\times\RR)\rightarrow\RR^{n+2}.
$$ By the construction in Step 1, 
\begin{equation}\label{mapsv}
	v|_{Q_i^{j,m}(f)\times\RR}\qquad\text{and}\qquad v|_{\tau(Q_i^{\ell,\xi}(g)\times\RR)},
\end{equation}
are bi-Lipschitz onto their image. Therefore, as $\tilde{T}\subseteq(Q_i^{j,m}(f)\times\RR)\cap \tau(Q_i^{\ell,\xi}(g)\times\RR)$, it is enough to show that  $$\HH^n\mres v(\tilde{T})$$ is $\sigma$-finite. Here a central point is that $\tilde{T}\subseteq D_{i,k}\times\RR\times\RR$ so that, by the construction in Step 1, the map $v$ as above will be suitable both for the part concerning $f$ and the part concerning $g$ (see \eqref{compcharts}). 
Now notice that $$v(\tilde{T})\subseteq( \hat{S}_i^{j,m}(f)\times\RR)\cap \tau(\hat{S}_i^{\ell,\xi}(g)\times\RR),$$
so that, by a standard result of in Geometric Measure Theory on Euclidean spaces, we can simply show that at every $p=(x,t,s)\in\tilde{T}$ it holds that $\hat{S}_i^{j,m}(f)\times\RR$ and $\tau(\hat{S}_i^{\ell,\xi}(g)\times\RR)$ intersect transversally at $v(p)$, or, equivalently, that $\hat{S}_i^{j,m}(f)\times\RR$ and $\tau(\hat{S}_i^{\ell,\xi}(g)\times\RR)$ have  different tangent spaces at $v(p)$. We can, and will, assume that $v(p)=0$. 

By our assumptions, compare with item iii) and item iv) of \Cref{prop:def_set_C_f}, we know that there exists a sequence $\rho_k\searrow 0$ and a proper metric space $(\ZZ,\dist_\ZZ)$ such that $(\ZZ\times\RR\times\RR,\dist_{\ZZ\times\RR\times\RR})$ realizes both the convergence
\begin{equation}\label{conv1}
\begin{split}
		(\XX\times\RR\times\RR,\rho_k^{-1}\dist_{\XX\times\RR\times\RR},&(\mass\otimes\HH^1\otimes\HH^1)_{p}^{\rho_k},p,\GG_f\times\RR)\\&\rightarrow (\RR^n\times\RR\times\RR,\dist_e,\underline{\mathcal L}^{n+2},0,H\times\RR\times\RR)
\end{split}
\end{equation}
and the convergence
\begin{equation}\label{conv2}
\begin{split}
		(\XX\times\RR\times\RR,\rho_k^{-1}\dist_{\XX\times\RR\times\RR},&(\mass\otimes\HH^1\otimes\HH^1)_{p}^{\rho_k},p,\tau(\GG_g\times\RR))\\&\rightarrow (\RR^n\times\RR\times\RR,\dist_e,\underline{\mathcal L}^{n+2},0,H'\times\RR\times\RR),
\end{split}
\end{equation}
 where $H$ and $H'$ are halfspaces in \(\RR^n\). Notice that this can be done since the $(n+1)$-coordinate of the $\nu$'s are zero, see the definition of $\tilde R$.  We have endowed $\RR^{n}\times\RR\times\RR$ with the coordinates given by the (locally uniform) limits of appropriate
 rescalings of the components of $z$, where $$z\defeq (A_{i}(x)u_{i,k},\pi^2,\pi^3)\colon B_\rho(p)\rightarrow\RR^{n+2},$$
 for some $\rho>0$ (see Remark \ref{convdelta}). To do so, we needed to take a not relabelled subsequence of $\{\rho_k\}_k$, but this will make no difference.
Hence, recalling also the definition of $\tilde{R}$, it follows that $H\ne H'$. 

	Fix $D\ge 5$ greater than the bi-Lipschitz constants of the maps in \eqref{mapsv} and such that 
	\begin{equation}\label{Dbound}
		|(A_i(x),\pi^1,\pi^2) c|\le (D-4)|c|\qquad\text{for every $c\in\RR^{n+2}$}.
	\end{equation}
 Let $\delta\in(0,D^{-1})$ be sufficiently small so that we find $a\in (\partial H\times\RR\times\RR)\cap B_{1}(0)\subseteq\RR^{n+2}$ such that $B_{ D\delta}(a)\cap (\partial H'\times\RR\times\RR)=\emptyset$.

  	As a consequence of the density assumption made by removing $R_i^5$, we can find a sequence $\{a^k\}_k\subseteq\XX\times\RR\times\RR$ with
  	$$
  	a^k\in( Q_i^{j,m}(f)\times\RR)\cap B_{\rho_k}(p)\qquad\text{for every }k\in\NN,
  	$$ 
  	and $a^k\rightarrow a$ in $\ZZ\times\RR\times\RR$, where here and below the superscript $k$ denotes the isometric image in $\ZZ\times\RR\times\RR$ through the embedding of the $\rho_k$-rescaled space.

By weak convergence of measures,
\begin{align*}
\liminf_{k\to\infty}\frac{\rho_k |\DIFF\chi_{\GG_f\times\RR}|(B_{D^{-1}\delta\rho_k}(a^k))}{C_p^{\rho_k}}&>0,\\
\limsup_{k\to\infty}\frac{\rho_k |\DIFF\chi_{\tau(\GG_g\times\RR)}|(B_{ D\delta \rho_k}(a^k))}{C_p^{\rho_k}}&=0.
\end{align*}
Recalling again the density assumption made by removing $R_i^5$ together with the bounds on $\Theta_n(|\DIFF\chi_{\GG_f}|,\,\cdot\,)$ by removing $R_i^2$,
and finally the weak convergence of measures, this reads as
\begin{align}
	\liminf_{k\to\infty}\rho_k^{-n-1} \HH^{n+1}\big(B_{D^{-1}\delta\rho_k}(a^k)\cap (Q_i^{j,m}\times \RR)\big)&>0,\label{aleq1}\\
	\limsup_{k\to\infty}\rho_k^{-n-1}\HH^{n+1}\big(B_{ D\delta \rho_k}(a^k)\cap\tau (Q_i^{\ell,\xi}\times \RR)\big)&=0\label{aleq2}.
\end{align}
It is easy to verify by contradiction that \eqref{aleq1} implies, by our choice of $D$, that
\begin{equation}\label{cont1}
	\liminf_{k\to\infty}\rho_k^{-n-1}\HH^{n+1}\big(B_{\delta\rho_k}(v(a^k))\cap (\hat{S}^{j,m}_i(f)\times\RR)\cap B_{2 D\rho_k}(0)\big)>0.
\end{equation}
Now we show that  
\begin{equation}\label{cont2}
	\liminf_{k\to\infty}\rho_k^{-n-1}\HH^{n+1}\big(B_{\delta\rho_k}(v(a^k))\cap\tau (\hat{S}^{\ell,\xi}_i(g)\times\RR)\cap B_{2 D\rho_k}(0)\big)=0.
\end{equation}
By Step 2, we get that for $\epsilon\in (0,\delta)$,  there exists $k_0$ such that if $k\ge k_0$, then for every $b\in (B_{2  D^{2} \rho_k}(p)\setminus B_{D \delta\rho_k}(a^k)\cap\tau(Q_i^{\ell,\xi}\times\RR))^k$ there exists $b'\in\partial H'\times\RR\times\RR$ such that $$\dist_{\ZZ\times\RR\times\RR}(b,b')<\epsilon.$$
Up to increasing $k_0$, we may assume that for every $k\ge k_0$,$$\dist_{\ZZ\times\RR\times\RR}(a,a^k)<\epsilon.$$
Notice that if $b$ is as above, then
$$|b'-a|\ge D\delta-2\epsilon$$
and, by local uniform convergence, up to enlarging $k_0$ and provided $\epsilon>0$ is small enough,
$$
|\rho_k^{-1}z(b)-\rho_k^{-1}z(a^k)|\ge | b'-a|- 2 \delta
$$
so that 
$$|z(b)-z(a^k)|\ge ((D-2)\delta-2\epsilon)\rho_k\ge(D-4)\delta\rho_k,$$
which implies, recalling \eqref{Dbound},
$$ |v(b)-v(a^k)|\ge \delta\rho_k.$$
Notice that the above inequality does \emph{not} follow from the fact that the maps in \eqref{mapsv} are $D$-bi-Lipschitz, but implies that \eqref{cont2} follows from \eqref{aleq2}, by the choice of $D$.

We can now conclude the proof of Step 3, as by \eqref{cont1} and \eqref{cont2} it follows easily that $\hat{S}_i^{j,m}(f)\times\RR$ and $\tau(\hat{S}_i^{\ell,\xi}(g)\times\RR)$ have different tangent spaces at $0$.
\medskip

\textbf{Step 4: A technical estimate.} For some $i\in\NN$, let us assume $\tilde{R}'$ is such that 
$$
\tilde{R}'\subseteq\tilde{R}\cap (\hat{G}_i(f)\times\RR\times\RR)\cap (\hat{G}_i(g)\times\RR\times\RR),
$$
and that $\tilde{R}'$ has finite $\HH^n$-measure. Let $p\in\tilde{R}'$ be fixed. We claim that
$$\lim_{r\searrow 0}\frac{\HH_5^n\big(\pi^{1,2}(\tilde{R}'\cap B_r(p))\big)}{r^n}=0.$$

Let us prove the claim. Take a sequence $\rho_k\searrow 0$.
We recall, that, with the same notation as above, up to a not relabelled subsequence, \eqref{conv1} and \eqref{conv2} hold.
Let 
$$
I:=I((\partial H\cap\partial H')\times\RR\times\RR)
$$ 
be a neighbourhood (in $\ZZ\times\RR\times\RR$) of $((\partial H\cap\partial H')\times\RR\times\RR)\cap B_2(0)$ that satisfies $$\HH^n_{5}(\pi^{1,2}(I))<\epsilon.$$ As a consequence of Step 2, there exists $k_0\in\NN$
such that $$B_1^{\ZZ\times\RR\times\RR}(p^k)\cap \tilde{R}' \subseteq I \qquad\text{for every $k\ge k_0$},$$
from which, taking the projection $\pi^{1,2}$, the claim follows.

\medskip 
\textbf{Step 5: Conclusion.} Let us finally prove \eqref{secondbullet}. By Step 3, it is enough to show that 
 $$(\abs{\DIFF f}\wedge\abs{\DIFF g})(\pi^1(\tilde{R}'))=0,$$
where $\tilde{R}'$ is as in Step 4. Fix $\epsilon>0$. For every $j\in\NN,\ j\ge 1$ we consider the sets $$\tilde{R}'_j\defeq\bigg\{p\in\tilde{R}'\;\bigg|\; \frac{\HH_5^n\big(\pi^{1,2}(\tilde{R}'\cap B_r(p))\big)}{r^n}<\epsilon\,\text{ for every $r\in (0,j^{-1})$}\bigg\}$$
 and $$ \tilde{R}''_j\coloneqq\tilde{R}'_j\setminus\bigcup_{i<j}\tilde{R}_i'.$$ 
Notice that, by Step 4, 
$$
\tilde{R}'=\bigcup_{j\geq 1}\tilde{R}''_j,
$$
and, by construction, this union is disjoint.
For every $j\geq 1$, we take a countable family of balls $\{B_{r_i^j}(p_i^j)\}_i$ such that, for every $i\in\NN$ it holds that $r_i^j<j^{-1}$ and $p_i^j\in\tilde{R}''_j$, as well as
\begin{equation}\label{hatr}
	\tilde{R}''_j\subseteq\bigcup_{i\in\NN}B_{r_i^j}(p_i^j) \qquad\text{and}\qquad
\sum_{i\in\NN}(r_i^j)^n\le2^n\HH^{n}(\tilde{R}''_j)+2^{-j}.
\end{equation}
We can compute, recalling the definition of $\tilde{R}''_j$ and \eqref{hatr},
\begin{align*}
	&\HH_5^n(\pi^{1,2}(\tilde{R}''_j))
	 \le\HH_5^n\Big(\pi^{1,2}\Big(\tilde{R}''\cap\bigcup_{i\in\NN}B_{r_i^j}(p_i^j)\Big)\Big)\le \sum_{i\in\NN} \epsilon (r_i^j)^n\le \epsilon (2^n\HH^n(\tilde{R}''_j)+2^{-j}).
\end{align*}
Therefore, it holds that
$$ 
\HH_5^n(\pi^{1,2}(\tilde{R}'))\le \epsilon (2^n\HH^n({\tilde{R'}})+1),
$$
and, being $\epsilon>0$ arbitrary, $|{\DIFF \chi_{\GG_f}}|(\pi^{1,2}(\tilde{R'}))=0$,
whence the conclusion follows thanks to Proposition \ref{prop:behaviour_TV_PI}.
\end{proof}

\begin{lem}\label{rank1lem0}
	Let $(\XX,\dist,\mass)$ be an $\RCD(K,N)$ space of essential dimension $n$ and let $f,g\in\BV(\XX)$. Choose two $\capa$-vector fields representatives for $\nu_f$ and $\nu_g$. Then $$\nu_f=\pm\nu_g\qquad(\abs{\DIFF f}\wedge\abs{\DIFF g})\text{-a.e.\ on $C_f\cap C_g$}.$$
\end{lem}
\begin{proof}
From \Cref{coincide} and \Cref{mainlemma} together with \Cref{thmcantor} we have that for $(|\DIFF f|\wedge|\DIFF g|)$-a.e.\ $x\in C_f\cap C_g$ there exists $\eta=\eta(x)\in (0,n^{-1})\cap\QQ$ such that $$\nu^{\boldsymbol u_\eta}_f(x)=\pm\nu^{\boldsymbol u_\eta}_g(x).$$

It remains to show that if for some $\eta\in (0,n^{-1})\cap\QQ$ it holds that $\nu^{\boldsymbol u_\eta}_f=\pm\nu^{\boldsymbol u_\eta}_g$ $\capa$-a.e.\ on a Borel set $A$, then $\nu_f=\pm\nu_g$ $\capa$-a.e.\ on $A$.
This follows since the gradients of the functions in ${\boldsymbol u}_{\eta,k}$ are a generating subspace of $L^0_\capa (T\XX)$ on $D_{\eta,k}$, for the $L^0_\capa (T\XX)$ module has local dimension at most $n$. Indeed, if $h_{1},\cdots,h_{n+1}\in\TestF(\XX)$ then ${\rm{det}}(\nabla h_i\,\cdot\,\nabla h_j)_{i,j}=0\ \mass$-a.e.\ hence $\capa$-a.e., so that is now easy to bound the local dimension of $L^0_\capa (T\XX)$.
\end{proof}

The following lemma is extracted from \cite[Proposition 3.30]{BGBV}.
\begin{lem}\label{rank1lem1}
	Let $(\XX,\dist,\mass)$ be an $\RCD(K,N)$ space of essential dimension $n$ and let $f,g\in\BV(\XX)$. Choose two $\capa$-vector fields representatives for $\nu_f$ and $\nu_g$. Then $$\nu_f=\pm\nu_g\qquad(\abs{\DIFF f}\wedge\abs{\DIFF g})\text{-a.e.\ on $J_f\cap J_g$}.$$
\end{lem}

\begin{proof}[Proof of \Cref{thm:RankOne}]
We first notice that for every $i=1,\dots,k$,$$(\nu_F)_i=\dv{|\DIFF F_i|}{|\DIFF F|}\nu_{F_i}\qquad|\DIFF F|\text{-a.e.}$$
The conclusion on the jump part is given by Lemma \ref{rank1lem1} applied to every pair of components of $F$ together with the well known fact that for every $i=1,\dots,k$, $|\DIFF F_i|(J_F\setminus J_{F_i})=0$. On the Cantor part, the result follows from Lemma \ref{rank1lem0} applied to every pair of components of $F$. 
\end{proof}

\appendix
\section{Rectifiability of the reduced boundary}\label{sec:Appendix}

In this appendix, we give an alternative proof of the known fact that reduced boundaries of sets of finite perimeter in finite dimensional $\RCD$ spaces are rectifiable. Roughly speaking, this is a consequence of the rectifiability result of \cite{BateMM} and the uniqueness of tangents to sets of finite perimeter proved in \cite{bru2019rectifiability}, once one takes into account the regularity result \Cref{thm:const_dim_cod1}.

Let us recall part of the statement of \cite[Theorem 1.2]{BateMM}.
\begin{thm}\label{thm:BATE}
	Let $(\XX,\dist)$ be a complete metric space, $k\in\mathbb N$, and $S\subseteq\XX$ such that $\mathcal{H}^k(S)<\infty$. Hence the following are equivalent:
	\begin{enumerate}
		\item $S$ is $k$-rectifiable,
		\item for $\mathcal{H}^k$-almost every $x\in S$ we have $\underline\Theta_{k}(S,x)>0$, and the existence of a $k$-dimensional Banach space $(\mathbb R^k,\|\cdot\|_k)$ such that 
		\begin{equation}\label{eqn:YA}
			\mathrm{Tan}_x(\XX,\dist,\mathcal{H}^k\llcorner S)=\{(\mathbb R^k,\|\cdot\|_x,\mathcal{H}^k,0)\}.
		\end{equation}
	\end{enumerate}
\end{thm}

Let us fix $(\XX,\dist,\mass)$ an $\mathrm{RCD}(K,N)$ space of essential dimension $n$. Let $E\subseteq\XX$ be a set of locally finite perimeter. Now by \Cref{thm:const_dim_cod1} and the first part of the argument of \Cref{thm:repr_per} we have that:
\begin{enumerate}
	\item $|\DIFF \chi_E|(\XX\setminus \mathcal{R}_n^*)=0$, and hence $|\DIFF \chi_E|$ is concentrated on $\mathcal{F}E$;
	\item $\mathcal{H}^{n-1}\llcorner \mathcal{F}E$ is a $\sigma$-finite Borel measure that is mutually absolutely continuous with respect to $|\DIFF \chi_E|$. Notice that for the precise computation of the density of $|\DIFF \chi_E|$ with respect to $\mathcal{H}^{n-1}\llcorner \mathcal{F}E$ in \Cref{thm:repr_per} we needed the rectifiability of $\mathcal{F}E$, that we will not use in the following argument.
	
	Hence let us call $f\in L^1_{\mathrm{loc}}(|\DIFF \chi_E|)$ the function such that $\mathcal{H}^{n-1}\llcorner\mathcal{F}E=f|\DIFF \chi_E|$, and let $\mathcal{D}\subseteq \mathcal{F}E$ be the set of the Lebesgue points of $f$ with respect to the asymptotically doubling measure $|\DIFF \chi_E|$ that are also differentiability points of $\mathcal{H}^{n-1}\llcorner\mathcal{F}E$ with respect to $|\DIFF \chi_E|$, i.e., for every $x\in \mathcal{D}$,
	\begin{equation}\label{eqn:InAverage}
		\lim_{r\to 0}\fint_{B_r(x)}\left|f-f(x)\right|\dd|\DIFF\chi_E|=0,
	\end{equation}
	and 
	\begin{equation}\label{eqn:fdensita}
		f(x)=\lim_{r\to 0}\frac{\mathcal{H}^{n-1}\llcorner\mathcal{F}E(B_r(x))}{|\DIFF\chi_E|(B_r(x))}.
	\end{equation}
	Notice that $|\DIFF \chi_E|(X\setminus \mathcal{D})=\mathcal{H}^{n-1}(\mathcal{F}E\setminus\mathcal{D})=0$, due to Lebesgue Differentiation Theorem \cite[page 77]{HKST15}, and Lebesgue--Radon--Nikod\'{y}m Theorem \cite[page 81 and Remark 3.4.29]{HKST15}. Notice, moreover, that since $|\DIFF\chi_E|$ is mutually absolutely continuous with respect to $\mathcal{H}^{n-1}\llcorner\mathcal{F}E$, hence $f(x)>0$ for $|\DIFF\chi_E|$-almost every $x\in \XX$, or equivalently for $\mathcal{H}^{n-1}\llcorner\mathcal{F}E$-almost every $x\in \XX$.
\end{enumerate}
Let us now prove that $\mathcal{F}E$ is $(n-1)$-rectifiable by exploiting \Cref{thm:BATE}. Let us verify item (2) in there. By the third line in \eqref{usefullimits}, together with the fact that $x\in\mathcal{R}_n^*$, and \eqref{eqn:fdensita}, we get that $\underline\Theta_{n-1}(\mathcal{F}E,x)>0$ for every $x\in \mathcal {D}$, and hence for $\mathcal{H}^{n-1}$-almost every $x\in\mathcal{F}E$. Let us now verify the second part of item (2). Let us fix $x\in\mathcal{D}$, and let us take an arbitrary sequence $r_i\to 0$. We have that up to subsequences
\[
\XX_i:=(\XX, r_i^{-1}\dist,\mass_x^{r_i},x,E)\to (\mathbb R,\dist_e,\underline{\mathcal{L}}^n,0,\{x_n>0\}),
\]
and in a realisation of the previous convergence, we have that $|\DIFF\chi_E|_{\XX_i}$ weakly converge to $|\DIFF\chi_{\{x_n>0\}}|$. For the sake of clarity, we denoted with $|\DIFF \chi_E|_{\XX_i}$ the perimeter measure of $E$ in the rescaled space $\XX_i$. Notice that $|\DIFF \chi_E|_{\XX_i}=r_i/C_x^{r_i}|\DIFF\chi_E|$, where $|\DIFF\chi_E|$ is the perimeter measure on $\XX$. Let $g\in C_{\mathrm{bs}}(\ZZ)$, where $\ZZ$ is a realisation of the previous convergence. Hence we have
\[\begin{split}
\int_{\XX_i} g\,\diff\frac{r_i\mathcal{H}^{n-1}\llcorner FE}{C_x^{r_i}} &= \int_{\XX_i} g f\,\diff|\DIFF\chi_E|_{\XX_i}\\
& = \int_{\XX_i} g(y) f(x)\,\diff|\DIFF\chi_E|_{\XX_i}(y) + \int_{\XX_i} g(y) (f(y)-f(x))\,\diff|\DIFF\chi_E|_{\XX_i}(y),
\end{split}\]
and hence, by using \eqref{eqn:InAverage} and the fact that $|\DIFF\chi_E|(B_{r_i}(x))\sim \frac{(n+1)\omega_{n-1}}{\omega_n}\frac{C_x^{r_i}}{r_i}$ as a consequence of the second and third line of \eqref{usefullimits}, we conclude that\footnote{Notice that in the following equation we are considering $\mathcal{H}^{n-1}\llcorner\mathcal{F}E$ in the original space $\XX$ and not in the rescaled space}
\begin{equation}\label{eqn:RESCALED}
	\frac{r_i\mathcal{H}^{n-1}\llcorner\mathcal{F}E}{C_x^{r_i}}\rightharpoonup f(x)|\DIFF \chi_{\{x_n>0\}}|,
\end{equation}
in the realisation $\ZZ$. This immediately implies that 
\[
\frac{\mathcal{H}^{n-1}\llcorner\mathcal{F}E}{\mathcal{H}^{n-1}\llcorner\mathcal{F}E(B_{r_i}(x))}\rightharpoonup \mathcal{H}^{n-1}\llcorner\{x_n=0\},
\]
because $\mathcal{H}^{n-1}\llcorner\{x_n=0\}$ is the surface measure on $\{x_n=0\}$ that gives measure one to the unit ball. 

Hence we have shown that for every $x\in\mathcal{D}$ and every sequence $r_i\to 0$, there is a realisation $\ZZ$ in which one has the convergence
\[
\left(\XX,\frac{\dist}{r_i},\frac{\mathcal{H}^{n-1}\llcorner\mathcal{F}E}{\mathcal{H}^{n-1}\llcorner\mathcal{F}E(B_{r_i}(x))},x\right)\to (\mathbb R^{n-1},\dist_e,\mathcal{H}^{n-1},0),
\]
which is exactly what one needed to show in order to verify \eqref{eqn:YA}  (recall \cite[Proposition 2.13]{BateMM}). Hence the application of \Cref{thm:BATE} gives the $(n-1)$-rectifiability of $\mathcal{F}E$.

\end{document}